\documentclass[a4paper,10pt]{article}
\usepackage[utf8x]{inputenc}
\usepackage{tracefnt,amsmath,tabu,array}
\usepackage{amssymb,graphicx,setspace,amsfonts,amsbsy}
\usepackage{pifont,latexsym,ifthen,amsthm,rotating,calc,textcase,booktabs}

\newtheorem{theorem}{Theorem}[section]
\newtheorem{lemma}[theorem]{Lemma}

\newtheorem{proposition}[theorem]{Proposition}
\newtheorem{definition}[theorem]{Definition}

\newtheorem{remark}[theorem]{Remark}
\newcommand{\filledbox}{\leavevmode
  \hbox to.77778em{%
  \hfil\vbox to.675em{\hrule width.6em height.6em}\hfil}}
\newcommand{\EE}{\mathcal E}
\newcommand{\MM}{\mathcal M}
\newcommand{\Rm}{{\mathbb R}}
\newcommand{\Hm}{{\mathbb H}}
\newcommand{\Db}{{\mathbf D}}
\newcommand{\eps}{\varepsilon}

\begin{document}
\tabulinesep=1.0mm
\title{A Semi-linear Shifted Wave Equation on the Hyperbolic Spaces with Application on a Quintic Wave Equation on $\Rm^2$}


\author{Ruipeng Shen\\
Department of Mathematics and Statistics\\
McMaster University\\
Hamilton, ON, Canada
\and Gigliola Staffilani\footnote{ The second author is funded in part by NSF
DMS 1068815.}\\
Department of Mathematics\\
Massachusetts Institute of Technology\\
Cambridge, MA, USA}





\maketitle
\begin{abstract}
  In this paper we consider a semi-linear, defocusing, shifted wave equation on the hyperbolic space
  \[
   \partial_t^2 u - (\Delta_{\Hm^n} + \rho^2) u = - |u|^{p-1} u, \quad (x,t)\in \Hm^n \times \Rm;
  \]
  and introduce a Morawetz-type inequality
  \[
   \int_{-T_-}^{T_+} \int_{\Hm^n} |u|^{p+1} d\mu dt < C E,
  \]
  where $E$ is the energy. Combining this inequality with a well-posedness theory, we can establish a scattering result for solutions with initial data in $H^{1/2,1/2} \times H^{1/2,-1/2}(\Hm^n)$ if $2 \leq n \leq 6$ and $1<p<p_c = 1+ 4/(n-2)$. As another application we show that a solution to the quintic wave equation $\partial_t^2 u - \Delta u = - |u|^4 u$ on $\Rm^2$ scatters if its initial data are radial and satisfy the conditions
  \[
   |\nabla u_0 (x)|, |u_1 (x)| \leq A(|x|+1)^{-3/2-\eps};\quad |u_0 (x)| \leq A(|x|)^{-1/2-\eps};\quad \eps >0.
  \]
\end{abstract}

\section{Introduction}

\paragraph{Wave equation on Euclidean spaces:} The question of well-posedness and scattering of solutions to the  non-linear wave equation
\begin{equation} \label{waveequRn}
 \left\{\begin{array}{l} \partial_t^2 u - \Delta u = F(u),\quad (x,t)\in \Rm^n \times \Rm,\\
 v |_{t=0} = u_0, \\
\partial_t u |_{t=0} = u_1\end{array}\right.
\end{equation}
in the Euclidean spaces $\Rm^n$ has been extensively studied, especially in three or higher dimensional spaces. In this work the nonlinearity $F(u)$ above is defined for $p>1$ as
\[
 F(u) = \zeta |u|^{p-1} u,
\]
where $\zeta = \pm 1$. If $\zeta = -1$, then the equation is called defocusing, otherwise focusing. Suitable solutions to this equation satisfy an energy conservation law:
\[
 E(u,\partial_t u) = \int_{\Rm^n} \left(\frac{1}{2}|\nabla u|^2 + \frac{1}{2} |\partial_t u|^2 -\zeta \frac{1}{p+1} |u|^{p+1} \right) dx = E(u_0,u_1).
\]
This equation also has a natural scaling property. Namely if $u(x,t)$ is a solution of this equation with initial data $(u_0,u_1)$, then for any $\lambda>0$, the function
\begin{equation} \label{dilation1}
 \frac{1}{\lambda^{\frac{2}{p-1}}} u \left(\frac{x}{\lambda}, \frac{t}{\lambda}\right)
\end{equation}
is another solution of the same equation with the initial data
\[
 (u_{0,\lambda}, u_{1,\lambda}) := \left(\frac{1}{\lambda^{\frac{2}{p-1}}} u_0 \left(\frac{x}{\lambda}\right),
\frac{1}{\lambda^{ \frac{2}{p-1}+ 1}} u_1 \left(\frac{x}{\lambda}\right)\right).
\]
One can check that $(u_0, u_1)$ and $(u_{0,\lambda}, u_{1, \lambda})$ share the same $\dot{H}^{s_p} \times \dot{H}^{s_p-1}$ norm if one choose $s_p = \frac{n}{2} - \frac{2}{p-1}$. This space is referred to as the critical Sobolev space for the equation.

\paragraph{Shifted wave equation on hyperbolic spaces} In the first part of the paper we consider a semi-linear shifted wave equation on the hyperbolic space $\Hm^n$
\begin{equation}\label{(CP1)}
\left\{\begin{array}{l} \partial_t^2 u - (\Delta_{\Hm^n} + \rho^2) u = \zeta |u|^{p-1} u,\quad (x,t)\in \Hm^n \times \Rm;\\
u |_{t=0} = u_0; \\
\partial_t u |_{t=0} = u_1.\end{array}\right.\qquad \qquad \qquad
\end{equation}
Here the constant $\rho=(n-1)/2$ and $\zeta = \pm 1$. As on the Euclidean spaces we call the equation \eqref{(CP1)} defocusing if $\zeta = -1$,
 otherwise we call it focusing. This is the $\Hm^n$ analogue of the wave equation (\ref{waveequRn}) defined in Euclidean space $\Rm^n$. We can understand this similarity in two different ways.
\begin{itemize}
  \item [(I)] The operator $-\Delta_{\Hm^n}-\rho^2$ in the hyperbolic space and the Laplace operator $-\Delta$ in $\Rm^n$ share the same Fourier symbol $\lambda^2$, see Definition \ref{Sobo}.
    \item [(II)] There is a transformation between solutions of a linear wave equation defined in a forward light cone of $\Rm^n \times \Rm$ and solutions of a linear shifted wave equation defined in the whole space-time $\Hm^n \times \Rm$. This transformation introduced by Tataru \cite{tataru} will help us to translate the same results that can be proved in the hyperbolic spaces back to the  Euclidean space set up later in this paper, see Section \ref{application1}.
\end{itemize}
Solutions to  \eqref{(CP1)} that are smooth enough satisfy the  energy conservation law
\begin{equation} \label{def of energy b}
 \EE(u,\partial_t u) = \int_{\Hm^n} \left[\frac{1}{2}(|\nabla u|^2 - \rho^2 |u|^2) + \frac{1}{2}|\partial_t u|^2 -\zeta \frac{1}{p+1}|u|^{p+1} \right] d\mu = \EE(u_0,u_1),
\end{equation}
where $d \mu$ is the volume element on $\Hm^n$. Since the spectrum of $-\Delta_{\Hm^n}$ is $[\rho^2, \infty)$, it follows that the integral of $|\nabla u|^2 - \rho^2 |u|^2$ above is always nonnegative. If we use  Definition \ref{Sobo} below, we can rewrite the energy in terms of certain norms of solutions
\[
 \EE = \frac{1}{2}\|u\|_{H^{0,1}(\Hm^n)}^2 + \frac{1}{2}\|\partial_t u\|_{L^2 (\Hm^n)}^2 -\zeta \frac{1}{p+1} \|u\|_{L^{p+1}(\Hm^n)}^{p+1}.
\]
As in the case of the wave equation in $\Rm^n$, a solution in the defocusing case always has a positive energy unless it is identically zero. However, a solution in the focusing case may come with a negative energy. The exponent $p_c = 1 + \frac{4}{n-2}$ is called the energy-critical exponent as in the Euclidean spaces, because
\begin{itemize}
  \item The exponent $p_c$ is the largest $p$ so that the $L^{p+1}$ norm of $u$, which appears in the definition of the energy, can be dominated by the $H^1$ norm of $u$ via the Sobolev embedding $H^1 \hookrightarrow L^{p+1}$.
  \item In the case of the wave equation in $\Rm^n$, the dilation (\ref{dilation1}) leaves the energy invariant if and only if $p = p_c$. The shifted wave equation \eqref{(CP1)} on the hyperbolic spaces, however, does not possess a similar natural dilation. This is one of the major differences between these two equations.
\end{itemize}
As in  $\Rm^n$, if $p < p_c$, then we call the equation \eqref{(CP1)} energy-subcritical; otherwise if $p > p_c$, we call it energy-supercritical.

\paragraph{Previous results on Euclidean spaces:} Semi-linear wave equations on $\Rm^n$ have been studied extensively in many works. For example, almost complete results about Strichartz estimates can be found in \cite{strichartz, endpointStrichartz}. Local and global well-posedness has been considered for example in \cite{loc1, ls}. In particular, global existence and well-posedness of solutions with small initial data was proved in the papers\footnote{These four original papers work for even worse nonlinear term $|u|^p$. In contrast, any nontrivial solution to the equation $\partial_t^2 u - \Delta u = |u|^p,\, (x,t)\in \Rm^n \times \Rm$ blows up in finite time if $1 <p <p_0$, as proved in \cite{blowup2, blowup3, counter}.} \cite{smallgs1, gwpwrn, smallgs2, smallgs3}, provided that the exponent $p$ satisfies
\begin{equation} \label{def of p0}
 p > p_0 = \frac{1}{2} + \frac{1}{n-1} + \sqrt{\left(\frac{1}{2} + \frac{1}{n-1}\right)^2 +
 \frac{2}{n-1}},
\end{equation}
as conjectured by Strauss \cite{strcon} in 1989. Questions on global behavior of larger solutions, such as scattering and blow-up, are usually considered more subtle.  Grillakis \cite{mg1, mg2} and Shatah-Struwe \cite{ss1,ss2} proved the global existence and scattering of  solutions in the energy-critical, defocusing case for any $\dot{H}^1 \times L^2$ initial data. The focusing, energy-critical case has been the subject of several more recent papers such as \cite{kenig} (dimension $3\leq n \leq 5$) and \cite{secret, tkm1} (dimension 3). The cases with energy-subcritical ($p < p_c$) or energy-supercritical ($p > p_c$) nonlinearity have been also studied, especially in dimension $3$, usually under an additional assumption on boundedness of the critical Sobolev norm of the solutions, in the papers \cite{dkm2, km, kv2} (supercritical, dimension 3), \cite{ab1} (supercritical, higher dimensions), \cite{kv3} (supercritical, all dimensions) and \cite{kv1, shen2} (subcritical), for instance.

\paragraph{Previous results on hyperbolic spaces:} Much less has been proved  in the case of hyperbolic spaces. Fontaine considered the case $n=2,3$ in \cite{wavehyper97}. Strichartz-type estimates have been discussed by Tataru in \cite{tataru} and  Ionescu in \cite{SThyper}. More recently  Anker,  Pierfelice and Vallarino gave a wide range of Strichartz estimates as well as a brief description on the local well-posedness theory in  \cite{wavehyper}. Global well-posedness is also considered in \cite{wavedr}.

\paragraph{Goal of this paper:} This paper is divided into two parts. The first is concerned with the analysis of the shifted wave equation \eqref{(CP1)} in hyperbolic spaces $\Hm^n$ with $2 \leq n \leq 6$. For this equation we prove a Morawetz-type inequality in the defocusing case similar to the one proved in \cite{hypersdg}. We then use this inequality to obtain scattering type results even in $\Hm^2$. In the second part we transfer the scattering results obtained in $\Hm^n$ to the wave equation (\ref{waveequRn}), with $\zeta = -1$, and we obtain some long time asymptotic behavior for certain large data even in $\Rm^2$. Here is a detailed list of topics discussed in this work.
\begin{itemize}
  \item We first improve a local well-posedness theory for \eqref{(CP1)} previously obtained in \cite{wavehyper}. In comparison with the results in \cite{wavehyper}  we  expand the range of $p$. In the two dimensional case, for example, our local theory works for all $p > 1$, in contrast with the old result $1 < p < 3+\sqrt{6}$ given in Remark 7.4 of \cite{wavehyper}.
      Besides the expanded range of the exponent $p$, we also add more details to the local theory, for example, long time perturbation theory, which is a powerful tool in the discussion of global well-posedness and scattering of solutions, see Section \ref{sec:localth} for details.
  \item The Morawetz-type inequalities, which give global space-time integral estimates, are powerful tools to understand the global behavior of solutions in the defocusing case. We recall that Perthame and  Vega \cite{benoit} established a Morawetz-type inequality
      \[
       \int_{-T_-}^{T_+} \int_{\Rm^n} \frac{|u(x,t)|^{p+1}}{|x|} dx dt < C \EE
      \]
      for the wave equation in the Euclidean space $\Rm^n$ if  dimension $n \geq 3$ with energy $\EE$. It turns out that a similar  Morawetz inequality can be established even  better in the hyperbolic spaces. The Morawerz-type inequality
      \begin{equation} \label{Morawetz}
        \int_{-T_-}^{T_+} \int_{\Hm^n} |u(x,t)|^{p+1} d\mu dt < C \EE
      \end{equation}
      for solutions $u$ to \eqref{(CP1)} proved in this work holds even in the  lower dimension $n=2$, see Theorem \ref{Morawetz1} for details. The advantage of the hyperbolic spaces when looking for a Morawetz-type inequality in $\Hm^n$ has been illustrated
       by the second author in  \cite{hypersdg} with Ionescu, where a similar Morawetz-type inequality for solutions of a semi-linear Schr\"{o}dinger equation in the hyperbolic spaces has been proved.
  \item As the first application of our Morawetz inequality, we show global existence in time and scattering of solutions with  finite energy for the  defocusing problem \eqref{(CP1)} when  $1 < p < p_c$, see Propositions \ref{scattering01},  \ref{scatteringn2pgeq5} and \ref{scattering3n6}. As far as the authors know, the scattering of a semi-linear equation under the same assumption is still unknown in the Euclidean spaces $\Rm^n$, although scattering results have been proved under various additional assumptions in a lot of papers such as \cite{dkm2, km, shen2} (dimension 3) and \cite{conformal2, conformal} (all dimensions).

      A solution in the focusing case, however, may blow up in finite time. For example, we show in the Appendix that a solution for \eqref{(CP1)} with a negative energy always blows up in both time directions. This is similar to the results for the wave equation on the Euclidean spaces, see \cite{kv1, negativeenergy}.
  \item The second application is a scattering result on the following defocusing quintic wave equation in $\Rm^2$
  \begin{equation}\label{(CP02)}
  \left\{\begin{array}{l} \partial_t^2 u - \Delta u = - |u|^4 u, \,\,\,\, (x,t)\in \Rm^2 \times \Rm; \\
  u |_{t=0} = u_0;  \\
  \partial_t u |_{t=0} = u_1. \end{array}\right.
    \end{equation}
  We show in the main Theorem \ref{main1} that the solution scatters if the initial data is radial and satisfies the conditions
  \begin{equation}
   |\nabla u_0 (x)|, |u_1 (x)| \leq A(|x|+1)^{-3/2-\eps};\;\;\; |u_0 (x)| \leq A(|x|)^{-1/2-\eps}. \label{condition1}
  \end{equation}
  One of the key ingredients of the proof is a transformation between solutions of \eqref{(CP02)} and those of \eqref{(CP1star)}
\begin{equation}\label{(CP1star)}
   \left\{\begin{array}{l} \partial_t^2 u - (\Delta_\Hm^2+\rho^2) u = - |u|^4 u, \,\,\,\, (y,t)\in \Hm^2 \times \Rm; \\
  u |_{t=0} = u_0;  \\
  \partial_t u |_{t=0} = u_1; \end{array}\right.
  \end{equation}
   which has been employed by Tataru in \cite{tataru}. This enables us to take advantage of the Morawetz-type inequality on the hyperbolic plane obtained in the first part of this paper, see Theorem \ref{Morawetz2}, which is not available in the Euclidian space $\Rm^2$.
  \end{itemize}
\paragraph{Main Results} For the convenience of readers, we briefly describe our main results as follows. We always assume that the spatial dimension $n$ and the exponent $p$ satisfy $2 \leq n \leq 6$ and $1 < p < p_c = 1 + 4/(n-2)$.
\begin{itemize}
  \item[(I)] The pair $(n,p)$ determines a minimal real number $\sigma(n,p) \in [0,1)$, such that for any positive number $\sigma > \sigma (n,p)$ and initial data $(u_0,u_1) \in H^{\sigma -\frac{1}{2}, \frac{1}{2}} \times H^{\sigma-\frac{1}{2}, -\frac{1}{2}} (\Hm^n)$, there exists a unique solution to the equation \eqref{(CP1)} in a maximal time interval $(-T_-,T_+)$. The choice $\sigma = \sigma (n,p)$ also works for some pairs $(n,p)$, as shown in Proposition \ref{localn36} and Proposition \ref{localpgeq5n2}. Here the definition of Sobolev space $H^{\sigma, \tau} (\Hm^n)$ is given in Definition \ref{Sobo} and the definition of solutions is introduced in Theorem \ref{localtheory}.
  \item[(II)] In addition, if $u$ is a solution to \eqref{(CP1)} in the defocusing case with initial data $(u_0,u_1) \in H^{\frac{1}{2},\frac{1}{2}}\times H^{\frac{1}{2},-\frac{1}{2}} (\Hm^n)$, then it exists globally in time, satisfies
      \[
       \int_{-\infty}^\infty \int_{\Hm^n} |u(x,t)|^{p+1} d\mu (x) dt < \infty
      \]
      and scatters. By scattering we mean that there exist two pairs $(u_0^{\pm}, u_1^{\pm}) \in H^{\frac{1}{2}, \frac{1}{2}} \times H^{\frac{1}{2}, -\frac{1}{2}} (\Hm^n)$, such that
      \[
       \lim_{t \rightarrow \pm \infty} \left\|\left(u(t) - S(t)(u_0^\pm, u_1^\pm), \partial_t u(t) - \partial_t S(t)(u_0^\pm, u_1^\pm)\right)\right\|_{H^{\sigma-\frac{1}{2}, \frac{1}{2}}\times H^{\sigma-\frac{1}{2}, -\frac{1}{2}}(\Hm^n)} = 0
      \]
      holds for each $\sigma \in [\sigma(n,p),1)$. Here $S(t)(u_0^\pm, u_1^\pm)$ is the solution to the linear shifted wave equation on $\Hm^n$ with initial data $(u_0^\pm, u_1^\pm)$.
  \item[(III)] Let $(u_0,u_1)$ be a pair of radial initial data defined on $\Rm^2$ satisfying the inequalities
      \begin{align*}
        |\nabla u_0|, |u_1| & \leq A (1+|x|)^{-\frac{3}{2} -\eps};\\
        |u_0| & \leq A (1+|x|)^{-\frac{1}{2}-\eps};
      \end{align*}
      with two positive constants $A$ and $\eps$. Then the solution $u$ of the wave equation \eqref{(CP02)} exists globally in time and scatters with a finite space-time norm
      \[
       \|u\|_{L^6 L^6 (\Rm \times \Rm^2)} < C(\eps,A) <\infty.
      \]
      Moreover there exists two pairs $(u_0^\pm, u_1^\pm) \in \dot{H}^{\frac{1}{2}} \times \dot{H}^{-\frac{1}{2}} (\Rm^2)$, such that
      \[
       \lim_{t \rightarrow \pm \infty} \left\|\left(u(t) - S(t)(u_0^\pm, u_1^\pm),
       \partial_t u(t) - \partial_t S(t)(u_0^\pm, u_1^\pm)\right)\right\|_{\dot{H}^{\frac{1}{2}} \times \dot{H}^{-\frac{1}{2}}} = 0.
      \]
\end{itemize}
\begin{remark}
A Morawetz inequality similar to \eqref{Morawetz} also holds for suitable solutions to the defocusing Klein-Gordon equations on hyperbolic spaces
\[
 \partial_t^2 u - \Delta_{\Hm^n} u + c u = - |u|^{p-1} u,
\]
where  $c > -\rho^2$ is  constant. As a result, scattering of solutions can be proved in a similar way as in this present work if $1 < p < p_c$. See \cite{wkghyper} for Strichartz estimates and local theory for  this equation.
\end{remark}
\begin{remark}
We recall that there are other and similar results addressing global well-posedness and scattering for problems such as \eqref{(CP02)}.
\begin{itemize}
  \item In \cite{conformal2, conformal} one can find a similar result to the one proved here obtained using the conformal conservation laws. This method works for a large number of defocusing problems like \eqref{(CP02)} if the initial data satisfies the following condition
  \[
    \int_{\Rm^2} \left[(|x|^2+1) (|\nabla u_0 (x)|^2 + |u_1|^2) + |u_0|^2 + |x|^2|u_0|^{p+1}\right] dx < \infty.
  \]
      It turns out that the conformal conservation laws require higher decay rate than we need to assume in our argument. More precisely, a pair $(u_0,u_1) = ((|x|+1)^{-1/2-\eps}, (|x|+1)^{-3/2-\eps})$ satisfies the finite integral condition above only if $\eps > 1/2$. On the other hand, our assumption is $\eps > 0$, which is also the minimal requirement to guarantee $(u_0,u_1) \in \dot{H}^{1/2} \times \dot{H}^{-1/2}(\Rm^2)$. This is in fact relevant since $\dot{H}^{1/2} \times \dot{H}^{-1/2}(\Rm^2)$ is the critical Sobolev space for \eqref{(CP02)}, although $\eps > -1/2$ is sufficient to guarantee that the energy is finite.
  \item  Tsutaya \cite{ktsutaya} also proves scattering results under similar point-wise assumptions on derivatives of the initial data as we did in (\ref{condition1}), and in his work the radial condition is not assumed. However, his work only applies to small data. The reason is that although the fixed-point argument used there applies to more general non-linear terms, it is usually insufficient to deal with global in time solutions generated by large data.
\end{itemize}
\end{remark}

\paragraph{Structure of this paper:} We  review some  preliminary results such as Fourier analysis, Sobolev spaces and Strichartz estimates in Section 2. We then discuss the local well-posedness  theory for all exponents $1 < p <p_c$ in Section 3. Next in Section 4 we prove a Morawetz-type inequality for the shifted wave equation in $\Hm^n$, which immediately leads to the scattering theory in Section 5. Finally in Sections 6 we show another application of the Morawetz inequality, namely, large data scattering results for a quintic wave equation on $\Rm^2$.

\section{Preliminary Results}

We start with a notation. Throughout this work the notation $A \lesssim B$ means that the inequality $A \leq cB$ holds for some constant $c$. Furthermore, a subscript of the symbol $\lesssim$ implies that the constant $c$ depends on the parameter(s) mentioned in the subscript but nothing else.

\subsection{Fourier Analysis}
\paragraph{Model for  hyperbolic space:} Let us first select the following model for the hyperbolic space $\Hm^n$. We consider the Minkowski space $\Rm^{n+1}$ equipped with the standard Minkowski metric $-(dx^0)^2 + (dx^1)^2 + \cdots + (dx^n)^2$ and the bilinear form $[x,y] = x_0 y_0 - x_1 y_1 - \cdots - x_n y_n$. The hyperbolic space $\Hm^n$ is defined as the hyperboloid $x_0^2 - x_1^2 -\cdots -x_n^2 =1$. The Minkowski metric then induces the metric, covariant derivative and measure on the hyperbolic space.

\paragraph{Radial functions:} We can introduce polar coordinates $(r, \Theta)$ on the hyperbolic spaces. More precisely, we use the pair $(r, \Theta) \in [0,\infty) \times {\mathbb S}^{n-1}$ to represent the point $(\cosh r, \Theta \sinh r) \in \Rm^{n+1}$ in the hyperboloid model above. One can check that the $r$ coordinate of a point in $\Hm^n$ represents the distance to the ``origin'', which is the point $(1,\mathbf{0})$ in the Minkiwski space. A function $f$ defined on $\Hm^n$ is radial if it is independent of $\Theta$. By convention we can use the notation $f(r)$ to mention a radial function $f$.

\paragraph{Fourier transform:} (See \cite{fourier1, fourier2} for details) The Fourier transform takes suitable functions defined on $\Hm^n$ to functions defined on $(\lambda, \omega) \in \Rm \times {\mathbb S}^{n-1}$. If we define $b(\omega) = (1,\omega)\in \Rm^{n+1}$ for $\omega \in {\mathbb S}^{n-1}$, we can write down the Fourier transform and its inverse as
\begin{align*}
 \tilde{f} (\lambda,\omega) &= \int_{\Hm^n} f(x) [x,b(\omega)]^{i\lambda -\rho} d\mu,\\
 f(x) & = \int_0^\infty \int_{{\mathbb S}^{n-1}} \tilde{f}(\lambda,\omega) [x,b(\omega)]^{-i\lambda -\rho}
 |\mathbf{c}(\lambda)|^{-2} d\lambda d\omega.
\end{align*}
Here $\mathbf{c}(\lambda)$ is the Harish-Chandra $\mathbf{c}$-function defined by
\[
 \mathbf{c}(\lambda) = C \frac{\Gamma (i \lambda)}{\Gamma (i \lambda + \rho)}
\]
for a constant $C$ determined by the dimension $n$. It is well-known that the Harish-Chandra $\mathbf{c}$-function satisfies the inequality $|\mathbf{c}(\lambda)|^{-2} \lesssim |\lambda|^2 (1+|\lambda|)^{n-3}$. The Fourier transform $f \rightarrow \tilde{f}$ then extends to an isometry from $L^2(\Hm^2)$ to $L^2 (\Rm^+\times {\mathbb S}^{n-1}, |\mathbf{c}(\lambda)|^{-2} d\lambda d\omega)$ with the Plancherel identity
\[
 \int_{\Hm^n} f_1 (x) \overline{f_2 (x)} dx = \int_0^\infty \int_{{\mathbb S}^{n-1}} \tilde{f}_1(\lambda,\omega) \overline{\tilde{f}_2 (\lambda, \omega)} |\mathbf{c}(\lambda)|^{-2} d\lambda d\omega.
\]
In addition, we also have the following identity for the Laplace operator $\Delta_{\Hm^n}$.
\[
 \widetilde{-\Delta_{\Hm^n} f} = (\lambda^2 + \rho^2) \tilde{f}.
\]
This implies that the operator $(-\Delta_{\Hm^n})$ is strictly positive. As a result, the Sobolev spaces defined below are actually inhomogeneous Sobolev spaces unless $\sigma =0$.

\subsection{Sobolev Spaces}

\begin{definition} \label{Sobo}
Let $D = (-\Delta_{\Hm^n} -\rho^2)^{1/2}$ and $\tilde{D} = (-\Delta_{\Hm^n} +1)^{1/2}$.  These operators can also be defined by Fourier multipliers $m_1(\lambda) =\lambda$ and $m_2(\lambda) = (\lambda^2 + \rho^2 +1)^{1/2}$, respectively. We define the following Sobolev spaces and norms for $\gamma < 3/2$.
\begin{align*}
 &H_q^{\sigma}(\Hm^n) = \tilde{D}^{-\sigma} L^q (\Hm^n),&
 &\|u\|_{H_q^{\sigma}(\Hm^n)}= \|\tilde{D}^\sigma u\|_{L^q (\Hm^n)};&\\
 &H^{\sigma,\gamma}(\Hm^n) = \tilde{D}^{-\sigma} D^{-\gamma} L^2 (\Hm^n),&
 &\|u\|_{H^{\sigma,\gamma}(\Hm^n)}= \|D^{\gamma}\tilde{D}^\sigma u\|_{L^2 (\Hm^n)}.&
\end{align*}
\end{definition}
\begin{remark}
If $\sigma$ is a positive integer, one can also define Sobolev spaces by the Riemannian structure. For example, we can first define the $W^{1,q}$ norm as
\[
 \|u\|_{W^{1,p}} = \left(\int_{\Hm^n} |\nabla u|^{q} d\mu\right)^{1/q}
\]
for suitable functions $u$ and then take the closure. Here $|\nabla u| = (\Db_\alpha u \Db^\alpha u)^{1/2}$. It turns out that these two definitions are equivalent to each other if $1 < q < \infty$, see \cite{tataru}. In other words, we have $\|u\|_{H_q^\sigma} \simeq \|u\|_{W^{\sigma,q}}$.
\end{remark}

\begin{definition}
Let $I$ be a time interval. The space-time norm is defined by
\[
 \|u(x,t)\|_{L^q L^r (I \times \Hm^n)} = \left( \int_{I} \left(\int_{\Hm^n} |u(x,t)|^r d\mu\right)^{q/r} dt \right)^{1/q}.
\]
\end{definition}
\begin{proposition} [Sobolev embedding]  Assume $1 < q_1 \leq q_2 < \infty$ and $\sigma_1, \sigma_2 \in \Rm$. If $\sigma_1 -\frac{n}{q_1} \geq \sigma_2 - \frac{n}{q_2}$, then we have Sobolev embedding $H_{q_1}^{\sigma_1}(\Hm^n) \hookrightarrow H_{q_2}^{\sigma_2}(\Hm^n)$.
\end{proposition}
For the proof see \cite{wavehyper, sob} and the references cited therein.
\begin{proposition} \label{Sobolevem1}
 If $q >2$, $0 < \tau < \frac{3}{2}$ and $\sigma + \tau \geq \frac{n}{2} - \frac{n}{q}$, then we have the Sobolev embedding
\[
 H^{\sigma, \tau}(\Hm^n) \hookrightarrow L^{q}(\Hm^n).
\]
\end{proposition}
\begin{proof}
 Choose $\phi: \Rm \rightarrow [0,1]$ be an even, smooth cut-off function so that
 \[
  \phi(r) = \left\{\begin{array}{ll} 1, & r<1;\\
  0, & r>2.
  \end{array}\right.
 \]
 We can define an operator $\mathbf{P}_1$ by the Fourier Multiplier $\lambda \rightarrow \phi(\lambda)$. Then the low frequency part of a function $u$ can be written as
 \[
  \mathbf{P}_1 u = \mathbf{P}_2 \tilde{D}^{\sigma} D^\tau u = (\tilde{D}^{\sigma} D^\tau u) \ast k.
 \]
 Here the operator $ \mathbf{P}_2$ is defined by the Fourier multiplier $\lambda \rightarrow \phi(\lambda) |\lambda|^{-\tau}(\lambda^2 + \rho^2 + 1)^{-\frac{1}{2}\sigma}$ and the function $k$ is the kernel of $ \mathbf{P}_2$. Fourier analysis shows $k \in L^2 (\Hm^n)$ if $\tau < 3/2$. Kunze-Stein phenomenon (see \cite{ksph1,ksph2}) gives the inequality
 \[
  \|(\tilde{D}^{\sigma} D^\tau u) \ast k\|_{L^{q}(\Hm^n)} \lesssim_{q} \|\tilde{D}^{\sigma} D^\tau u\|_{L^2(\Hm^n)} \|k\|_{L^2(\Hm^2)}.
 \]
In other words, we have $\|\mathbf{P}_1 u\|_{L^{q}} \lesssim_q \|u\|_{H^{\sigma, \tau}}$. The high frequency part can be handled by the regular Sobolev embedding
\[
 \|(1-\mathbf{P}_1) u\|_{L^{q}} \lesssim_q \|(1-\mathbf{P}_1) u\|_{H^{\sigma+\tau} (\Hm^n)} \lesssim \|u\|_{H^{\sigma, \tau}(\Hm^n)}.
\]
Combining both low and high frequency parts, we finish the proof.
\end{proof}
\subsection{Strichartz Estimates}

\begin{definition}
 Let $n \geq 2$. A couple $\left(p_1, q_1\right)$ is called admissible if $\left(\frac{1}{p_1}, \frac{1}{q_1}\right)$ belongs to the set
 \begin{align*}
  & T_n = \left\{\left(\frac{1}{p_1}, \frac{1}{q_1}\right) \in \left(0,\frac{1}{2}\right] \times \left(0, \frac{1}{2}\right) \quad \mbox{and} \quad  \frac{2}{p_1} + \frac{n-1}{q_1} \geq \frac{n-1}{2}\right\},& &\quad \hbox{if}\quad \; n\geq 3;&\\
  & T_2 = \left\{\left(\frac{1}{p_1}, \frac{1}{q_1}\right) \in \left(0,\frac{1}{2}\right] \times \left(0, \frac{1}{2}\right)\quad \mbox{and} \quad\frac{2}{p_1} + \frac{1}{q_1} > \frac{1}{2}\right\},& &\quad \hbox{if} \quad \; n=2.&\\
 \end{align*}
\end{definition}
\begin{theorem} [Strichartz estimates]
 Let $(p_1,q_1)$ and $(p_2, q_2)$ be two admissible pairs. Assume that the real numbers $\sigma_1$ and $\sigma_2$ satisfy
 \begin{align*}
  &\sigma_1 \geq \beta (q_1) = \frac{n+1}{2} \left(\frac{1}{2} - \frac{1}{q_1}\right);&
  &\sigma_2 \geq \beta (q_2) = \frac{n+1}{2} \left(\frac{1}{2} - \frac{1}{q_2}\right);&
 \end{align*}
Assume $u(x,t)$ is solution to the linear shifted wave equation
\[
 \left\{\begin{array}{l} \partial_t^2 u - (\Delta_{\Hm^n}+\rho^2) u = F(x,t), \,\,\,\, (x,t)\in \Hm^n \times I;\\
u |_{t=0} = u_0; \\
\partial_t u |_{t=0} = u_1.\end{array}\right.
\]
Then we have
\begin{align}\label{strichartz0}
 \|u\|_{L^{p_1} L^{q_1}(I \times \Hm^n)}&
  + \|(u, \partial_t u)\|_{C(I; H^{\sigma_1 -\frac{1}{2},\frac{1}{2}}\times H^{\sigma_1 -\frac{1}{2},-\frac{1}{2}}(\Hm^n))} \\\notag
 & \leq C \left( \|(u_0, u_1)\|_{H^{\sigma_1 -\frac{1}{2},\frac{1}{2}}\times H^{\sigma_1 -\frac{1}{2},-\frac{1}{2}}(\Hm^n)} + \|F\|_{L^{p'_2} (I; H_{q'_2}^{\sigma_1 + \sigma_2 - 1}(\Hm^n))} \right).
\end{align}
The constant $C$ above does not depend on the time interval $I$.
\end{theorem}
For the proof see Theorem 6.3 and Remark 6.5 in \cite{wavehyper}.
\noindent Now let us rewrite \eqref{strichartz0} in a more convenient form. Applying the operator $\tilde{D}^{\sigma-\sigma_1}$ on both sides of the equation, we obtain
\begin{align}\label{strichartz1}
\|u\|_{L^{p_1} (I; H_{q_1}^{\sigma - \sigma_1}(\Hm^n))}& +  \|(u, \partial_t u)\|_{C(I; H^{\sigma-\frac{1}{2},\frac{1}{2}}\times H^{\sigma-\frac{1}{2},-\frac{1}{2}})}\\\notag
 & \leq C \left( \|(u_0, u_1)\|_{H^{\sigma-\frac{1}{2},\frac{1}{2}}\times H^{\sigma-\frac{1}{2},-\frac{1}{2}}} + \|F\|_{L^{p'_2} (I; H_{q'_2}^{\sigma + \sigma_2 - 1}(\Hm^n))} \right).
\end{align}
In order to make \eqref{strichartz1} as strong as possible, we will always choose $\sigma_1 = \beta (q_1)$ and $\sigma_2 = \beta (q_2)$. Then  \eqref{strichartz1}  becomes
\begin{align}\label{strichartz2}
\|u\|_{L^{p_1} (I; H_{q_1}^{\sigma - \beta(q_1)}(\Hm^n))}& +  \|(u, \partial_t u)\|_{C(I; H^{\sigma-\frac{1}{2},\frac{1}{2}}\times H^{\sigma-\frac{1}{2},-\frac{1}{2}})}\\\notag
 & \leq C \left( \|(u_0, u_1)\|_{H^{\sigma-\frac{1}{2},\frac{1}{2}}\times H^{\sigma-\frac{1}{2},-\frac{1}{2}}} + \|F\|_{L^{p'_2} (I; H_{q'_2}^{\sigma + \beta(q_2) - 1}(\Hm^n))} \right).
\end{align}
Let us fix a number $\sigma\in (0,1)$. If $\sigma - \beta(q_1) \geq 0$, then we can use Sobolev embedding $H_{q_1}^{\sigma - \beta(q_1)}(\Hm^n) \hookrightarrow L^{\tilde{q}_1}$ on the left hand side of \eqref{strichartz2} and obtain 
\begin{align*}
\|u\|_{L^{p_1} (I; L^{\tilde{q}_1} (\Hm^n))}& +  \|(u, \partial_t u)\|_{C(I; H^{\sigma-\frac{1}{2},\frac{1}{2}}\times H^{\sigma-\frac{1}{2},-\frac{1}{2}})}\\
 & \leq C \left( \|(u_0, u_1)\|_{H^{\sigma-\frac{1}{2},\frac{1}{2}}\times H^{\sigma-\frac{1}{2},-\frac{1}{2}}} + \|F\|_{L^{p'_2} (I; H_{q'_2}^{\sigma + \beta(q_2) - 1}(\Hm^n))} \right).
\end{align*}
Here $\tilde{q}_1$ is an arbitrary positive number satisfying
\[
 \frac{1}{\tilde{q}_1} \in \left\{ \begin{array}{ll} \left(0, \frac{1}{q_1}\right], & \hbox{if}\; n=2\; \hbox{and}\; \frac{1}{q_1} \leq 2\sigma - \frac{3}{2};\\
 \left[\frac{n-1}{2n}\frac{1}{q_1} + \frac{n+1}{4n} - \frac{\sigma}{n}, \frac{1}{q_1}\right], & \hbox{otherwise}. \end{array} \right.
\]
The new admissible pair $(p_1,\tilde{q}_1)$ satisfies the inequality
\begin{equation*}
 \frac{1}{p_1} + \frac{n}{\tilde{q}_1} \geq \frac{n}{2} - \sigma.
\end{equation*}

If $\sigma + \beta(q_2) - 1 \leq 0$, we can also apply Sobolev embedding $L^{\tilde{q}'_2} \rightarrow H_{q'_2}^{\sigma + \beta(q_2) - 1}$ on the right hand side of \eqref{strichartz2} and obtain 
\begin{align}
\|u\|_{L^{p_1} (I; L^{\tilde{q}_1} (\Hm^n))}& +  \|(u, \partial_t u)\|_{C(I; H^{\sigma-\frac{1}{2},\frac{1}{2}}\times H^{\sigma-\frac{1}{2},-\frac{1}{2}})}\nonumber \\
 & \leq C \left( \|(u_0, u_1)\|_{H^{\sigma-\frac{1}{2},\frac{1}{2}}\times H^{\sigma-\frac{1}{2},-\frac{1}{2}}} + \|F\|_{L^{p'_2} (I; L^{\tilde{q}'_2}(\Hm^n))} \right) \label{stri01}.
\end{align}
Here $\tilde{q}'_2$ is an arbitrary positive number satisfying
\[
 \frac{1}{\tilde{q}'_2} \in \left\{ \begin{array}{ll} \left[\frac{1}{q'_2},1\right), & \hbox{if}\; n=2\; \hbox{and}\; \frac{1}{q'_2} \geq \frac{1}{2} + 2 \sigma;\\
 \left[\frac{1}{q'_2}, \frac{n-1}{2n} \frac{1}{q'_2} + \frac{n+5}{4n} - \frac{\sigma}{n} \right], & \hbox{otherwise}. \end{array} \right.
\]
By collecting all possible pairs $(p_1, \tilde{q}_1)$ and $(p_2, \tilde{q}_2)$ in the inequality (\ref{stri01}), we have
\begin{proposition}[New version of Strichartz estimates] \label{newStri}
Assume that $0<\sigma<1$,
and the pairs $(p_1,q_1)$, $(p_2,q_2)$ satisfy all conditions listed in the table
\begin{center}
\begin{tabu}{|c|c|c|}
  \hline
  dimension & $n\geq 3$ & $n=2$\\
  \hline
  & (PQ1)\space$\frac{1}{p_1} + \frac{n}{q_1} \geq \frac{n}{2} -\sigma$ & (PQ1)\space$\frac{1}{p_1} + \frac{2}{q_1} > 1 -\sigma $\\
  Left hand & $\frac{1}{p_1} \in \left(0,\frac{1}{2}\right]$ & $\frac{1}{p_1} \in \left(0,\frac{1}{2}\right]$\\
  & (Q1)\space$\frac{1}{q_1} \in \left[\frac{1}{2} - \frac{2 \sigma}{n+1},\frac{1}{2}\right)$ & (Q1)\space$\frac{1}{q_1} \in \left[\frac{1}{2} - \frac{2 \sigma}{3},\frac{1}{2}\right)$ and $\frac{1}{q_1}>0$\\
  \hline
  & (PQ2)\space$\frac{1}{p_2} + \frac{n}{q_2} \geq \frac{n}{2} -1 + \sigma$ & (PQ2)\space$\frac{1}{p_2} + \frac{2}{q_2} > \sigma $\\
  Right hand & $\frac{1}{p_2} \in \left(0,\frac{1}{2}\right]$ & $\frac{1}{p_2} \in \left(0,\frac{1}{2}\right]$\\
  & (Q2)\space$\frac{1}{q_2} \in \left[\frac{n-3}{2(n+1)} + \frac{2 \sigma}{n+1},\frac{1}{2}\right)$ & (Q2)\space$\frac{1}{q_2} \in \left[ \frac{2 \sigma}{3} - \frac{1}{6},\frac{1}{2}\right)$ and $\frac{1}{q_2}>0$\\
  \hline
\end{tabu}
\end{center}
then we have the Strichartz estimate
\begin{align}
\|u\|_{L^{p_1} (I; L^{q_1}(\Hm^n))}& +  \|(u, \partial_t u)\|_{C(I; H^{\sigma-\frac{1}{2},\frac{1}{2}}\times H^{\sigma-\frac{1}{2},-\frac{1}{2}})}\nonumber\\
 & \leq C \left( \|(u_0, u_1)\|_{H^{\sigma-\frac{1}{2},\frac{1}{2}}\times H^{\sigma-\frac{1}{2},-\frac{1}{2}}} + \|F\|_{L^{p'_2} L^{q'_2}(I \times \Hm^n)} \right). \label{newStrichartz}
\end{align}
\end{proposition}
At the end of the paper  we attach 5 figures, which exhibit the admissible pairs $(p_1,q_1)$ we can choose in the left hand side of the Strichartz estimates, with different spatial dimensions and regularity $\sigma$ of the initial data. The lighter region shows the original admissible pairs that are still available for use with the given regularity $\sigma$. While the darker region  illustrates new admissible pairs available.
\begin{remark}
It is necessary to choose $q_1, q_2$ so that $\beta (q_1) \leq \sigma $ and $\beta(q_2) \leq 1 -\sigma$ in order to apply the Sobolev embedding in the argument above. Therefore our assumption $\sigma \in (0,1)$ is necessary.
\end{remark}
\begin{remark}
 In the Strichartz estimates for wave equations on $\Rm^n$ (see \cite{strichartz}), both the inequalities (PQ1) and (PQ2) are replaced by the corresponding identities, which come from the natural dilation of solutions. The lack of dilation is a major difference of hyperbolic spaces from Euclidean spaces when we are discussing wave equations.
\end{remark}
\begin{remark}
 The name of inequalities (PQ1), (Q1), (PQ2) and (Q2) as shown in the table will be used later in this paper.
\end{remark}
\begin{definition}
 If a pair $(p_1,q_1)$ satisfies the conditions listed in  Proposition \ref{newStri}, namely (PQ1), (Q1) and $(1/p_1, 1/q_1) \in (0,1/2] \times (0,1/2)$, then we say $(p_1,q_1)$ is a {\bf(left) $\sigma$-admissible pair}. In particular, if the pair $(p_1,q_1)$ satisfies the strict version of the inequalities mentioned above, then we say $(p_1,q_1)$ is an {\bf open $\sigma$-admissible pair}.
\end{definition}
\begin{remark} \label{openadmissible}
 The definition of an open $\sigma$-admissible pair $(p_1,q_1)$ is equivalent to saying $(1/p_1,1/q_1)$ is in the interior of the admitted region shown in the figures attached at the end of the paper. An open $\sigma$-admissible pair guarantees every pair in a small neighborhood is also $\sigma$-admissible.
\end{remark}
\begin{definition}
 A pair $(p_1,q_1)$ is called a {\bf $(p,\sigma)$-control pair}, if and only if the pair $(p_2,q_2)$ defined by
\[
 (\frac{1}{p_2}, \frac{1}{q_2}) = (1 - \frac{p}{p_1}, 1 -\frac{p}{q_1}).
\]
satisfies $(p_2,q_2) \in (0,\frac{1}{2}]\times (0,\frac{1}{2})$ and the inequalities (Q2), (PQ2).
\end{definition}
\begin{remark} Assume the pair $(p_1,q_1)$ is a $(p,\sigma)$-control pair.
Given a solution to \eqref{(CP1)} with a $\sigma$-admissible pair $(\tilde{p},\tilde{q})$, we can control its norm by
\begin{equation}\label{strichartz100}
 \|u\|_{L^{\tilde{p}} L^{\tilde{q}}(I \times \Hm^n)} \leq C \left( \|(u_0, u_1)\|_{H^{\sigma-\frac{1}{2},\frac{1}{2}}\times H^{\sigma-\frac{1}{2},-\frac{1}{2}}} + \|u\|_{L^{p_1} L^{q_1}(I \times \Hm^n)}^{p} \right).
\end{equation}
\end{remark}

\begin{remark}
 The $L^{p_1} L^{q_1}$ norm of the left hand side of \eqref{strichartz100} could be substituted with $L^{p_1} (I; H_{q_1}^{\kappa})$ if we ``save'' the unused derivative. The range of $\kappa$ is given by
\[
\begin{array}{ll}
 \kappa \in [0, \sigma - \beta (q_1)], & \hbox{if}\; (p_1,q_1)\; \hbox{is an original admissible pair};\\
 \kappa \in [0, \frac{n}{q_1} + \frac{1}{p_1} - \frac{n}{2} + \sigma], & \hbox{if}\; (p_1,q_1)\; \hbox{is not an original admissible pair and}\; n \geq 3;\\
 \kappa \in [0, \frac{2}{q_1} + \frac{1}{p_1} - 1 + \sigma), & \hbox{if}\; (p_1,q_1)\; \hbox{is not an original admissible pair and}\; n = 2.
\end{array}
\]
\end{remark}
\section{Local Theory} \label{sec:localth}
We start with a summary of already know facts.
\subsection{Local Theory Basics}
\begin{definition}
Let $p>1$. We call a pair $(p_1,q_1)$ to be $(p,\sigma)$-compatible if it satisfies both of the following conditions.
\begin{itemize}
  \item The pair $(p_1,q_1)$ is $\sigma$-admissible.
  \item The pair $(p_1,q_1)$ is a $(p,\sigma)$-control pair.
\end{itemize}
\end{definition}
\noindent By the Strichartz estimates, we immediately have
\begin{lemma} Assume $(p_1,q_1)$ is $(p,\sigma)$-compatible. There exists a constant $C$, such that if $u(x,t)$ is a solution to the linear shifted wave equation $\partial_t^2 u -(\Delta_{\Hm^n}+\rho^2) u = F(x,t)$ in a time interval $I$ with initial data $(u_0,u_1)$, then the following inequality holds
\begin{align*}
 \|u\|_{L^{p_1} L^{q_1}(I \times \Hm^n)}&
 +\|(u, \partial_t u)\|_{C(I; H^{\sigma -\frac{1}{2},\frac{1}{2}}\times H^{\sigma -\frac{1}{2},-\frac{1}{2}})} \\
 & \leq C \left( \|(u_0, u_1)\|_{H^{\sigma -\frac{1}{2},\frac{1}{2}}\times H^{\sigma -\frac{1}{2},-\frac{1}{2}}} + \|F\|_{L^{p_1/p} L^{q_1/p}(I \times \Hm^n)} \right).
\end{align*}
\end{lemma}
\noindent 
Combing the lemma above with the inequalities
\begin{align*}
 \|F(u)\|_{L^{p_1/p} L^{q_1/p}(I \times \Hm^n)} &\leq  \|u\|_{L^{p_1} L^{q_1}(I \times \Hm^n)}^p;\\
 \|F(u_1)-F(u_2)\|_{L^{p_1/p} L^{q_1/p}(I \times \Hm^n)} &\leq C_p \left(\|u_1\|_{L^{p_1} L^{q_1}(I \times \Hm^n)}^{p-1} + \|u_2\|_{L^{p_1} L^{q_1}(I \times \Hm^n)}^{p-1} \right)\\
 & \qquad \times \|u_1-u_2\|_{L^{p_1} L^{q_1}(I \times \Hm^n)},
\end{align*}
we obtain the following theorem via a fixed-point argument. (Our argument is standard, see for instance, \cite{bahouri, locad1, kenig, kenig1, ls, local1, ss2} for more details.)
\begin{theorem}[Local theory] \label{localtheory}
 Assume $(p_1,q_1)$ is a $(p,\sigma)$-compatible pair. We say $u(t)$ is a solution of the equation \eqref{(CP1)} in a time interval $I$, if $(u(t),\partial_t u(t)) \in C(I;{H^{\sigma-1/2, 1/2}}\times{H^{\sigma -1/2,-1/2}})$, with a finite norm $\|u\|_{L^{p_1} L^{q_1}(J \times \Hm^n)}$ for any bounded closed interval $J \subseteq I$ so that the integral equation
\[
 u(t) = S(t)(u_0,u_1) + \int_0^t \frac{\sin ((t-\tau)\sqrt{-\Delta_{\Hm^n}-\rho^2})} {\sqrt{-\Delta_{\Hm^n}-\rho^2}} F(u(\tau)) d\tau
\]
holds for all time $t \in I$. Here $S(t)(u_0,u_1)$ is the solution of the linear shifted wave equation with initial data $(u_0,u_1)$. Our local well-posedness results include
 \begin{itemize}
   \item [(a)] (Existence) For any initial data $(u_0,u_1) \in H^{\sigma-1/2,1/2} \times H^{\sigma-1/2,-1/2}(\Hm^n)$, there is a maximal interval $(-T_{-}(u_0,u_1), T_{+}(u_0,u_1))$ in which the equation \eqref{(CP1)} has a solution.
   \item [(b)] (Uniqueness) If there is another solution $\tilde{u}$ with the same initial data in the time interval $J$, then we have $J \subseteq (-T_-,T_+)$ and $u=\tilde{u}$ in $J$.
	   \item [(c)] (Scattering with small data) There exists a constant $\delta_1 > 0$ such that if the norm of the initial data $\|(u_0,u_1)\|_{H^{\sigma-1/2 ,1/2} \times H^{\sigma-1/2 ,-1/2}(\Hm^n)} < \delta_1$, then the Cauchy problem \eqref{(CP1)} has a solution $u$ defined globally in time with $\|u\|_{L^{p_1} L^{q_1} (\Rm \times \Hm^n)} < \infty$.
   \item [(d)] (Standard finite time blow-up criterion) If $T_{+} < \infty$, then $\|u\|_{L^{p_1} L^{q_1}([0,T_{+})\times \Hm^n)} = \infty$.
   \item [(e)] (Finite $L^{p_1} L^{q_1}$ norm implies scattering) If $\|u\|_{L^{p_1} L^{q_1}(\Rm^+ \times \Hm^n)} < \infty$, then there exists a pair $(u_0^+,u_1^+) \in H^{\sigma-1/2,1/2} \times H^{\sigma-1/2,-1/2}(\Hm^n)$, such that
         \[
          \lim_{t \rightarrow + \infty} \left\|\left(u(t) - S(t)(u_0^+, u_1^+),
         \partial_t u(t) - \partial_t S(t)(u_0^+, u_1^+)\right)\right\|_{H^{\sigma-\frac{1}{2},\frac{1}{2}} \times H^{\sigma-\frac{1}{2},-\frac{1}{2}}} = 0.
         \]
       A similar result also holds in the negative time direction.
   \item [(f)] (Long-time perturbation theory, see also \cite{kenig1, shen2, pertao}) Let $M$ be a positive constant. There exists a constant $\eps_0 = \eps_0 (M)>0$, such that if $\eps < \eps_0$, then for any approximation solution $\tilde{u}$ defined on $\Hm^n \times I$ ($0\in I$) and any initial data $(u_0,u_1) \in H^{\sigma-1/2, 1/2} \times H^{\sigma-1/2, -1/2}(\Hm^n)$ satisfying
         \begin{align*}
           &\partial_t^2 \tilde{u} - (\Delta_{\Hm^n} + \rho^2) \tilde{u}=  F(\tilde{u}) + e(x,t), \qquad (x,t) \in \Hm^n \times I; \\
           &\|\tilde{u}\|_{L^{p_1} L^{q_1} (I \times \Hm^n)} < M; \qquad \|(\tilde{u}(0),\partial_t\tilde{u}(0))\|_{H^{\sigma-\frac{1}{2},\frac{1}{2}} \times H^{\sigma-\frac{1}{2},-\frac{1}{2}}}< \infty;\\
           &\|e(x,t)\|_{L^{p_1/p} L^{q_1/p}(I \times \Hm^n)}+ \|S(t)(u_0-\tilde{u}(0),u_1 - \partial_t \tilde{u}(0))\|_{L^{p_1} L^{q_1} (I\times \Hm^n)} \leq \eps;
         \end{align*}
       there exists a solution $u(x,t)$ of \eqref{(CP1)} defined in the interval $I$ with the initial data $(u_0,u_1)$ and satisfying
         \[
           \|u(x,t) - \tilde{u}(x,t)\|_{L^{p_1} L^{q_1} (I\times \Hm^n)} < C(M) \eps.
         \]
         \[
           \sup_{t \in I} \left\|\left(\begin{array}{c} u(t)\\ \partial_t u(t)\end{array}\right)
            - \left(\begin{array}{c} \tilde{u}(t)\\ \partial_t \tilde{u}(t)\end{array}\right)
            - S(t)\left(\begin{array}{c} u_0 - \tilde{u}(0)\\ u_1 -\partial_t \tilde{u}(0)\end{array}\right)
           \right\|_{H^{\sigma-\frac{1}{2},\frac{1}{2}} \times H^{\sigma-\frac{1}{2},-\frac{1}{2}}} < C(M)\eps.
         \]
        Here the notation $S(t)$ with a column vector entry represents
        \[
          S(t) \left(\begin{array}{c} v_0\\ v_1\end{array}\right) = \left(\begin{array}{c} v(t)\\ \partial_t v(t)\end{array}\right)
        \]
        if $v(t)$ is the solution to the linear shifted wave equation with the given initial data $(v_0,v_1)$.
 \end{itemize}
\end{theorem}
\begin{remark}
 The long time perturbation theory implies  immediately the continuous dependence of solutions with respect to the  initial data.
\end{remark}
\begin{remark}
 Given two different $(p,\sigma)$-compatible pairs, we can build a local theory with initial data in $H^{\sigma-\frac{1}{2},\frac{1}{2}} \times H^{\sigma-\frac{1}{2}, -\frac{1}{2}}$ by using either one  of them. However, a basic application of Strichartz estimates shows that these two theories are exactly the same. Namely, a solution in one theory is still the unique solution given by the other. Thus it is absolutely not necessary to point out which $(p,\sigma)$-compatible pair is being used when we mention a local theory. We will simply mention the local theory with (initial data) regularity $\sigma$ in the rest of this paper.
\end{remark}
\begin{remark} \label{admissiblef}
If $u$ is a solution to \eqref{(CP1)} in a local theory with initial data regularity $\sigma$ and a maximal lifespan $(-T_-,T_+)$, then for any closed time interval $J \subset (-T_-,T_+)$ and any $\sigma$-admissible pair $(\tilde{p},\tilde{q})$, we have $\|u\|_{L^{\tilde{p}} L^{\tilde{q}}(J \times \Hm^n)} < \infty$ by the Strichartz estimates.
\end{remark}

According to Theorem \ref{localtheory}, it is sufficient to find a $(p,\sigma)$-compatible pair in order to build a local theory on the equation \eqref{(CP1)}  with exponent $p$ and initial data in the space $H^{\sigma-\frac{1}{2},\frac{1}{2}} \times H^{\sigma-\frac{1}{2},-\frac{1}{2}}(\Hm^n)$. Our main goals now  include
\begin{itemize}
  \item Given an exponent $p$, we are interested in finding a local theory so that the assumption on the regularity of the initial data is as weak as possible. In other words, we are seeking minimal ``working $\sigma$''.
  \item We also try to establish a local theory with regularity $\sigma = 1^-$ for each suitable $p$. According to Remark \ref{admissiblef}, a local theory with initial data regularity $\sigma$ can help us estimate the $L^{\tilde{p}} L^{\tilde{q}}(J \times \Hm^n)$ norm of a solution for a suitable time interval $J$ as long as the pair $(\tilde{p}, \tilde{q})$ is $\sigma$-admissible. By the fact that a $\sigma_1$-admissible pair is also $\sigma_2$-admissible if $\sigma_2 > \sigma_1$, we know that a local theory with a higher regularity level gives us more finite $L^{\tilde{p}} L^{\tilde{q}}$ space-time norms of solutions if the initial data permit. This is the primary reason why we are interested in a local theory with regularity $\sigma=1^-$.
\end{itemize}
\subsection{Local Theory for a Smaller $p$}
The simplest case to consider is when $1 < p < p_{conf} = 1 + 4/(n-1)$ and when $n \geq 3$ we also allow $p = p_{conf}$. In this case we can choose a universal coefficient  $\sigma = \frac{1}{2}$ and
\[
 \left(\frac{1}{p_1}, \frac{1}{q_1}\right) 
 = \left(\frac{1}{p+1}, \frac{1}{p+1}\right).
\]
One can check this pair is a $\left(p,\frac{1}{2}\right)$-compatible pair. In the following proposition we  summarize for the readers' convenience the  allowed $p$ in different dimensions.
\begin{proposition}
 If the exponent $p>1$ satisfies the condition
\[
 \left\{\begin{array}{ll} p \leq p_{conf} = 1 + \frac{4}{n-1}, & n \geq 3;\\
 p < p_{conf} = 5, & n =2;\end{array}\right.
\]
then the pair $(p+1,p+1)$ is a $\left(p,\frac{1}{2}\right)$-compatible pair.
\end{proposition}
\begin{remark}
 If $p < p_{conf}$, the choice $\sigma =\frac{1}{2}$ is not the optimal regularity, as we will  find out later in this section.
 However, the $L^{p+1} L^{p+1}$ norm used here turns out to be among the most {\it friendly} ones to work with in the defocusing case, since, as  the Morawetz inequality will indicate,  it can be dominated solely by the energy.
\end{remark}

\subsection{Local Theory for $3 \leq n \leq 6$}
Let us assume\footnote{The restriction $n\leq 6$ is explained in Remark \ref{remark1}.} $3 \leq n \leq  6$. Our goal is to establish a local theory for the shifted wave equation \eqref{(CP1)}  when the exponent $p$ is greater than $1$ but smaller than the energy-critical exponent $p_c = 1 + 4/(n-2)$. In order to seek possible $(p,\sigma)$-compatible pairs, we first write down all the necessary conditions.
\begin{align}
 &(PQ1)\quad \;\frac{1}{p_1} + \frac{n}{q_1} \geq \frac{n}{2} -\sigma; &
 &(Q1)\quad \;\frac{1}{q_1} \geq \frac{1}{2} - \frac{2 \sigma}{n+1};& \nonumber\\
 &(PQ2)\quad \;\frac{1}{p_2} + \frac{n}{q_2} \geq \frac{n}{2} -1 + \sigma; &
 &(Q2)\quad \;\frac{1}{q_2} \geq \frac{n-3}{2(n+1)} + \frac{2 \sigma}{n+1}; & \nonumber\\
 &p \cdot \frac{1}{p_1} = 1 - \frac{1}{p_2},& & p \cdot \frac{1}{q_1} = 1 - \frac{1}{q_2};& \nonumber\\
 &(\frac{1}{p_i},\frac{1}{q_i}) \in \left(0,\frac{1}{2}\right] \times \left(0, \frac{1}{2}\right).& &&\nonumber
\end{align}
It seems that there are a lot of parameters to be determined given a value $p\in (1,p_c)$. So we  invert the problem. We fix the regularity of the initial data $\sigma \in (0,1)$ and then we figure out the maximum  value of $p$ so that a $(p,\sigma)$-compatible pair can be found. 
\paragraph{Large $\sigma$:} If $\sigma$ is relatively large, we will choose $(p_2,q_2)$ so that both inequalities (PQ2) and (Q2) are identities. A simple computation shows
\[
 \left(\frac{1}{p_2}, \frac{1}{q_2}\right) = \left(\frac{(n-1)(1-\sigma)}{n+1},\frac{n-3 + 4\sigma}{2(n+1)}\right)
 = \left(\frac{n-1}{n+1}(1-\sigma), \frac{1}{2} - \frac{2 (1-\sigma)}{n+1}\right).
\]
As indicated above, this pair can also be given by the point ``A'' in the Figure \ref{drawingd41} at the end of the paper  if we substitute $\sigma$ there by $1 -\sigma$. This is admissible if
\[
 \left\{\begin{array}{ll} \sigma \in (0,1), & n=3;\\
 \sigma \in [(n-3)/2(n-1) ,1), & 4 \leq n\leq 6. \end{array}\right.
\]
The pair $\left(\frac{1}{p_1}, \frac{1}{q_1}\right)$ is chosen as
\[
 \left(\frac{1}{p_1}, \frac{1}{q_1}\right) = \frac{1}{p}\left(\frac{1}{p'_2}, \frac{1}{q'_2}\right)
 =\left( \frac{2 + (n-1)\sigma}{(n+1)p}, \frac{n+5 -4\sigma}{2(n+1)p} \right).
\]
Now let us check which values of $p$ produce admissible pairs $(p_1,q_1)$ satisfying the conditions listed above.
The inequalities (PQ1) and (Q1) are equivalent to
\[
 p \leq p_4(\sigma) = 1 + \frac{4}{n -2\sigma},\qquad \qquad p \leq p_3 (\sigma) = 1 + \frac{4}{n+1 -4\sigma},
\]
respectively. The second upper bound is greater than the first one if and only if $\sigma>\frac{1}{2}$.
We also need to check whether the pair $\left(\frac{1}{p_1},\frac{1}{q_1}\right)$ is contained in $\left(0,\frac{1}{2}\right]\times \left(0,\frac{1}{2}\right)$. The final result is that $p$ must satisfies
\[
 p \geq \frac{4 + 2(n-1)\sigma}{n+1}, \qquad \qquad p > 1 + \frac{4-4\sigma}{n+1}.
\]
A careful computation shows that if $3\leq n \leq 6$ and $\sigma > (n-3)/2(n-1)$, the maximal ``working'' $p = \min\{p_3(\sigma), p_4(\sigma)\}$ always satisfies these inequalities.
\paragraph{Smaller $\sigma$:} When $4\leq n \leq 6$ and $\sigma \leq \frac{(n-3)}{2(n-1)}$, we choose
\[
 \left(\frac{1}{p_2},\frac{1}{q_2}\right) = \left(\frac{1}{2}, \frac{n-3}{2n}+\frac{\sigma}{n}\right).
\]
Thus we have
\[
 \left(\frac{1}{p_1},\frac{1}{q_1}\right)= \left( \frac{1}{2p}, \frac{1}{p}\left(\frac{n+3}{2n}-\frac{\sigma}{n}\right) \right).
\]
As we did for larger $\sigma$, finally we can determine that the  maximum``working $p$'' is
\[
 p_2 (\sigma) = 1 + \frac{3n+3 +2(n-1)\sigma}{n(n+1 -4\sigma)}.
\]
We have the limit $\lim_{\sigma \rightarrow 0^+} p_2 (\sigma) = 1 + \frac{3}{n}$. An exponent $p$ smaller or equal to this value enables us to establish a local theory with arbitrarily small $\sigma > 0$. We have
\begin{proposition} \label{localn36}
Assume $3 \leq n \leq 6$. We can summarize our local well-posedness theory results in two tables, one for $n=3$ and one for $3<n\leq 6$. The first column of each table shows different ranges of $p$, the second column gives minimal assumption on the regularity of initial data $(u_0,u_1) \in H^{\sigma- \frac{1}{2}, \frac{1}{2}} \times H^{\sigma-\frac{1}{2}, -\frac{1}{2}}$, while the last column displays a possible $(p,\sigma)$-compatible pair when $\sigma$ is equal (or close) to the minimal value $\sigma_p$.
\begin{center}
\begin{tabu}{|c|c|c|}
 \hline
 \multicolumn{3}{|c|}{Dimension $n=3$} \\
 \hline
 Range of $p$ & Minimal regularity assumption & $(p,\sigma)$-compatible pair with minimal $\sigma$\\
 \hline
 $3 \leq p <5$ & $\sigma \geq \sigma_3(p) = \frac{3}{2} - \frac{2}{p-1}$ &  $\left(\frac{2p}{1+\sigma_p},\frac{2p}{2 -\sigma_p}\right)$\\
 \hline
 $2 < p \leq 3$ & $\sigma \geq \sigma_2(p) = 1 - \frac{1}{p-1}$ &  $\left(\frac{2p}{1+\sigma_p},\frac{2p}{2 -\sigma_p}\right)$ \\
 \hline
 $1 <p\leq 2$ & $\sigma > \sigma_0(p) = 0$ & $\left(2p, \frac{2}{1-\sigma}\right)$ if $\sigma < \frac{p-1}{p}$ \\
 \hline
 \end{tabu}
\end{center}
\begin{center}
\begin{tabu}{|c|c|c|}
 \hline
 \multicolumn{3}{|c|}{Dimension $4 \leq n \leq 6$} \\
 \hline
 Range of $p$ & Minimal regularity assumption & $(p,\sigma)$-compatible \\
 & & pair with minimal $\sigma$ \\
 \hline
 $1 + \frac{4}{n-1} \leq p < 1 + \frac{4}{n-2}$ & $\sigma \geq \sigma_3(p) = \frac{n}{2} - \frac{2}{p-1}$ &
 $\left(\frac{(n+1)p}{2+(n-1)\sigma_p}, \frac{2(n+1)p}{n+5 -4\sigma_p}\right)$ \\
 \hline
 $1+ \frac{4(n-1)}{(n-1)^2 + 4} \leq p \leq 1 + \frac{4}{n-1}$ & $\sigma \geq \sigma_2(p) = \frac{n+1}{4} - \frac{1}{p-1}$ & $\left(\frac{(n+1)p}{2+(n-1)\sigma_p}, \frac{2(n+1)p}{n+5 -4\sigma_p}\right)$ \\
 \hline
 $1 + \frac{3}{n} < p \leq 1+ \frac{4(n-1)}{(n-1)^2 + 4}$ & $\sigma \geq \sigma_1(p)= \frac{(n+1)(np-n-3)}{4np -2n -2}$ & $\left(2p, \frac{2np}{n+3-2\sigma_p}\right)$ \\
 \hline
 $1 < p \leq 1 + \frac{3}{n}$ & $\sigma > \sigma_0(p) = 0$ & $\left(2p, \frac{2(n+1)}{n+1 -4\sigma}\right)$ if $\sigma < \frac{p-1}{2p}$ \\
 \hline
\end{tabu}
\end{center}
\end{proposition}


\begin{remark} \label{perturbationn35}
 Let us assume $3 \leq n\leq 5$ and $p_{conf}< p <p_c$. Then for any $\sigma \in (\sigma_p, 1)$, one
can check that the $(p,\sigma)$-compatible pair $(p_1,q_1) =\left(\frac{(n+1)p}{2+(n-1)\sigma},\frac{2(n+1)p}{n+5 -4\sigma}\right)$ is an open $\sigma$-admissible pair.
\end{remark}

\subsection{Local Theory for  $n=2$}

In order to establish a local theory for the equation \eqref{(CP1)}  in two spatial dimension, we need to find admissible pairs $(p_1,q_1)$ and $(p_2,q_2)$ satisfying
\begin{align*}
 &(PQ1)\quad \; \frac{1}{p_1} + \frac{2}{q_1} > 1 -\sigma; &
 &(Q1)\quad \; \frac{1}{q_1} \geq \frac{1}{2} - \frac{2}{3} \sigma; & \\
 &(PQ2)\quad \; \frac{1}{p_2} + \frac{2}{q_2} > \sigma; &
 &(Q2)\quad \; \frac{1}{q_2} \geq \frac{2}{3} \sigma - \frac{1}{6}; &\\
 & p \cdot \frac{1}{p_1} = 1 - \frac{1}{p_2},& & p \cdot \frac{1}{q_1} = 1 - \frac{1}{q_2};&\\
 &(AP)\quad \;(\frac{1}{p_i}, \frac{1}{q_i}) \in (0,\frac{1}{2}]\times (0,\frac{1}{2}). & &&
\end{align*}
The idea is exactly the same as the one used for $3 \leq n \leq 6$. Therefore we only show the final conclusion here.
\begin{proposition}\label{localpgeq5n2} Let $n=2$. The local theory of \eqref{(CP1)}  with exponent $p$ is shown in the table
\begin{center}
\begin{tabu}{|c|c|c|}
\hline
 \multicolumn{3}{|c|}{Dimension $n=2$} \\
\hline
 Range of $p$ & Minimal regularity assumption & A possible $(p,\sigma)$-compatible \\
 & & pair with minimal $\sigma$ \\
 \hline
 $5 \leq p <\infty$ & $\sigma > \sigma_3 (p) = 1 -\frac{2}{p-1}$ & $\left(\frac{3p}{2 + \sigma - 3\eps}, \frac{6p}{7- 4\sigma}\right)$ \\
 & & if $0 < \eps < \min\left\{2-(1-\sigma)(p-1),  \frac{1}{3}\right\}$\\
 \hline
 $3< p <5$ & $\sigma \geq \sigma_2 (p) = \frac{3}{4}-\frac{1}{p-1}$ & $\left(\frac{3p}{1+3\sigma_p}, \frac{6}{3-4\sigma_p}\right)$ \\
 \hline
 $2 <p \leq 3$& $\sigma > \sigma_1(p) = \frac{3}{4} - \frac{3}{2p}$ & $\left(\frac{p}{1-\sigma_p},\frac{p}{1-\sigma + \sigma_p}\right)$ if $\sigma - \sigma_p \in (0,1/2)$ \\
 \hline
 $1 < p \leq 2$ & $\sigma > 0$ & $\left(2p, \frac{6}{3-4\sigma}\right)$ if $\sigma < \frac{3(p-1)}{4p}$\\
 \hline
\end{tabu}
\end{center}
\end{proposition}
\begin{remark} \label{perturbationallowance}
 If $p \geq 5$, a careful calculation shows the choice $(p_1,q_1) = (\frac{3p}{2 + \sigma - 3\eps}, \frac{6p}{7- 4\sigma})$ above is an open $\sigma$-admissible pair.
\end{remark}

\subsection{Local theory for $\sigma = 1^-$} \label{sec:localoneminus}
The admissible pairs mentioned in the lemmas below can be used to construct a local theory with a regularity level of initial data close to $1$. However, the primary role of these pairs is to serve as one of the two endpoints in an interpolation of admissible pairs, which will be used frequently in our discussion.
\begin{lemma} Let $3 \leq n \leq 6$. We have
\begin{itemize}
  \item If $1 < p < 2$, then the pair $\left(2, \frac{2(n+1)p}{n+3-2\sigma}\right)$ is both a $(p,\sigma)$-compatible pair and a $(p,\frac{1+\sigma}{2})$-control pair when $\sigma$ is sufficiently close to $1$.
  \item If $2 \leq p < p_c$, then the pair $\left(\frac{2(n+1)p}{n+3+(n-1)\sigma}, \frac{2(n+1)p}{n+3-2\sigma}\right)$ is both a $(p,\sigma)$-compatible pair and a $(p,\frac{1+\sigma}{2})$-control pair when $\sigma$ is sufficiently close to $1$.
\end{itemize} \label{s1minuslocalfor3}
\end{lemma}
\begin{lemma} Let $n=2$. We have
\begin{itemize}
  \item If $1 < p < 2$, then the pair $\left(2, \frac{6p}{5-2\sigma}\right)$ is both a $(p,\sigma)$-compatible pair and a $(p,\frac{1+\sigma}{2})$-control pair when $\sigma$ is sufficiently close to $1$.
  \item If $2 \leq p < p_c = +\infty$, then the pair $\left(\frac{6p}{5+\sigma}, \frac{6p}{5-2\sigma}\right)$ is both a $(p,\sigma)$-compatible pair and a $(p,\sigma_1)$-control pair for all $\sigma_1 \in (0, \frac{\sigma+1}{2})$ when $\sigma$ is sufficiently close to $1$.
\end{itemize} \label{s1minuslocalfor2}
\end{lemma}
\begin{remark} \label{existenceloc}
An interpolation between $\sigma = \sigma_p$ (or $\sigma = \sigma_p^+$ if $\sigma_p$ is forbidden) and $\sigma = 1^{-}$ shows that a local theory with a regularity level $\sigma$ is always available whenever $\sigma \in (\sigma_p,1)$.
\end{remark}
\begin{remark} \label{perturbationn6}
In the case $n=6$ and $p_{conf} < p < p_c = 2$, an interpolation between $\sigma= \sigma_p$ and $\sigma = 1^-$ gives a $(p, \sigma)$-compatible pair that is also open $\sigma$-admissible for each $\sigma \in (\sigma_p,1)$. Here we use the admissible pairs given in Proposition \ref{localn36} and Lemma \ref{s1minuslocalfor3}, respectively.
\end{remark}

\subsection{Uniqueness of solutions on different regularity levels}
As one may find in the argument above, given $p \in (1,p_c)$, we can establish different local well-posedness theories on the same equation \eqref{(CP1)}  if we choose two different Sobolev spaces $H^{\sigma_1-\frac{1}{2}, \frac{1}{2}} \times H^{\sigma_1-\frac{1}{2}, -\frac{1}{2}}$ and $H^{\sigma_2-\frac{1}{2}, \frac{1}{2}} \times H^{\sigma_2-\frac{1}{2}, -\frac{1}{2}}$ with two different allowed coefficients $\sigma_1 < \sigma_2$. If the initial data permit, we might obtain two solutions, $(u_1, \partial_t u_1) \in C(I_1; H^{\sigma_1-\frac{1}{2}, \frac{1}{2}} \times H^{\sigma_1-\frac{1}{2}, -\frac{1}{2}})$ with a maximal lifespan $I_1$ and $(u_2, \partial_t u_2) \in C(I_2; H^{\sigma_2-\frac{1}{2}, \frac{1}{2}} \times H^{\sigma_2-\frac{1}{2}, -\frac{1}{2}})$ with a maximal lifespan $I_2$. The natural question is whether they are exactly the same. First of all, Fourier analysis gives the embedding
\[
 H^{\sigma_2-\frac{1}{2}, \frac{1}{2}} \times H^{\sigma_2-\frac{1}{2}, -\frac{1}{2}}(\Hm^n) \hookrightarrow
 H^{\sigma_1-\frac{1}{2}, \frac{1}{2}} \times H^{\sigma_1-\frac{1}{2}, -\frac{1}{2}}(\Hm^n).
\]
In addition, the Strichartz estimates give that $\|u_2\|_{L^{p_1} L^{q_1}(J \times \Hm^n)} < \infty$ for any $J$ compactly supported in $I_2$ and any $\sigma_1$-admissible pair $(p_1,q_1)$. Thus the solution $u_2$ defined on a higher regularity level $\sigma_2$ remains a solution in the local theory with a lower regularity level $\sigma_1$. This means $I_2 \subseteq I_1$ and $u_1 = u_2$ in $I_2$. But the question is whether $I_2 = I_1$, or the solution might blow-up in a space of higher regularity at some time but continue to exist after this time in a space of lower regularity . The answer is exactly as we expect it: Local theories in Sobolev spaces of different regularity levels do share the same solution in the same interval of time. Let us start with a lemma.
\begin{lemma}\label{lemma1g}
  Assume $2\leq n\leq 6$ and $1 < p < p_c$. Let $\sigma_p$ be the critical level of regularity as we found it  in the Proposition \ref{localn36} and Proposition \ref{localpgeq5n2}. Given a closed interval $K = [a,b] \subset (\sigma_p,1)$, there exists a constant $\delta_0 = \delta_0 (n,p,K) > 0$, such that if the initial data $(u_0,u_1)$ are in the space $H^{\sigma + \delta -\frac{1}{2}, \frac{1}{2}} \times H^{\sigma + \delta-\frac{1}{2}, -\frac{1}{2}}(\Hm^n)$ for a regularity level $\sigma \in K$ and a positive coefficient $\delta < \delta_0$, then the lifespan $I_{\sigma+\delta}$ of the solution $u$ in the higher regularity setting $\sigma +\delta$ is the same as its lifespan $I_\sigma$ in the lower regularity setting $\sigma$.
\end{lemma}
\begin{proof}
 Choosing a regularity level $\sigma_1 < a$ with a $(p,\sigma_1)$-compatible pair $(p_1,q_1)$ and another $\sigma_2 = 1^- > b$ with a $(p,\sigma_2)$-compatible pair $(p_2,q_2)$ as described in the subsection \ref{sec:localoneminus}. If we consider the pair $(\tilde{p},\tilde{q})$ defined by the interpolation
\[
 \left(\frac{1}{\tilde{p}},\frac{1}{\tilde{q}}\right) = \frac{\sigma_2 -\sigma}{\sigma_2 -\sigma_1} \left(\frac{1}{p_1}, \frac{1}{q_1}\right)
 + \frac{\sigma -\sigma_1}{\sigma_2 -\sigma_1}\left(\frac{1}{p_2}, \frac{1}{q_2}\right),
\]
then it is both a $\sigma$-admissible pair and a $(p,\sigma + \delta)$-control pair for any positive number $\delta < \delta_0$, where the constant $\delta_0 \in (0, 1-b)$ is independent of $\sigma \in K$. If the initial data is in the space $H^{\sigma + \delta -\frac{1}{2}, \frac{1}{2}} \times H^{\sigma + \delta-\frac{1}{2}, -\frac{1}{2}}(\Hm^n)$, Strichartz estimates immediately give
\[
 \|u\|_{L^{p_3} L^{q_3}([0,T] \times \Hm^n)} \lesssim \|(u_0,u_1)\|_{H^{\sigma + \delta -\frac{1}{2}, \frac{1}{2}} \times H^{\sigma + \delta-\frac{1}{2}, -\frac{1}{2}}(\Hm^n)} + \|u\|_{L^{\tilde{p}} L^{\tilde{q}}([0,T]\times \Hm^n)}^p < \infty,
\]
for any $[0,T]$ compactly supported in the lifespan $I_\sigma$ and any $(\sigma+\delta)$-admissible pair $(p_3,q_3)$. But this implies that the solution has not yet blown up in the space of higher regularity by the finite time blow-up criterion. This finishes our proof.
\end{proof}
This lemma settles the question  above the critical regularity level. The remaining task is to show if a similar result is still true  exactly at the critical level of regularity.
\begin{lemma}\label{lemma2g}
  Assume $2\leq n\leq 6$ and $1 < p < p_c$. Let $\sigma_p$ be the critical level of regularity with a $(p,\sigma_p)$-compatible pair $(\tilde{p},\tilde{q})$ as in the Proposition \ref{localn36} and Proposition \ref{localpgeq5n2}, if applicable. Given any initial data $(u_0,u_1) \in H^{\sigma -\frac{1}{2}, \frac{1}{2}} \times H^{\sigma-\frac{1}{2}, -\frac{1}{2}}(\Hm^n)$ with $\sigma \in(\sigma_p,1)$, we can always find a number $\sigma_1 \in (\sigma_p,\sigma]$, such that the lifespan $I_{\sigma_1}$ of the solution $u$ in the higher regularity setting $\sigma_1$ is the same as the lifespan $I_{\sigma_p}$ in the lower regularity setting $\sigma_p$.
\end{lemma}
\begin{proof}
 Given a $(p,\sigma_p)$-compatible pair $(\tilde{p},\tilde{q})$ as in Proposition \ref{localn36} or Proposition \ref{localpgeq5n2}, we can always find a small positive number $\eps \leq \frac{2}{n+1} (\sigma - \sigma_p)$, such that the pairs $(p_1,q_1)$ and $(p_2,q_2)$ defined by
\begin{align*}
 \left(\frac{1}{p_1},\frac{1}{q_1}\right) & = \left(\frac{1}{\tilde{p}} + \frac{n-1}{2}\eps, \frac{1}{\tilde{q}}-\eps\right);\\
 \left(\frac{1}{p_2},\frac{1}{q_2}\right) & = \left(1 - \frac{p}{\tilde{p}} - \frac{n-1}{2}\eps, 1 - \frac{p}{\tilde{q}} +\eps\right)
\end{align*}
satisfy $(\frac{1}{p_1},\frac{1}{q_1}), (\frac{1}{p_2},\frac{1}{q_2}) \in (0,\frac{1}{2})\times (0,\frac{1}{2})$. The point here is that we never choose $\tilde{p} =2$ when $\sigma_p$ is the critical regularity.
Let $\sigma_1 = \sigma_p + \frac{n+1}{2}\eps \in (\sigma_p, \sigma]$ and $(p_3,q_3)$ be a $(p,\sigma_1)$-compatible pair. Now we claim $I_{\sigma_1} = I_{\sigma_p}$. If this were false, we could assume the right endpoint (blow-up time) $T_1$ of $I_{\sigma_1}$ is finite and still contained in $I_{\sigma_p}$, since the wave equation is time-reversible. The choices of $(p_1,q_1)$ and $(p_2,q_2)$ enable us to apply Strichartz estimates and obtain
\begin{align*}
 \|u\|_{L^{p_3} L^{q_3}([a,b]\times \Hm^n)} + \|u\|_{L^{p_1} L^{q_1}([a,b]\times \Hm^n)} \leq & C \|(u(a),\partial_t u(a))\|_{H^{\sigma_1-\frac{1}{2},\frac{1}{2}}\times H^{\sigma_1-\frac{1}{2},\frac{1}{2}}(\Hm^n)}\\
  &\qquad + C \|F(u)\|_{L^{p'_2}L^{q'_2}([a,b]\times \Hm^n)}
\end{align*}
for any time interval $[a,b]\subseteq I_{\sigma_1}$. In addition, since
\begin{align*}
 &\frac{1}{p'_2} = \frac{p-1}{\tilde{p}} + \frac{1}{p_1},& &\frac{1}{q'_2} = \frac{p-1}{\tilde{q}} + \frac{1}{q_1};&
\end{align*}
we have
\begin{align}
 \|u\|_{L^{p_3} L^{q_3}([a,b]\times \Hm^n)} + \|u\|_{L^{p_1} L^{q_1}([a,b]\times \Hm^n)} \leq  & C \|(u(a),\partial_t u(a))\|_{H^{\sigma_1-\frac{1}{2},\frac{1}{2}}\times H^{\sigma_1-\frac{1}{2},\frac{1}{2}}(\Hm^n)}\nonumber\\
 & \qquad +  C \|u\|_{L^{\tilde{p}} L^{\tilde{q}}([a,b]\times \Hm^n)}^{p-1} \|u\|_{L^{p_1} L^{q_1}([a,b]\times \Hm^n)}. \label{mixedStrichartz}
\end{align}
Since $T_1 \in I_{\sigma_p}$, we can always find a time $a < T_1$ so that $C\|u\|_{L^{\tilde{p}} L^{\tilde{q}}([a,T_1]\times \Hm^n)}^{p-1} < \frac{1}{2}$. Combining this estimate with (\ref{mixedStrichartz}) and the fact that the left hand norms there are finite for any $b \in (a,T_1)$, we immediately obtain
\[
 \|u\|_{L^{p_3} L^{q_3}([a,b]\times \Hm^n)} + \|u\|_{L^{p_1} L^{q_1}([a,b]\times \Hm^n)} \leq 2C  \|(u(a),\partial_t u(a))\|_{H^{\sigma_1-\frac{1}{2},\frac{1}{2}}\times H^{\sigma_1-\frac{1}{2},\frac{1}{2}}(\Hm^n)}
\]
for any time $b \in (a, T_1)$. Sending $b$ to $T_1$, we have $\|u\|_{L^{p_3} L^{q_3}([a,T_1)\times \Hm^n)} < \infty$.
This immediately gives a contradiction with the finite time blow-up criterion.
\end{proof}
The combination of Lemma \ref{lemma1g} and \ref{lemma2g}  yields
\begin{proposition} \label{uniquesolution}
 Although we can define solutions to the equation \eqref{(CP1)}  on different levels of regularity if the initial data permit, the solution with given initial data is always unique.
\end{proposition}

\section{A Morawetz Inequality in the Hyperbolic Spaces}\label{section4}

In this section, we show that a solution to \eqref{(CP1)} in the defocusing case satisfies a Morawetz inequality. the approach is similar to the one proposed in \cite{hypersdg} for the NLS equation.
\begin{theorem} \label{Morawetz1}
 Assume $2 \leq n \leq 6$ and $1 < p < p_c$. Let $u$ be a solution of \eqref{(CP1)}  in the defocusing case with initial data $(u_0,u_1)\in H^1 \times (H^{\frac{1}{2},-\frac{1}{2}})(\Hm^n)$ and a maximal lifespan $(-T_-,T_+)$. Then $u$ satisfies the following inequality
\[
 \int_{-T_-}^{T_+} \int_{\Hm^n} |u|^{p+1} d\mu dt < \frac{4(p+1)}{p-1} \EE.
\]
\end{theorem}
The main ingredients of the proof is  the following informal computation for a solution $u$
\begin{align*}
   & -\frac{d}{dt} \int_{\Hm^n} \left( u_t \Db^\alpha a \Db_\alpha u + u_t u \cdot \frac{\Delta a}{2} \right) d\mu \\
 = & \int_{\Hm^n} \left(\Db_\beta u \Db^\beta \Db^\alpha a \Db_\alpha u \right) d\mu
 - \frac{1}{4} \int_{\Hm^n} \left( |u|^2 \Delta \Delta a\right) d\mu
 + \frac{p-1}{2(p+1)} \int_{\Hm^n} \left( |u|^{p+1}\Delta a \right) d\mu.
\end{align*}
and Lemma \ref{def of a}. We consider real-valued solutions throughout this section for convenience, but one can consider complex-valued solutions as well in exactly the same way. The full proof of Theorem \ref{Morawetz1} is given in Subsection \ref{T1proof}.
\begin{remark}
 Fourier analysis and Sobolev embedding gives us
\begin{align*}
 &H^1 \times (H^{\frac{1}{2},-\frac{1}{2}})(\Hm^n) \hookrightarrow (H^{0,1}\cap L^{p+1}) \times L^2,&
 &H^1 \times (H^{\frac{1}{2},-\frac{1}{2}})(\Hm^n) \hookrightarrow H^{\sigma-\frac{1}{2}, \frac{1}{2}}\times H^{\sigma-\frac{1}{2}, -\frac{1}{2}}
\end{align*}
for all $\sigma \in (0,1)$. Thus the assumptions on the initial data enable us to apply the local theory discussed in the previous Section \ref{sec:localth} and guarantee that the energy is finite.
\end{remark}
\subsection{Preliminary Results} In this subsection we collect some technical lemmata we will need to prove Theorem \ref{Morawetz1} above.
\begin{lemma} \label{def of a} (See Lemma 4.2 in \cite{hypersdg})
There is a smooth, radial function $a : \Hm^n \rightarrow [0, \infty)$ with the following properties:
\[
\left\{\begin{array}{ll}
  \Delta a = 1, &\hbox{in}\; \Hm^n;\\
  |\nabla a| = |\Db^\alpha a \Db_\alpha a|^{1/2} \leq C, & \hbox{in}\; \Hm^n;\\
  \Db^2 a \geq 0,&  \hbox{in}\; \Hm^n.
 \end{array}\right.
\]
\end{lemma}
\begin{remark} \label{gradient of a} By considering the polar representation of $\Delta a = 1$ we have that the radial function $a(r)$ is defined by the equation
\[
 \left(\partial_r^2 +(n-1) \frac{\cosh r}{\sinh r}\partial_r \right) a (r) = 1.
\]
From here
\begin{align*}
 \partial_r a (r) & = \frac{1}{(\sinh r)^{n-1}} \int_0^r (\sinh s)^{n-1} ds \\
 & \leq \frac{1}{(\sinh r)^{n-1}} \int_0^r (\sinh s)^{n-2} (\cosh s) ds \\
  &= \frac{1}{n-1}
  = \frac{1}{2\rho}.
\end{align*}
Thus we can choose $C =\frac{1}{2\rho}$ in Lemma \ref{def of a}.
\end{remark}
\begin{remark}
As it was remarked in \cite{hypersdg}, such a function $a(r)$ does not exist in the Euclidean spaces $\Rm^n$. This is the main reason why we are able to prove a more convenient Morawetz inequality in the hyperbolic spaces and it reflects the difference in the geometrical nature of these two spaces.
\end{remark}
\noindent  An alternative estimate to the forbidden endpoint $(p_2,q_2)=(\infty,2)$ of the regular Strichartz estimates addressed in Proposition \ref{newStri} is given in the following Lemma \ref{Strichartzbasic}. This helps us obtaining  an estimate on the norm $\|(u,\partial_t u)\|_{H^1 \times L^2(\Hm^n)}$ in Lemma \ref{CH1} below.
\begin{lemma} \label{Strichartzbasic}
 If $v$ is the solution to the linear shifted wave equation $\partial_t^2 v - (\Delta_{\Hm^n}+\rho^2) v = F$ in the time interval $[0,T]$ with initial data $(v_0,v_1)$, then
\[
 \|(v, \partial_t v)\|_{C([0,T]; H^{0,1} \times L^2 (\Hm^n))} \lesssim \|(v_0,v_1)\|_{H^{0,1} \times L^2 (\Hm^n)} +
 \|F\|_{L^1 L^2 ([0,T]\times \Hm^n)}.
\]
\end{lemma}
\begin{proof}
The lemma immediately follows the Plancherel identity and the below.
\begin{align*}
\tilde{v}(\lambda, \omega, t) = & [\cos (t \lambda)] \tilde{v}_0 (\lambda, \omega) + \frac{\sin (t \lambda)}{\lambda} \tilde{v}_1 (\lambda, \omega) + \int_0^t \frac{\sin (t-\tau)\lambda}{\lambda} \tilde{F} (\lambda, \omega, \tau) d\tau;\\
\widetilde{\partial_t v} (\lambda, \omega, t) = & -[\lambda \sin (t\lambda)] \tilde{v}_0 (\lambda, \omega) + [\cos (t\lambda)] \tilde{v}_1 (\lambda, \omega) + \int_0^t [\cos (t-\tau)\lambda] \tilde{F} (\lambda, \omega, \tau) d\tau.
\end{align*}
\end{proof}
\begin{lemma} \label{CH1} Let $u$ be a solution as in Theorem \ref{Morawetz1}. Then for any time interval $[-T_1,T_2]\subset (-T_-,T_+)$, we have
\[
 M := \|(u, \partial_t u)\|_{ C([-T_1,T_2]; H^1 \times L^2 (\Hm^n))}  < \infty.
\]
\end{lemma}
\begin{proof} By Remark \ref{admissiblef}, Remark \ref{existenceloc} and Proposition \ref{uniquesolution}, our choice of initial data  enables us to claim that the space-time norm $\|u\|_{L^{p_1} L^{q_1}([-T_1,T_2] \times \Hm^n)}$ is finite as long as $(p_1,q_1)$ is a $\sigma$-admissible pair with $\sigma \in (\sigma_p,1)$. Our goal is to show
\[
 \|u\|_{L^{p} L^{2p}([-T_1,T_2]\times \Hm^n)} < \infty.
\]
This is equivalent to $F(u) \in L^1 L^2 ([-T_1,T_2]\times \Hm^n)$. There are two cases
\begin{itemize}
  \item If $p\geq 2$, the pair $(p,2p)$ is $\sigma$-admissible if
  \begin{align*}
     &(PQ1)\quad \;\frac{1}{p} + \frac{n}{2p} \geq \frac{n}{2}- \sigma = (1 +\frac{n}{2})\frac{1}{p_c} +1 - \sigma;& &\frac{1}{2p} \geq \frac{1}{2} - \frac{2 \sigma}{n+1}& &\hbox{if}\; n \geq 3;&\\
     &\qquad\qquad\frac{1}{p} + \frac{2}{2p} > 1- \sigma;& &\frac{1}{2p} \geq \frac{1}{2} - \frac{2 \sigma}{3}& &\hbox{if}\; n = 2;&
  \end{align*}
  This is always true if $\sigma$ is sufficiently close to $1$ as long as $p < p_c$.
  \item If $1 < p < 2$, we have to use the pair $(2,2p)$ instead. It turns out that this pair is still $\sigma$-admissible for $\sigma$ sufficiently close to $1$ as long as $p < p_c$ and $n \leq 6$. By the embedding $L^2([-T_1,T_2]) \hookrightarrow L^p([-T_1,T_2])$, we still have $u \in L^p L^{2p}([-T_1,T_2]\times \Hm^n)$.
\end{itemize}
Now we can  apply Lemma \ref{Strichartzbasic} and obtain
\begin{equation}\label{l21}
 \|(u, \partial_t u)\|_{C([-T_1,T_2]; H^{0,1} \times L^2 (\Hm^n))} < \infty.
\end{equation}
Integrating $\partial_t u$ in $t$, we immediately obtain that $u \in C([-T_1,T_2]; L^2(\Hm^n))$. Combining this $L^2$ estimate with (\ref{l21}), we finish the proof.
\end{proof}
\begin{remark}\label{remark1} Above we have been restricting the dimension  to $n<7$.
 In fact, the main obstacle when $n \geq 7$ is the inequality (PQ1). If we plug in $(2,2p)$, it reads
\begin{equation}
  \frac{1}{2} + \frac{n}{2} \cdot \frac{1}{p} \geq \left(1 + \frac{n}{2}\right)\cdot \frac{1}{p_c} + (1-\sigma). \label{PQ1star}
\end{equation}
 If $3 \leq n \leq 6$, then the first term above $\frac{1}{2} \geq \frac{1}{p_c}\equiv \frac{n-2}{n+2}$. When $p <p_c$ one can easily check that
\[
 \frac{1}{2} + \frac{n}{2} \cdot \frac{1}{p} > \left(1 + \frac{n}{2}\right)\cdot \frac{1}{p_c},
\]
thus the inequality (\ref{PQ1star}) holds as $\sigma \rightarrow 1^-$. However, if $n \geq 7$, we have $\frac{1}{2}  < \frac{1}{p_c}$, which makes the inequality (\ref{PQ1star}) fail for all $\sigma < 1$ if $p$ is sufficiently close to $p_c$.
\end{remark}
\begin{definition} \label{smoothcutoff}
 Let $\psi: \Rm \rightarrow [0,1]$ be a smooth cut-off function satisfying
\[
 \psi(r) = \left\{\begin{array}{ll}
  1, & r < 1;\\
  0, & r > 2.
 \end{array}\right.
\]
If $\delta \in \left(0,\frac{1}{10}\right]$, we define a radial smooth cut-off function on $\Hm^n$
\[
 \psi_{\delta}(r) = \psi (\delta r).
\]
It is clear that $|\nabla \psi_\delta| \lesssim \delta$.
\end{definition}
\begin{lemma}\label{Lqmultiplier} (See Theorem 5.1 of \cite{staten}, and \cite{ap1}) Assume $q \in (1,\infty)$. If $m(\lambda)$ is an even analytic function defined in the region
\[
 S = \{\lambda \in {\mathbb C}: |\hbox{Im}\, \lambda| < \rho\},
\]
and satisfying the following symbol-type bounds in $S$
\begin{equation} \label{symbol bound1}
 |\partial_x^\alpha m(x+yi)| \leq C_{\alpha}(1 + |x|)^{-\alpha},\quad \alpha = 0,1,2,\cdots, N;
\end{equation}
then the operator $\mathbf{T}_m$ defined by the Fourier multiplier $\lambda \rightarrow m(\lambda)$ is bounded from $L^q (\Hm^n)$ to itself. Here $N$ is an integer determined by the dimension $n$. In fact, an upper bound for the norm $\|\mathbf{T}\|_{L^q \rightarrow L^q}$ can be determined by the constants $C_\alpha$.
\end{lemma}
\begin{lemma}\label{smoothing}  Assume $\eps \in \left(0,\frac{1}{10}\right)$.
 Let $\mathbf{\tilde{P}}_\eps$ be the smoothing operator defined by the Fourier multiplier $\lambda \rightarrow e^{-\eps^2 \lambda^2}$. Given any $2 \leq q < \infty$, we have $\|\mathbf{\tilde{P}}_\eps\|_{L^q(\Hm^n) \rightarrow L^q(\Hm^n)} \leq C_q < \infty$ for each $\eps$. Furthermore, if $v \in L^q (\Hm^n)$, then $\|v - \mathbf{\tilde{P}}_\eps v\|_{L^q} \rightarrow 0$ as $\eps \rightarrow 0$.
\end{lemma}
\begin{proof}
A simple calculation shows the symbols $m_\eps (\lambda) = e^{-\eps^2 \lambda^2}$ satisfy the condition (\ref{symbol bound1}) with the constants $C_{\alpha}$ independent of $\eps$. As a result, we immediately obtain a universal upper bound $C_q$ for all norms $\|\mathbf{\tilde{P}}_\eps\|_{L^q \rightarrow L^q}$ independent of $\eps$ according to Lemma \ref{Lqmultiplier}. If $v \in H^{\sigma}$ for a large number $\sigma$, then Sobolev embedding gives us
\[
 \|v - \mathbf{\tilde{P}}_\eps v\|_{L^q} \lesssim \|v- \mathbf{\tilde{P}}_\eps v\|_{H^{\sigma}} \rightarrow 0.
\]
Since $H^{\sigma} \cap L^q$ is dense in $L^q$, we can prove the convergence $\|v - \mathbf{\tilde{P}}_\eps v\|_{L^q} \rightarrow 0$ for a general $L^q$ function $v$ by basic approximation techniques.
\end{proof}
\paragraph{Space-time smoothing operator} Choose a smooth, nonnegative, even function $\phi(t)$ compactly supported in $[-1,1]$ with $\int_{-1}^{1} \phi (t) dt =1$. Given a closed interval $[-T_1,T_2] \subset (-T_-,T_+)$, let $\eps < \eps_0 = \frac{1}{2}\min\{1/10, T_+ - T_2, T_{-} -T_1\}$. We can smooth out the solution $u$ and the non-linear term $F(u)$ by defining $u_\eps$ and $F_\eps$ as
\begin{equation}\label{newequation}
 u_\eps (\cdot, t) = \int_{-1}^{+1} \phi(s) \mathbf{\tilde{P}}_\eps u(\cdot,t+s\eps) ds;\quad
 F_\eps (\cdot, t) = \int_{-1}^{+1} \phi(s) \mathbf{\tilde{P}}_\eps F(u(\cdot,t+s\eps)) ds.
\end{equation}
The function $u_\eps$ is a smooth solution to the shifted wave equation
\[
 \partial_t^2 u_\eps - (\Delta_{\Hm^n}+ \rho^2) u_\eps = F_{\eps}
\]
in the time interval $[-T_1,T_2]$. Combining Lemma \ref{CH1}, Lemma \ref{smoothing}, the fact $u \in L^p L^{2p} ([-T_1-\eps_0, T_2 + \eps_0])$ and the inequality
\begin{align*}
 \|F(u_\eps) - F_\eps\|_{L^1 L^2} \leq & \|F(u_\eps) - F(u)\|_{L^1 L^2} + \|F(u) - F_\eps\|_{L^1 L^2}\\
 \leq & C_p \|u_\eps - u\|_{L^p L^{2p}} \left(\|u_\eps\|_{L^p L^{2p}}^{p-1} + \|u\|_{L^p L^{2p}}^{p-1}\right) + \|F(u) - F_\eps\|_{L^1 L^2},
\end{align*}
we immediately have the following lemma.
\begin{lemma} \label{convergenceueps} Let $u$ be a solution as in Theorem \ref{Morawetz1} and $u_\epsilon, F_\epsilon$ as above.  Then for any
$t_0$   in $[-T_1,T_2]$
\begin{align*}
  &\lim_{\eps \rightarrow 0} \|F(u_\eps) - F_\eps \|_{L^1 L^2 ([-T_1,T_2]\times \Hm^n)} = 0;\\
  &\lim_{\eps \rightarrow 0} \|(u_\eps (t_0), \partial_t u_\eps (t_0))- (u (t_0),\partial_t u(t_0))\|_{H^1 \times L^2 (\Hm^n)} =0;\\
  &\lim_{\eps \rightarrow 0} \|u_\eps (t_0)- u (t_0)\|_{L^{p+1} (\Hm^n)} =0;\\
  &M_1 := \sup_{\eps < \eps_0} \|(u_\eps, \partial_t u_\eps)\|_{C([-T_1,T_2]; H^1\times L^2 (\Hm^n))} < \infty.
\end{align*}
\end{lemma}
Note that the third limit is a combination of the Sobolev embedding $H^1 \hookrightarrow L^{p+1}$ and of  the second limit.

\subsection{Energy Conservation Law}
\begin{proposition} [Invariance of Energy] Let $u$ be a solution as in Theorem \ref{Morawetz1}. Then we have
\[
 \EE(t) = \frac{1}{2}\|Du(\cdot,t)\|_{L^2(\Hm^n)}^2 + \frac{1}{2}\|\partial_t u(\cdot,t)\|_{L^2(\Hm^n)}^2 + \frac{1}{p+1}\|u(\cdot,t)\|_{L^{p+1}(\Hm^n)}^{p+1}
\]
is a (finite) constant for each $t \in (-T_-, T_+)$.
\end{proposition}
\begin{proof}
 By Lemma \ref{CH1} and Sobolev embedding, we have known $\EE(t) < \infty$ for each $t \in (-T_-,T_+)$. Let us define
\[
 E_{\eps, \delta}(t) = \int_{\Hm^n} \left(\frac{1}{2} |\nabla u_\eps|^2 - \frac{\rho^2}{2} |u_\eps|^2 + \frac{1}{2} |\partial_t u_\eps|^2 + \frac{1}{p+1} |u_\eps|^{p+1}\right)\psi_\delta d\mu
\]
for each $t\in [-T_1,T_2]$. Differentiation in $t$ gives
\begin{align*}
 E_{\eps,\delta}'(t) = & \int_{\Hm^n} \left(\Db_\alpha u_\eps \Db^\alpha (\partial_t u_\eps) - \rho^2 u_\eps \partial_t u_\eps + \partial_t u_\eps \partial_t^2 u_\eps - F(u_\eps) \partial_t u_\eps\right)\psi_\delta d\mu\\
 = & \int_{\Hm^n}\ \left(-(\partial_t u_\eps) \Db^\alpha \Db_\alpha u_\eps  - \rho^2 u_\eps \partial_t u_\eps + \partial_t u_\eps \partial_t^2 u_\eps - F(u_\eps) \partial_t u_\eps\right)\psi_\delta d\mu\\
  & + \int_{\Hm^n} (-\Db^\alpha \psi_{\delta} \Db_\alpha u_\eps)  \partial_t u_\eps d\mu\\
 = & \int_{\Hm^n} \left[\partial_t^2 u_\eps - (\Delta_{\Hm^n}+ \rho^2) u_\eps - F_{\eps} \right](\partial_t u_\eps) \psi_\delta d\mu  + \int_{\Hm^n} \left[F_\eps - F(u_\eps) \right](\partial_t u_\eps) \psi_\delta d\mu\\
 & - \int_{\Hm^n} (\Db^\alpha \psi_{\delta} \Db_\alpha u_\eps)  \partial_t u_\eps d\mu\\
 = & \int_{\Hm^n} \left[F_\eps - F(u_\eps) \right](\partial_t u_\eps) \psi_\delta d\mu
 - \int_{\Hm^n} (\Db^\alpha \psi_{\delta} \Db_\alpha u_\eps)  \partial_t u_\eps d\mu.
\end{align*}
Without loss of generality let us now assume that $t_0 \in [0,T_2]$. By integrating we obtain
\begin{align*}
 \left|E_{\eps,\delta}(t_0) - E_{\eps,\delta}(0)\right| \leq & \int_0^{t_0} \int_{\Hm^n} \left|F_\eps - F(u_\eps) \right||\partial_t u_\eps| \psi_\delta d\mu dt
+\int_0^{t_0} \int_{\Hm^n} \left|(\Db^\alpha \psi_{\delta} \Db_\alpha u_\eps)  \partial_t u_\eps\right| d\mu dt\\
\lesssim & \|F_\eps - F(u_\eps)\|_{L^1 L^2([0,t_0] \times \Hm^n)} \|\partial_t u_\eps\|_{L^\infty L^2([0,t_0] \times \Hm^n)}\\
& \qquad + \delta t_0 \|u_\eps\|_{L^\infty ([0,t_0]; H^1(\Hm^n))} \|\partial_t u_\eps\|_{L^\infty L^2 ([0,t_0] \times \Hm^n)}\\
\leq & M_1 \|F_\eps - F(u_\eps)\|_{L^1 L^2} + \delta t_0 M_1^2.
\end{align*}
Here we use the universal bound found in Lemma \ref{convergenceueps}. Sending both $\delta$ and $\eps$ to zero, we obtain $E(u(\cdot, t_0),\partial_t u(\cdot, t_0)) = E(u_0,u_1)$. This finishes the proof.
\end{proof}
\begin{remark} \label{invariance of energy focus}
 The energy of a solution to \eqref{(CP1)}  in the focusing case is also a constant under the same assumptions, because the defocusing assumption has not been used in the argument above.
\end{remark}
\subsection{Proof of Theorem \ref{Morawetz1}}\label{T1proof}
We start by defining
\[
 M(t) = -\int_{\Hm^n}  \partial_t u_\eps (x,t) \left(\Db^\alpha a \Db_\alpha u_\eps (x,t) + u_\eps(x,t) \cdot \frac{\Delta a}{2} \right)\psi_\delta d\mu(x)
\]
for any $t\in [-T_1,T_2]$, where $u_\eps$ and $F_\eps$ are as in \eqref{newequation}, the function $a$ is the smooth function given in Lemma \ref{def of a} and  the function $\psi_{\delta}$ is the smooth cut-off function introduced in Definition \ref{smoothcutoff}. By differentiating in $t$ we obtain
\begin{align*}
 M'(t) = & -\int_{\Hm^n}\!  \partial_t^2 u_\eps \left(\Db^\alpha a \Db_\alpha u_\eps + u_\eps \cdot \frac{\Delta a}{2} \right)\psi_\delta d\mu - \!\int_{\Hm^n}\!  \partial_t u_\eps \left(\Db^\alpha a \Db_\alpha (\partial_t u_\eps) + \partial_t u_\eps \cdot \frac{\Delta a}{2} \right)\psi_\delta d\mu\\
 = & -\int_{\Hm^n} \left((\Delta_{\Hm^n}+ \rho^2) u_\eps + F_\eps \right) \left(\Db^\alpha a \Db_\alpha u_\eps + u_\eps \cdot \frac{\Delta a}{2} \right)\psi_\delta d\mu\\
 & - \int_{\Hm^n}  \left(\frac{1}{2}\Db^\alpha a \Db_\alpha (\partial_t u_\eps)^2 + (\partial_t u_\eps)^2 \cdot \frac{\Delta a}{2} \right)\psi_\delta d\mu\\
 = & - \int_{\Hm^n} \Delta_{\Hm^n} u_\eps \left(\Db^\alpha a \Db_\alpha u_\eps + u_\eps \cdot \frac{\Delta a}{2} \right)\psi_\delta d\mu - \int_{\Hm^n} \rho^2 u_\eps \left(\Db^\alpha a \Db_\alpha u_\eps + u_\eps \cdot \frac{\Delta a}{2} \right)\psi_\delta d\mu\\
 & - \int_{\Hm^n} F(u_\eps) \left(\Db^\alpha a \Db_\alpha u_\eps + u_\eps \cdot \frac{\Delta a}{2} \right)\psi_\delta d\mu + \frac{1}{2} \int_{\Hm^n} (\partial_t u_\eps)^2 \Db_\alpha \psi_\delta \Db^{\alpha} a d\mu\\
 & - \int_{\Hm^n} \left[F_\eps - F(u_\eps)\right] \left(\Db^\alpha a \Db_\alpha u_\eps + u_\eps \cdot \frac{\Delta a}{2} \right)\psi_\delta d\mu\\
 = & I_1 + I_2 + I_3 + I_4 + I_5.
\end{align*}
Let us keep in mind that $\Delta a =1$ and integrate by parts in $I_1$:
\begin{align*}
 I_1  =& - \int_{\Hm^n} \Db^\beta \Db_\beta u_\eps \left(\Db^\alpha a \Db_\alpha u_\eps + u_\eps \cdot \frac{\Delta a}{2} \right)\psi_\delta d\mu\\
 = & \int_{\Hm^n} \left(\Db_\beta u_\eps \Db^\beta \Db^\alpha a \Db_\alpha u_\eps + \Db_\beta u_\eps \Db^\alpha a \Db^\beta  \Db_\alpha u_\eps + \Db_\beta u_\eps \Db^\beta u_\eps \frac{\Delta a}{2}\right) \psi_\delta d\mu\\
  & + \int_{\Hm^n} \Db^\beta \psi_\delta \Db_\beta u_\eps \left(\Db^\alpha a \Db_\alpha u_\eps + u_\eps \cdot \frac{\Delta a}{2} \right) d\mu\\
  \geq & \int_{\Hm^n} \left( \frac{1}{2}\Db^\alpha a \Db_\alpha (\Db_\beta u_\eps \Db^\beta u_\eps) + \Db_\beta u_\eps \Db^\beta u_\eps \frac{\Delta a}{2}\right) \psi_\delta d\mu - C \delta \|u_\eps\|_{H^1 (\Hm^n)}^2\\
  = & - \frac{1}{2} \int_{\Hm^n} (\Db_\alpha \psi_\delta \Db^\alpha a )(\Db_\beta u_\eps \Db^\beta u_\eps) d\mu - C \delta \|u_\eps\|_{H^1 (\Hm^n)}^2\\
  \geq & - C \delta \|u_\eps\|_{H^1 (\Hm^n)}^2 \geq - C \delta M_1^2.
\end{align*}
The letter $C$ may represent different constants in each step throughout the proof. Similarly we have
\begin{align*}
 I_2 = & - \rho^2 \int_{\Hm^n} u_\eps \left(\Db^\alpha a \Db_\alpha u_\eps + u_\eps \cdot \frac{\Delta a}{2} \right)\psi_\delta d\mu\\
 = & - \frac{\rho^2}{2} \int_{\Hm^n} \left( \Db^\alpha a \Db_\alpha (u_\eps^2) + u_\eps^2 \cdot \Delta a \right)\psi_\delta d\mu\\
 = & \frac{\rho^2}{2} \int_{\Hm^n} u_\eps^2 \Db_\alpha \psi_\delta \Db^\alpha a d\mu\\
 \geq & -C \delta M_1^2.
\end{align*}
The third term gives
\begin{align*}
 I_3 = &\int_{\Hm^n} |u_\eps|^{p-1} u_\eps \left(\Db^\alpha a \Db_\alpha u_\eps + u_\eps \cdot \frac{\Delta a}{2} \right)\psi_\delta d\mu\\
  = & \int_{\Hm^n} \left(\frac{1}{p+1}\Db^\alpha a \Db_\alpha (|u_\eps|^{p+1}) + \frac{1}{2}\;|u_\eps|^{p+1} \cdot \Delta a \right)\psi_\delta d\mu\\
  = & \frac{p-1}{2(p+1)} \int_{\Hm^n} |u_\eps|^{p+1} \psi_\delta d\mu -\frac{1}{p+1} \int_{\Hm^n} (\Db_\alpha \psi_\delta \Db^\alpha a) |u_\eps|^{p+1} d\mu\\
  \geq & \frac{p-1}{2(p+1)} \int_{\Hm^n} |u_\eps|^{p+1} \psi_\delta d\mu - C \delta M_1^{p+1}.
\end{align*}
It is clear that $I_4 \geq - C\delta M_1^2$. Finally we have
\begin{align*}
 \int_{-T_1}^{T_2} I_5(t) dt & = - \int_{-T_1}^{T_2} \int_{\Hm^n} \left[F_\eps - F(u_\eps)\right] \left(\Db^\alpha a \Db_\alpha u_\eps + u_\eps \cdot \frac{\Delta a}{2} \right)\psi_\delta d\mu dt\\
 & \geq - C\left\|F_\eps - F(u_\eps)\right\|_{L^1 L^2 ([-T_1,T_2]\times \Hm^n)} \| |\nabla u_\eps|+|u_\eps|\|_{L^\infty L^2([-T_1,T_2] \times \Hm^n)}\\
  & \geq - C M_1 \left\|F_\eps - F(u_\eps)\right\|_{L^1 L^2 ([-T_1,T_2]\times \Hm^n)}.
\end{align*}
Collecting all terms above, we obtain
\begin{align}
 M(T_2) - M(-T_1) \geq \frac{p-1}{2(p+1)} \int_{-T_1}^{T_2} \int_{\Hm^n} |u_\eps|^{p+1} \psi_\delta d\mu dt & -C M_1 \left\|F_\eps - F(u_\eps)\right\|_{L^1 L^2 ([-T_1,T_2]\times \Hm^n)}\nonumber\\
 &-C \delta (T_2 + T_1)(M_1^{p+1}+M_1^2). \label{inequality1}
\end{align}
On the other hand, we can estimate $M(t_0)$ for any given $t_0 \in [-T_1,T_2]$ by
\begin{align*}
|M(t_0)| =& \left| \int_{\Hm^n} \partial_t u_\eps (x,t_0) \left( \Db^\alpha a \Db_\alpha u_\eps(x,t_0) + \frac{1}{2} u_\eps (x,t_0) \right) \psi_\delta d\mu(x) \right|\\
 \leq &\frac{1}{2} \int_{\Hm^n} \left( |\partial_t u_\eps|^2 + \left(\Db^\alpha a \Db_\alpha u_\eps + \frac{1}{2}u_\eps \right)^2 \right) \psi_\delta d\mu\\
 = &\frac{1}{2} \int_{\Hm^n} \left( |\partial_t u_\eps|^2 + |\Db^\alpha a \Db_\alpha u_\eps|^2 + \frac{1}{4}|u_\eps|^2 + \frac{1}{2} \Db^\alpha a \Db_\alpha (|u_\eps|^2) \right)\psi_\delta d\mu\\
 \leq &\frac{1}{2} \int_{\Hm^n} \left( |\partial_t u_\eps|^2 + |\nabla a|^2 |\nabla u_\eps|^2 + \frac{1}{4}|u_\eps|^2
  - \frac{1}{2} |u_\eps|^2 \Db_\alpha \Db^\alpha a \right)\psi_\delta d\mu\\
 &\qquad -\frac{1}{4} \int_{\Hm^n} (\Db_\alpha \psi_\delta \Db^\alpha a) u_\eps^2 d\mu\\
 \leq &\frac{1}{2} \int_{\Hm^n} \left( |\partial_t u_\eps|^2 + \frac{1}{4\rho^2}(|\nabla u_\eps|^2 - \rho^2 |u_\eps|^2) \right)\psi_\delta d\mu + C\delta M_1^2.
\end{align*}
Here we use the upper bound of $|\nabla a|$ in Remark \ref{gradient of a}. Combining this with the inequality (\ref{inequality1}) and letting $\delta \rightarrow 0$, we obtain
\begin{align*}
 \frac{p-1}{2(p+1)} \int_{-T_1}^{T_2} \int_{\Hm^n} |u_\eps|^{p+1} d\mu dt \leq & \EE_0(u_\eps(\cdot, -T_1),\partial_t u_\eps(\cdot,-T_1)) + \EE_0(u_\eps(\cdot,T_2), \partial_t u_\eps (\cdot,T_2))\\
 & + C M_1 \left\|F_\eps - F(u_\eps)\right\|_{L^1 L^2 ([-T_1,T_2]\times \Hm^n)}.
\end{align*}
Here $\EE_0 (v_0, v_1)$ is the energy for the linear shifted wave equation:
\[
 \EE_0 (v_0, v_1) := \int_{\Hm^n} \left[\frac{1}{2}(|\nabla v_0|^2 - \rho^2 |v_0|^2) + \frac{1}{2}|v_1|^2 \right] d\mu = \frac{1}{2}\|v_0\|_{H^{0,1}(\Hm^n)}^2 + \frac{1}{2}\|v_1\|_{L^2 (\Hm^n)}^2.
\]
Sending $\eps$ to zero gives
\begin{align*}
 \frac{p-1}{2(p+1)} \int_{-T_1}^{T_2} \int_{\Hm^n} |u|^{p+1} d\mu dt \leq & \EE_0(u(\cdot, -T_1),\partial_t u(\cdot,-T_1)) + \EE_0(u(\cdot,T_2), \partial_t u (\cdot,T_2))\\
 \leq & 2 \EE(u_0,u_1).
\end{align*}
Since the argument above is valid for any time interval $[-T_1,T_2]$ satisfying $-T_- < -T_1 < 0 < T_2 <T_+$, we can finish the proof by letting $T_1 \rightarrow T_-$ and $T_2 \rightarrow T_+$.

\subsection{Further Improvement on Morawetz Inequality}
We conclude this section by showing that our Morawetz inequality still holds under a weaker assumption. More precisely, we eliminate the $L^2$ assumption on $u_0$. Although it plays an important role in the process of the proof, it is actually a technical assumption instead of an essential one.
\begin{theorem} \label{Morawetz2}
Let $2 \leq n \leq 6$, $1<p<p_c$ and $(u_0,u_1)\in H^{\frac{1}{2}, \frac{1}{2}} \times H^{\frac{1}{2}, -\frac{1}{2}}(\Hm^n)$ be initial data. Assume $u$ is the solution of \eqref{(CP1)}  in the defocusing case with initial data $(u_0,u_1)$, then the energy
\[
 \EE(u,\partial_t u) = \frac{1}{2}\|u(\cdot,t)\|_{H^{0,1}(\Hm^n)}^2 + \frac{1}{2}\|\partial_t u(\cdot,t)\|_{L^2(\Hm^n)}^2 + \frac{1}{p+1}\|u(\cdot,t)\|_{L^{p+1}(\Hm^n)}^{p+1}
\]
is constant for every $t$ in the maximal lifespan $(-T_-,T_+)$. In addition, we have a Morawetz-type inequality
\begin{equation}\label{equation2}
 \int_{-T_-}^{T_+} \int_{\Hm^n} |u|^{p+1} d\mu dt \leq \frac{4(p+1)}{p-1} \EE.
\end{equation}
\end{theorem}
\begin{proof}
The idea is to use approximation techniques. By cutting off the lower frequency part of $u_0$, we can always manufacture a sequence of initial data\footnote{We may choose $u_{1,\eps} =u_1$ for each $\eps$.} $(u_{0,\eps}, u_{1,\eps})$ such that
\begin{align}
 & (u_{0,\eps}, u_{1,\eps}) \in H^1 \times H^{\frac{1}{2},-\frac{1}{2}}(\Hm^n);&
 & (u_{0,\eps}, u_{1,\eps}) \rightarrow (u_0,u_1)\; \hbox{in}\; H^{\frac{1}{2},\frac{1}{2}} \times H^{\frac{1}{2},-\frac{1}{2}}(\Hm^n).& \label{convergence01}
\end{align}
The convergence in the Sobolev spaces above also implies $(u_{0,\eps},u_{1,\eps}) \rightarrow (u_{0},u_1)$ in the space $(H^{0,1} \cap L^{p+1}) \times L^2 (\Hm^n)$ by basic Fourier analysis and Proposition \ref{Sobolevem1}. Thus we have
\begin{equation} \label{convergence0}
 \EE(u_{0,\eps},u_{1,\eps}) \rightarrow \EE(u_0,u_1).
\end{equation}
Now choose $\sigma$ sufficiently close to $1$ with a $(p,\sigma)$-compatible pair $(p_1,q_1)$ so that the local theory is available for the given exponent $p$. By Fourier analysis, the convergence (\ref{convergence01}) also implies
\[
 (u_{0,\eps}, u_{1,\eps}) \rightarrow (u_0,u_1)\; \hbox{in}\; H^{\sigma - \frac{1}{2},\frac{1}{2}} \times H^{\sigma - \frac{1}{2},-\frac{1}{2}}(\Hm^n).
\]
Given any $[-T_1,T_2] \subset (-T_-, T_+)$, by long time perturbation theory, we know the solution $u_\eps$ to \eqref{(CP1)} with initial data $(u_{0,\eps},u_{1,\eps})$ exists in the time interval $[-T_1,T_2]$ if $\eps$ is sufficiently small. Furthermore, we have
\begin{equation} \label{convergence1}
  \|u(x,t) - u_\eps (x,t)\|_{L^{p_1} L^{q_1} ([-T_1,T_2]\times \Hm^n)} \rightarrow 0.
\end{equation}
\begin{equation} \label{convergence2}
  \sup_{t \in [-T_1,T_2]} \left\|(u(t),\partial_t u(t))-(u_\eps (t), \partial_t u_\eps (t) )\right\|_{H^{\sigma-\frac{1}{2},\frac{1}{2}} \times H^{\sigma-\frac{1}{2},-\frac{1}{2}}(\Hm^n)} \rightarrow 0.
\end{equation}
Applying Theorem \ref{Morawetz1} to the solutions $u_\eps$, we obtain
\[
 \frac{1}{2}\|u_\eps (\cdot, t)\|_{H^{0,1}(\Hm^n)}^2 + \frac{1}{2}\|\partial_t u_\eps(\cdot, t)\|_{L^2(\Hm^n)}^2 + \frac{1}{p+1}\|u_\eps (\cdot,t)\|_{L^{p+1}(\Hm^n)}^{p+1} = \EE (u_{0,\eps},u_{1,\eps}).
\]
\[
 \int_{-T_1}^{T_2} \int_{\Hm^n} |u_\eps|^{p+1} d\mu dt \leq \frac{4(p+1)}{p-1} \EE(u_{0,\eps},u_{1,\eps}).
\]
Combing these with  (\ref{convergence0}), (\ref{convergence1}), (\ref{convergence2}) and letting $\eps \rightarrow 0$, we have
\[
  \frac{1}{2}\|u(\cdot,t_0)\|_{H^{0,1}(\Hm^n)}^2 + \frac{1}{2}\|\partial_t u(\cdot,t_0)\|_{L^2(\Hm^n)}^2 + \frac{1}{p+1}\|u(\cdot,t_0)\|_{L^{p+1}(\Hm^n)}^{p+1} \leq \EE(u_0,u_1);
\]
\[
 \int_{-T_1}^{T_2} \int_{\Hm^n} |u|^{p+1} d\mu dt \leq \frac{4(p+1)}{p-1} \EE(u_0,u_1).
\]
In the first inequality $t_0$ is an arbitrary time in $[-T_1,T_2]$. Sending $T_1 \rightarrow T_-$, $T_2 \rightarrow T_+$, we obtain the Morawetz inequality and a one-way energy estimate
\[
  \EE(u(\cdot,t_0), \partial_t u(\cdot,t_0)) \leq \EE(u_0,u_1)
\]
for each $t_0 \in (-T_-,T_+)$. The combination of this estimate with the fact that $(u(\cdot,t_0), \partial_t u(\cdot,t_0)) \in H^{\sigma -\frac{1}{2}, \frac{1}{2}} \times H^{\sigma- \frac{1}{2}, -\frac{1}{2}}$ implies $(u(\cdot,t_0), \partial_t u(\cdot,t_0)) \in H^{\frac{1}{2}, \frac{1}{2}} \times H^{\frac{1}{2}, -\frac{1}{2}}$. By considering the backward Cauchy problem \eqref{(CP1)}  with initial data $(u(\cdot,t_0), \partial_t u(\cdot,t_0))$, we obtain $\EE(u_0,u_1) \leq \EE(u(\cdot,t_0), \partial_t u(\cdot,t_0))$ by the same argument above. This finishes the proof of energy conservation law.
\end{proof}
\section{Scattering Results with Initial Data in the Energy Space}\label{section5}

Let $2 \leq n \leq 6$. In this section we show that the solutions to \eqref{(CP1)}  in the defocusing case scatters if
\begin{itemize}
  \item The exponent $p$  is less than the energy critical exponent $p_c = 1 + \frac{4}{n-2}$.
  \item The initial data is in the space $H^{\frac{1}{2},\frac{1}{2}} \times  H^{\frac{1}{2},-\frac{1}{2}}(\Hm^n)$.
\end{itemize}
The idea is to combine the basic local theory with the Morawetz inequality and energy conservation law obtained in the last section. We start with the simplest case: when $p$ does not exceed the conformal exponent $p_{conf} = 1 + \frac{4}{n-1}$. (When $n=2$, we require $p< p_{conf} =5$ instead.)

\subsection{Scattering for $p \in (1, p_{conf}]$}
\begin{proposition} \label{scattering01}
 Let $2 \leq n \leq 6$ and $p > 1$ satisfy
 \[
  \left\{\begin{array}{ll}
          p \leq p_{conf} = 1 + \frac{4}{n-1}, & 3 \leq n \leq 6;\\
          p < p_{conf} = 5, & n =2.
         \end{array}\right.
 \]
 Then the solution $u$ to the equation \eqref{(CP1)} in the defocusing case with initial data $(u_0,u_1) \in H^{\frac{1}{2},\frac{1}{2}}(\Hm^n) \times H^{\frac{1}{2},-\frac{1}{2}}(\Hm^n)$ exists globally in time and scatters. More precisely, there exist two pairs $(u_0^{\pm},u_1^{\pm}) \in H^{0, \frac{1}{2}} \times H^{0, -\frac{1}{2}}(\Hm^n)$, such that
 \begin{equation} \label{scattering01a}
  \lim_{t \rightarrow \pm \infty} \|(u(t),\partial_t u(t))-(S(t)(u_0^\pm,u_1^\pm), \partial_t S(t)(u_0^\pm,u_1^\pm))\|_{H^{0, \frac{1}{2}} \times H^{0, -\frac{1}{2}}(\Hm^n)} = 0.
 \end{equation}
\end{proposition}
\begin{proof}
 The Morawetz inequality proved in Theorem \ref{Morawetz2} claims $u \in L^{p+1}L^{p+1}((-T_-,T_+)\times \Hm^n)$. Since $(p+1, p+1)$ is a $(p,\frac{1}{2})$-compatible pair by our assumption, this implies global existence of the solution in time and the scattering in $H^{0,\frac{1}{2}} \times H^{0,-\frac{1}{2}}$ by part (d) and (e) of Theorem \ref{localtheory}.
\end{proof}
\begin{remark} \label{strongerh}
The conclusion of Proposition \ref{scattering01} can be upgraded to a stronger version. In fact, we have $(u_0^\pm, u_1^\pm) \in H^{\frac{1}{2}, \frac{1}{2}} \times H^{\frac{1}{2}, -\frac{1}{2}}(\Hm^n)$ and
\[
 \lim_{t \rightarrow \pm \infty} \left\|(u(t),\partial_t u(t))-(S(t)(u_0^\pm,u_1^\pm), \partial_t S(t)(u_0^\pm,u_1^\pm))\right\|_{H^{\sigma-\frac{1}{2}, \frac{1}{2}} \times H^{\sigma-\frac{1}{2}, -\frac{1}{2}}(\Hm^n)} = 0
\]
holds for any $\sigma \in [1/2, 1)$.
\end{remark}
\begin{proof}
First of all, applying backward linear propagation on (\ref{scattering01a}), we obtain
\[
 \lim_{t \rightarrow \pm \infty} \left\|S(-t)\left(\begin{array}{c} u(t)\\ \partial_t u(t)\end{array}\right) - \left(\begin{array}{c} u_0^\pm\\ u_1^\pm\end{array}\right) \right\|_{H^{0, \frac{1}{2}} \times H^{0, -\frac{1}{2}}(\Hm^n)} = 0.
\]
Here the linear propagation operator $S(t)$ is defined in the same manner as in Theorem \ref{localtheory}. Combining this convergence with the fact
\[
 \left\|S(-t)\left(\begin{array}{c} u(t)\\ \partial_t u(t)\end{array}\right)\right\|_{H^{0,1}\times L^2} =
 \left\|\left(\begin{array}{c} u(t)\\ \partial_t u(t)\end{array}\right)\right\|_{H^{0,1}\times L^2} \lesssim \left(\EE(u_0,u_1)\right)^{1/2} < \infty, \;\; \forall\, t \in \Rm,
\]
we obtain that $(u_0^{\pm}, u_1^\pm) \in H^{0,1} \times L^2 (\Hm^n)$. This implies that $(u_0^\pm, u_1^\pm) \in H^{\frac{1}{2}, \frac{1}{2}} \times H^{\frac{1}{2}, -\frac{1}{2}}(\Hm^n)$ since
\[
   \left(H^{0, \frac{1}{2}} \cap H^{0,1}\right)\times \left(H^{0, -\frac{1}{2}} \cap L^2\right)(\Hm^n) \hookrightarrow H^{\frac{1}{2}, \frac{1}{2}} \times H^{\frac{1}{2}, -\frac{1}{2}}(\Hm^n).
\]
Moreover, we obtain
\[
 \left\|\left(\begin{array}{c} u(t)\\ \partial_t u(t)\end{array}\right) - S(t) \left(\begin{array}{c} u_0^\pm\\ u_1^\pm\end{array}\right) \right\|_{H^{0,1} \times L^2(\Hm^n)} \lesssim \left(\EE(u_0,u_1)\right)^{1/2} + \|(u_0^{\pm}, u_1^\pm)\|_{H^{0,1} \times L^2 (\Hm^n)} < \infty
\]
for each $t \in \Rm$. Using this estimate on the high frequency part and the convergence (\ref{scattering01a}) on the low frequency part, we finish the proof.
\end{proof}
\subsection{Scattering Results for $p\in [5,\infty), n=2$}
In this section, we show the following scattering result.
\begin{proposition} \label{scatteringn2pgeq5}
 Let $n=2$, $p\in [5,\infty)$. Then a solution $u$ to the equation \eqref{(CP1)} in the defocusing case with initial data $(u_0,u_1) \in H^{\frac{1}{2},\frac{1}{2}}(\Hm^2) \times H^{\frac{1}{2},-\frac{1}{2}}(\Hm^2)$ exists globally in time and scatters. More precisely, there exist two pairs $(u_0^{\pm},u_1^{\pm}) \in H^{\frac{1}{2}, \frac{1}{2}} \times H^{\frac{1}{2}, -\frac{1}{2}}(\Hm^2)$, such that
 \[
  \lim_{t \rightarrow \pm \infty} \|(u(t),\partial_t u(t))-(S(t)(u_0^\pm,u_1^\pm), \partial_t S(t)(u_0^\pm,u_1^\pm))\|_{H^{\sigma-\frac{1}{2}, \frac{1}{2}} \times H^{\sigma-\frac{1}{2}, -\frac{1}{2}}(\Hm^2)} = 0
 \]
 for any $\sigma \in (1-\frac{2}{p-1},1)$.
\end{proposition}
\begin{proof}
 Fix $\sigma \in (1-\frac{2}{p-1},1)$. Let us choose a $(p,\sigma)$-compatible pair that is also open $\sigma$-admissible as below, (see our discussion in Proposition \ref{localpgeq5n2} and Remark \ref{perturbationallowance} for the choice of $\eps$ and other details):
\[
 (p_1,q_1) = \left(\frac{3p}{2 + \sigma - 3\eps}, \frac{6p}{7- 4\sigma}\right).
\]
For a small positive number $\kappa$, we can define a pair $(\tilde{p}_1,\tilde{q}_1)$ by
\[
 \left(\frac{1}{p_1},\frac{1}{q_1}\right) = (1-\kappa)\left(\frac{1}{\tilde{p}_1},\frac{1}{\tilde{q}_1}\right) + \kappa \left(\frac{1}{p+1},\frac{1}{p+1}\right).
\]
This  means we have the following inequality for any time interval $I = [a,b) \subseteq (-T_-,T_+)$
\begin{equation}\label{interpolation1}
 \|u\|_{L^{p_1} L^{q_1}(I \times \Hm^2)} \leq \|u\|_{L^{p+1} L^{p+1}(I \times \Hm^2)}^\kappa \cdot
 \|u\|_{L^{\tilde{p}_1} L^{\tilde{q}_1}(I \times \Hm^2)}^{1-\kappa}.
\end{equation}
According to Remark \ref{openadmissible}, the pair $(\tilde{p}_1, \tilde{q}_1)$ is still $\sigma$-admissible if we choose a sufficiently small constant $\kappa$. As a result, we have a Strichartz estimate
\begin{equation}
\|u\|_{L^{\tilde{p}_1}L^{\tilde{q}_1}(I \times \Hm^2)} \leq C \left( \|(u(a), \partial_t u(a)) \|_{H^{\sigma-\frac{1}{2}, \frac{1}{2}}\times H^{\sigma-\frac{1}{2},-\frac{1}{2}}} + \|F(u)\|_{L^{p_1/p} L^{q_1/p}(I \times \Hm^2)} \right). \label{Strichartztilde}
\end{equation}
Combining this with (\ref{interpolation1}), we obtain
\begin{equation}
 \|u\|_{L^{\tilde{p}_1}L^{\tilde{q}_1}(I \times \Hm^2)} \leq  C \|(u(a), \partial_t u(a)) \|_{H^{\sigma-\frac{1}{2}, \frac{1}{2}}\times H^{\sigma-\frac{1}{2},-\frac{1}{2}}}
 + C \|u\|_{L^{p+1} L^{p+1}(I \times \Hm^2)}^{p\kappa}  \|u\|_{L^{\tilde{p}_1} L^{\tilde{q}_1}(I \times \Hm^2)}^{p(1-\kappa)}. \label{keyestimaten2}
\end{equation}
Now let us show the norm $\|(u(a), \partial_t u(a)) \|_{H^{\sigma-\frac{1}{2}, \frac{1}{2}}\times H^{\sigma-\frac{1}{2},-\frac{1}{2}}}$ above is uniformly bounded independent of the choice of $a$. We start by choosing
\begin{align*}
 & \sigma_1 = \frac{1}{4} + \frac{3}{2} \cdot \frac{1}{p+1} \in \left(\frac{1}{4}, \frac{1}{2}\right];&
 & \left(\frac{1}{p_2}, \frac{1}{q_2}\right) = \left(\frac{1}{4}, \frac{1}{p+1}\right).&
\end{align*}
One can check that the pair $(p_2,q_2)$ satisfies
\[
 \frac{1}{p_2} + \frac{2}{q_2} > \sigma_1;\qquad \frac{1}{q_2} = \frac{2\sigma_1}{3} - \frac{1}{6}; \qquad \frac{1}{p_2}, \frac{1}{q_2} \in \left(0,\frac{1}{2}\right).
\]
Therefore we are able to apply Strichartz estimates in Proposition \ref{newStri} and obtain
\begin{align}
\|(u, \partial_t u) &\|_{C((-T_-,T_+);H^{\sigma_1-\frac{1}{2}, \frac{1}{2}}\times H^{\sigma_1-\frac{1}{2},-\frac{1}{2}})}\nonumber\\
&\leq C \left( \|(u_0, u_1) \|_{H^{\sigma_1-\frac{1}{2}, \frac{1}{2}}\times H^{\sigma_1-\frac{1}{2},-\frac{1}{2}}} + \|F(u)\|_{L^{4/3} L^{(p+1)/p}((-T_-,T_+) \times \Hm^2)} \right). \label{finitelowernorm}
\end{align}
We claim the right hand side is finite. In fact
\begin{align}\notag
  \|(u_0, u_1) \|_{H^{\sigma_1-\frac{1}{2}, \frac{1}{2}}\times H^{\sigma_1-\frac{1}{2},-\frac{1}{2}}}& \leq
 \|(u_0, u_1) \|_{H^{\frac{1}{2}, \frac{1}{2}} \times H^{\frac{1}{2},-\frac{1}{2}}} < \infty;\\\label{eq3}
  \|F(u)\|_{L^{\frac{4}{3}} L^{\frac{p+1}{p}}((-T_-,T_+) \times \Hm^2)} &\leq \|F(u)\|_{L^{\frac{p+1}{p}} L^{\frac{p+1}{p}}((-T_-,T_+) \times \Hm^2)}^{\frac{3(p+1)}{4p}} \|F(u)\|_{L^{\infty} L^{\frac{p+1}{p}}((-T_-,T_+) \times \Hm^2)}^{\frac{p-3}{4p}}\\\notag
 & \leq \|u\|_{L^{p+1} L^{p+1}((-T_-,T_+) \times \Hm^2)}^{\frac{3(p+1)}{4}} \|u\|_{L^{\infty} L^{p+1}((-T_-,T_+) \times \Hm^2)}^{\frac{p-3}{4}} < \infty.
\end{align}
The last step uses the Morawetz inequality \eqref{equation2} and the energy conservation law. Thus the inequality (\ref{finitelowernorm}) implies
\begin{equation} \label{lowernorm}
 \sup_{t \in (-T_-,T_+)} \|(u(t),\partial_t u(t))\|_{H^{\sigma_1-\frac{1}{2}, \frac{1}{2}}\times H^{\sigma_1-\frac{1}{2},-\frac{1}{2}}} < \infty.
\end{equation}
On the other hand, the energy conservation law also gives the bound
\begin{equation} \label{energynorm}
 \sup_{t \in (-T_-,T_+)} \|(u(t),\partial_t u(t))\|_{H^{0,1}\times L^2 (\Hm^2)} < \infty.
\end{equation}
Using (\ref{lowernorm}) for low frequency part of $u$ and (\ref{energynorm}) for high frequency part of $u$, we obtain
\begin{equation} \label{boundfornorm}
 M := \sup_{t \in (-T_-,T_+)} \|(u(t),\partial_t u(t))\|_{H^{\sigma-\frac{1}{2}, \frac{1}{2}}\times H^{\sigma-\frac{1}{2},-\frac{1}{2}}} < \infty.
\end{equation}
Plugging this into (\ref{keyestimaten2}), we have
\begin{equation} \label{recurrence1}
 \|u\|_{L^{\tilde{p}_1}L^{\tilde{q}_1}(I \times \Hm^2)} \leq  C M
 + C \|u\|_{L^{p+1} L^{p+1}(I \times \Hm^2)}^{p\kappa}  \|u\|_{L^{\tilde{p}_1} L^{\tilde{q}_1}(I \times \Hm^2)}^{p(1-\kappa)}.
\end{equation}
Again using \eqref{equation2}, let $\eta$ be a small positive constant such that $2 CM > CM + C\eta^{p\kappa} (2CM)^{p(1-\kappa)}$ and $0 = a_0<a_1<
\cdots<a_m = T_+$ be a partition of the time interval $[0,T_+)$, such that we have
\[
 \|u\|_{L^{p+1}L^{p+1}([a_j,a_{j+1})\times \Hm^2)} < \eta
\]
for each $0 \leq j\leq m-1$. If we choose $I =[a_j,t') \subset [a_j,a_{j+1})$  we can rewrite (\ref{recurrence1}) into
\[
 \|u\|_{L^{\tilde{p}_1}L^{\tilde{q}_1}([a_j,t') \times \Hm^2)} \leq  C M
 + C \eta^{p\kappa}  \|u\|_{L^{\tilde{p}_1} L^{\tilde{q}_1}([a_j,t') \times \Hm^2)}^{p(1-\kappa)}.
\]
The norm $\|u\|_{L^{\tilde{p}_1}L^{\tilde{q}_1}([a_j,t') \times \Hm^2)}$ above is finite due to the fact that $(\tilde{p}_1,\tilde{q}_1)$ is $\sigma$-admissible, 
as long as $t' \neq T_+$. A continuity argument then immediately gives the  upper bound for each $j$
\[
 \|u\|_{L^{\tilde{p}_1}L^{\tilde{q}_1}([a_j,a_{j+1}) \times \Hm^2)} \leq 2CM.
\]
This means $\|u\|_{L^{\tilde{p}_1}L^{\tilde{q}_1}([0,T_+) \times \Hm^2)} < \infty$. We deduce from  (\ref{interpolation1}) and \eqref{equation2} that
\[
 \|u\|_{L^{p_1}L^{q_1}([0,T_+) \times \Hm^2)} < \infty.
\]
 Recalling the fact that $(p_1,q_1)$ is a $(p,\sigma)$-compatible pair, we obtain  global existence of the solution in time and  scattering in the positive time direction by part (d) and (e) of Theorem \ref{localtheory}. Since the wave equation is time-reversible, the negative time direction can  be handled in the same way. Moreover we can prove that the pairs $(u_0^\pm, u_1^\pm)$ are contained in the space $H^{\frac{1}{2}, \frac{1}{2}} \times H^{\frac{1}{2}, -\frac{1}{2}}(\Hm^2)$, by following the same argument as in the proof of Remark \ref{strongerh}. Finally we claim the pairs $(u_0^\pm,u_1^\pm)$ do not depend on the choice of $\sigma \in (1 - \frac{2}{p-1},1)$, by the embedding
\[
 H^{\sigma_2-\frac{1}{2}, \frac{1}{2}}\times H^{\sigma_2-\frac{1}{2},-\frac{1}{2}} (\Hm^2) \hookrightarrow H^{\sigma_1-\frac{1}{2}, \frac{1}{2}}\times H^{\sigma_1-\frac{1}{2},-\frac{1}{2}} (\Hm^2),
\]
if $\sigma_2 > \sigma_1$.
\end{proof}

\subsection{Scattering for $p_{conf} < p <p_c, 3 \leq n \leq 6$}

\begin{proposition} \label{scattering3n6}
 Let $3 \leq n \leq 6$, $p_{conf} < p < p_c$. Then a solution $u$ to the equation  \eqref{(CP1)} in the defocusing case with initial data $(u_0,u_1) \in H^{\frac{1}{2},\frac{1}{2}}(\Hm^n) \times H^{\frac{1}{2},-\frac{1}{2}}(\Hm^n)$ exists globally in time and scatters. More precisely, there exist two pairs $(u_0^{\pm},u_1^{\pm}) \in H^{\frac{1}{2}, \frac{1}{2}} \times H^{\frac{1}{2}, -\frac{1}{2}}(\Hm^n)$, such that
 \begin{equation}\label{scatconvergence1}
  \lim_{t \rightarrow \pm \infty} \|(u(t),\partial_t u(t))-(S(t)(u_0^\pm,u_1^\pm), \partial_t S(t)(u_0^\pm,u_1^\pm))\|_{H^{\sigma-\frac{1}{2}, \frac{1}{2}} \times H^{\sigma-\frac{1}{2}, -\frac{1}{2}}(\Hm^n)} = 0
 \end{equation}
for any $\sigma \in (\frac{n}{2}-\frac{2}{p-1},1)$.
\end{proposition}
\begin{proof}
 The proof is similar to that of Proposition \ref{scatteringn2pgeq5}. The proof consists of two main ingredients besides the energy conservation law and the Morawetz inequality.
 \begin{itemize}
   \item To show that the norm of the solution is bounded independent of time $t$ in the maximal lifespan. More precisely, for any $\sigma \in (\frac{n}{2}-\frac{2}{p-1},1)$ we have
    \begin{equation} \label{boundfornorm1}
      M := \sup_{t \in (-T_-,T_+)} \|(u(t),\partial_t u(t))\|_{H^{\sigma-\frac{1}{2}, \frac{1}{2}}\times H^{\sigma-\frac{1}{2},-\frac{1}{2}}} < \infty.
    \end{equation}
   \item To show that for any $\sigma \in (\frac{n}{2}-\frac{2}{p-1},1)$ there exists a $(p,\sigma)$-compatible pair $(p_1,q_1)$ which is also open $\sigma$-admissible. In other words, it satisfies
    \begin{align*}
     &\frac{1}{p_1} + \frac{n}{q_1} > \frac{n}{2} -\sigma;& &\frac{1}{q_1} > \frac{1}{2} - \frac{2\sigma}{n+1} ;&\\
     &\frac{1}{p_1}, \frac{1}{q_1} \in (0,\frac{1}{2}).& & &
    \end{align*}
   So we can rewrite it as an interpolation between $(p+1,p+1)$ and $(\tilde{p}_1,\tilde{q}_1)$, here the pair $(\tilde{p}_1,\tilde{q}_1)$ is still $\sigma$-admissible. 
 \end{itemize}
 \paragraph{First Ingredient}
 The combination of the Morawetz inequality and the energy conservation law yields $\|F(u)\|_{L^2 L^{(p+1)/p}((-T_-,T_+)\times \Hm^n)} < \infty$ as in \eqref{eq3}. If we choose
 \[
  \sigma_1 = \frac{n+1}{2(p+1)}-\frac{n-3}{4} \in \left(\frac{n-1}{2n}, \frac{1}{2}\right),
 \]
 then a Strichartz estimate gives the following uniform bound
 \begin{align*}
 \sup_{t \in (-T_-,T_+)}& \|(u(t),\partial_t u(t))\|_{H^{\sigma_1-\frac{1}{2}, \frac{1}{2}}\times H^{\sigma_1-\frac{1}{2},-\frac{1}{2}}}\\
 & \lesssim  \|(u_0,u_1)\|_{H^{\sigma_1-\frac{1}{2}, \frac{1}{2}}\times H^{\sigma_1-\frac{1}{2},-\frac{1}{2}}} + \|F(u)\|_{L^2 L^\frac{p+1}{p}((-T_-,T_+)\times \Hm^n)} < \infty.
 \end{align*}
 Combining this with the energy conservation law, we obtain the bound (\ref{boundfornorm1}).
\paragraph{Second Ingredient} This directly follows from Remark \ref{perturbationn35} and Remark \ref{perturbationn6}.
\end{proof}
\begin{remark} Proposition \ref{scattering3n6} actually implies that the scattering also happens in the critical Sobolev space $H^{\sigma_p-\frac{1}{2}, \frac{1}{2}}\times H^{\sigma_p-\frac{1}{2},-\frac{1}{2}} (\Hm^n)$. Namely, we have
\[
 \lim_{t \rightarrow \pm \infty} \|(u(t),\partial_t u(t))-(S(t)(u_0^\pm,u_1^\pm), \partial_t S(t)(u_0^\pm,u_1^\pm))\|_{H^{\sigma_p-\frac{1}{2}, \frac{1}{2}} \times H^{\sigma_p-\frac{1}{2}, -\frac{1}{2}}(\Hm^n)} = 0.
\]
This is simply a combination of the convergence (\ref{scatconvergence1}) and the embedding of spaces ($\sigma > \sigma_p$)
\[
 H^{\sigma-\frac{1}{2}, \frac{1}{2}}\times H^{\sigma-\frac{1}{2},-\frac{1}{2}} (\Hm^n) \hookrightarrow H^{\sigma_p-\frac{1}{2}, \frac{1}{2}}\times H^{\sigma_p-\frac{1}{2},-\frac{1}{2}} (\Hm^n).
\]
\end{remark}
\section{An Application on a Quintic Wave Equation on $\Rm^2$ }\label{application1}
In this section we consider the defocusing quintic wave equation on $\Rm^2$
\begin{equation}\label{(CP2)}
\left\{\begin{array}{lr} \partial_t^2 u - \Delta u = - |u|^4 u, \,\,\,\, (x,t)\in \Rm^2 \times \Rm; & \\
u |_{t=0} = u_0; & \qquad\qquad\\
\partial_t u |_{t=0} = u_1. & \end{array}\right.
\end{equation}
Suitable solutions satisfy the energy conservation law
\begin{equation} \label{def of energy a}
 E(u,\partial_t u) = \int_{\Rm^2} \left( \frac{1}{2} |u_t|^2 + \frac{1}{2}|\nabla u|^2 + \frac{1}{6}|u|^6 \right) dx =E(u_0,u_1).
\end{equation}
We also recall that the homogeneous space $\dot{H}^\frac{1}{2} \times \dot{H}^{-\frac{1}{2}}(\Rm^2)$ is critical for this problem. We show that if the initial data $(u_0,u_1)$ are radial, sufficiently smooth and small near infinity, then the solution to the equation \eqref{(CP2)} scatters.
Unlike the scattering results in higher dimensional spaces we mentioned in the introduction, the 2-dimensional case seems a little more difficult to deal with, since the scattering of the linear propagation is weaker in lower dimensions. As in the hyperbolic spaces, the main ingredients of our proof include
\begin{itemize}
  \item [(I)] A local well-posedness theory in a suitable space of functions, which usually depends on corresponding Strichartz estimates. This will be discussed in Subsection \ref{sec:localtheory} below.
  \item [(II)] An appropriate global space-time integral estimate. In \cite{benoit} a Morawetz-type inequality
      \begin{equation}
       \int_{-T_-}^{T_+} \int_{\Rm^n} \frac{|u|^{p+1}}{|x|} dx dt \leq C E \label{morawetzRn}
      \end{equation}
      was proved for the wave equation $\partial_t^2 u - \Delta u = -|u|^{p-1}u$ in $\Rm^n$ with $n \geq 3$.
      The global estimate was strongly used to gain a scattering theory here. Unfortunately, this does not work in the two dimensional case. Instead we 
      translate the solution to the wave equation \eqref{(CP2)} into a solution to the shifted wave equation \eqref{(CP1star)} on the hyperbolic plane and then apply the Morawetz-type inequality Theorem \ref{Morawetz2} and the scattering result Proposition \ref{scatteringn2pgeq5} there.
\end{itemize}
We start by claiming the main result of this section, which will be further improved in the last subsection.
\begin{theorem} \label{main1}
 Assume  the initial data $(u_0,u_1)$ are radial, smooth and satisfy the following inequalities
 \begin{align*}
  |\nabla u_0 (x)|, |u_1(x)| & \leq A {(|x|+1)^{-\frac{3}{2} - \eps}};\\
  |u_0(x)| & \leq A {(|x|+1)}^{-\frac{1}{2} -\eps};
 \end{align*}
 with two positive constants $A$ and $\eps$. Then the solution $u$ to the wave equation \eqref{(CP2)} with initial data $(u_0,u_1)$ exists globally in time and scatters with a finite space-time norm
 \[
  \|u\|_{L^6 L^6 (\Rm \times \Rm^2)} < \infty.
 \]
 Equivalently there exist two pairs $(u_0^\pm, u_1^\pm) \in \dot{H}^{\frac{1}{2}} \times \dot{H}^{-\frac{1}{2}} (\Rm^2)$, such that
 \[
  \lim_{t \rightarrow \pm \infty} \left\|\left(u(t) - S(t)(u_0^\pm, u_1^\pm),
 \partial_t u(t) - \partial_t S(t)(u_0^\pm, u_1^\pm)\right)\right\|_{\dot{H}^{\frac{1}{2}} \times \dot{H}^{-\frac{1}{2}}} = 0.
 \]
 Here $S(t)(u_0^\pm,u_1^\pm)$ is the solution of the linear wave equation with initial data $(u_0^\pm,u_1^\pm)$.
\end{theorem}

\begin{remark} \label{H12A}
 The assumptions on initial data $(u_0,u_1)$ in Theorem \ref{main1} immediately give
 \[
  \int_{\Rm^2} \left(|\nabla u_0|^{\frac{4}{3}} + |u_1|^{\frac{4}{3}}\right) dx \leq 2 A^{\frac{4}{3}} \int_{\Rm^2} (1+|x|)^{-2-\frac{4}{3} \eps} dx \lesssim_\eps A^{\frac{4}{3}}.
 \]
 Thus we have $(u_0,u_1) \in \dot{W}^{1,\frac{4}{3}} \times L^{\frac{4}{3}} (\Rm^2)$. Sobolev embedding implies that $(u_0,u_1)$ is in the space $\dot{H}^{\frac{1}{2}} \times \dot{H}^{-\frac{1}{2}} (\Rm^2)$, which is the critical Sobolev space for this problem. Furthermore, we can show that this pair of initial data comes with a finite energy
 \begin{align*}
 E(u_0,u_1) &\lesssim \int_{\Rm^2} \left[|\nabla u_0|^2 + |u_1|^2 + |u_0|^6 \right]dx\\
 &\lesssim \int_{\Rm^2} \left[A^2 (|x|+1)^{-3-2\eps} + A^6 (|x|+1)^{-3-6\eps} \right]dx\\
 &\lesssim A^2 + A^6.
\end{align*}
\end{remark}

\subsection{Local theory in $\dot{H}^{\frac{1}{2}} \times \dot{H}^{-\frac{1}{2}}(\Rm^2)$} \label{sec:localtheory}

The basic tool to develop a local theory is the following Strichartz estimate from Proposition 3.1 of \cite{strichartz}.

\begin{lemma} [A Strichartz estimate] 
\label{strichartz}
Let $u$ be the solution to the following linear wave equation in a time interval $I$ containing $0$
\[
 \left\{\begin{array}{l} \partial_t^2 u - \Delta u = F(x,t), \,\,\,\, (x,t)\in \Rm^2 \times I;\\
u |_{t=0} = u_0; \\
\partial_t u |_{t=0} = u_1.\end{array}\right.
\]
Then we have the following space-time norm estimate
\begin{align*}
 \|u\|_{L^6 L^6 (I \times \Rm^2)} + &\|(u,\partial_t u)\|_{C(I;\dot{H}^{\frac{1}{2}} \times \dot{H}^{-\frac{1}{2}}(\Rm^2))}\\
 & \leq C \left[\|(u_0,u_1)\|_{\dot{H}^{\frac{1}{2}} \times \dot{H}^{-\frac{1}{2}}(\Rm^2)} + \|F\|_{L^{\frac{6}{5}} L^{\frac{6}{5}} (I \times \Rm^2)}\right].
\end{align*}
The constant $C$ does not depend on the time interval $I$.
\end{lemma}

\noindent Before moving to the proof of well-posedness we introduce the definition of a solution in our context. The notation $F(u)$ will be used to represent the non-linear term $-|u|^4 u$ in this whole section.
\begin{definition} [Solutions]
We say $u(t)$ is a solution of the Cauchy problem \eqref{(CP2)} in the time interval $I$,
if $(u(t),\partial_t u(t)) \in C(I;{\dot{H}^{\frac{1}{2}}}\times{\dot{H}^{-\frac{1}{2}}}(\Rm^2))$, with a finite norm $\|u\|_{L^6 L^6(J \times \Rm^2)}$ for any bounded closed interval $J \subseteq I$ so that the integral equation
\[
 u(t) = S(t)(u_0,u_1) + \int_0^t \frac{\sin ((t-\tau)\sqrt{-\Delta})}{\sqrt{-\Delta}} F(u(\tau)) d\tau
\]
holds for all time $t \in I$.
\end{definition}
\noindent By the Strichartz estimates and a fixed-point argument, we have the following results. See, for instance, \cite{bahouri, locad1, kenig, kenig1, ls, local1, ss2} for more details.
\begin{lemma} \label{local1}
There exists a constant $\delta > 0$ such that if the initial data $(u_0,u_1) \in \dot{H}^{1/2} \times \dot{H}^{-1/2}(\Rm^2)$ satisfy
\[
 \|S(t)(u_0,u_1)\|_{L^6 L^6 (I \times \Rm^2)} < \delta
\]
for a time interval $I$ containing $0$, then there exists a solution to \eqref{(CP2)} in the time interval $I$ with the given initial data $(u_0,u_1)$.
\end{lemma}
This is a typical statement for a well-posedness result when the data belong to a critical space.
\begin{theorem} [Local solution]
For any initial data $(u_0,u_1) \in \dot{H}^{\frac{1}{2}} \times \dot{H}^{-\frac{1}{2}}(\Rm^2)$, there is a maximal interval $(-T_{-}(u_0,u_1), T_{+}(u_0,u_1))$ in which the equation has a solution.
\end{theorem}
\begin{theorem} [Scattering with small data]
There exists a constant $\delta_1 > 0$ such that if the norm of the initial data
$\|(u_0,u_1)\|_{\dot{H}^{\frac{1}{2}} \times \dot{H}^{-\frac{1}{2}}(\Rm^2)} < \delta_1$, then the Cauchy problem \eqref{(CP2)} has a global-in-time solution $u$ with $\|u\|_{L^6 L^6 (\Rm \times \Rm^2)} < \infty$.
\end{theorem}
\begin{lemma} [Standard finite time blow-up criterion] \label{finite time criterion}
If $T_{+} < \infty$, 
then $\|u\|_{L^6 L^6([0,T_{+})\times \Rm^2)} = \infty$.
\end{lemma}
\begin{proposition} [Long-time perturbation theory] \label{perturbation theory} (see also \cite{kenig1, shen2, pertao}) Let $M$ be a positive constant. There exists a constant $\eps_0 = \eps_0 (M)>0$, such that if $\eps < \eps_0$, then for any approximation solution $\tilde{u}$ defined on $\Rm^2 \times I$ ($0\in I$)
and any initial data $(u_0,u_1) \in \dot{H}^{1/2} \times \dot{H}^{-1/2}(\Rm^2)$ satisfying
\[
 (\partial_t^2 - \Delta) (\tilde{u}) - F(\tilde{u}) = e(x,t), \,\,\,\,\, (x,t) \in \Rm^2 \times I;
\]
\[
 \|\tilde{u}\|_{L^6 L^6 (I \times \Rm^2)} < M;\qquad \|(\tilde{u}(0),\partial_t\tilde{u}(0))\|_{\dot{H}^{1/2}\times \dot{H}^{-1/2}(\Rm^2)}< \infty;
\]
\[
 \|e(x,t)\|_{L^{6/5} L^{6/5}(I \times \Rm^2)}+ \|S(t)(u_0-\tilde{u}(0),u_1 - \partial_t \tilde{u}(0))\|_{L^6 L^6 (I\times \Rm^2)} \leq \eps;
\]
there exists a solution $u(x,t)$ of \eqref{(CP2)} defined in the interval $I$ with the initial data $(u_0,u_1)$ and satisfying
\[
 \|u(x,t) - \tilde{u}(x,t)\|_{L^6 L^6(I\times \Rm^2)} < C(M) \eps.
\]
\end{proposition}

\subsection{Preliminary Estimates on Solution $u$} \label{sec:pointwisees}

The first step of the proof for the main Theorem \ref{main1} is to show that for any given $t>0$ the solution $u(x,t)$ and its derivatives decay at a certain rate when $|x| \rightarrow \infty$ as its initial data does. The first main tool is the following estimate on linear solutions.
\begin{lemma} \label{lm1} Let $u$ be the solution to the following linear wave equation in a time interval $[0,T]$
\[
 \left\{\begin{array}{l} \partial_t^2 u - \Delta u = F(x,t), \,\,\,\, (x,t)\in \Rm^2 \times [0,T];\\
u |_{t=0} = u_0; \\
\partial_t u |_{t=0} = u_1.\end{array}\right.
\]
In addition, we assume that for $R, A, B> 0$ and  $0 < \alpha, \beta < 1/2$,
\begin{align*}
 |u_0(x)| & \leq A |x|^{-1/2-\alpha}, & & \hbox{if}\; |x| > R;\\
 |\nabla u_0 (x)| &\leq A |x|^{-3/2 - \alpha}, & &\hbox{if}\; |x| > R;\\
 |u_1(x)| &\leq A |x|^{-3/2 - \alpha}, & &\hbox{if}\; |x| > R;\\
 |F(x,t)| &\leq B |x|^{-5/2} (|x|-t)^{-\beta}, & &\hbox{if}\; |x|> R + t.
\end{align*}
Then there exists a constant $C = C (\alpha, \beta)\geq 1$ such that the solution $u$ satisfies
\[
 |u(x,t)| \leq C |x|^{-1/2}\left[ A (|x|-t)^{-\alpha} +  B (|x|-t)^{-\beta} \right], \; \hbox{if}\; t\in [0,T]\;\hbox{and}\; |x| > R + t.
\]
\end{lemma}

\medskip
In order to deal with the Poisson's kernel involved in the proof of this lemma, we need to introduce a few technical lemmata first.
\begin{lemma}\label{sphere}
 Let $|x| > r > 0$. Then we have
\[
 \int_{|y-x|=r} |y|^{-\kappa} dS(y) \leq \left\{ \begin{array}{ll}
 \displaystyle C(\kappa) \min \left\{\frac{1}{(|x|-r)^{\kappa-1}}, \frac{r}{(|x|-r)^{\kappa}} \right\}, & \hbox{if}\;\;\kappa >1;\\
 \displaystyle C(\kappa) \min \left\{|x|^{1-\kappa}, \frac{r}{(|x|-r)^{\kappa}} \right\}, &
 \hbox{if}\;\; 0 < \kappa <1. \end{array} \right.
\]
\end{lemma}
\begin{proof} Since $|y| \geq |x| -r$, we have
\begin{equation} \label{intes1}
 \int_{|y-x|=r} |y|^{-\kappa} dS(y) \leq \int_{|y-x|=r} (|x|-r)^{-\kappa} dS(y) = \frac{2\pi r}{(|x|-r)^{\kappa}}.
\end{equation}
\begin{figure}[h]
\centering
\includegraphics[scale=.80]{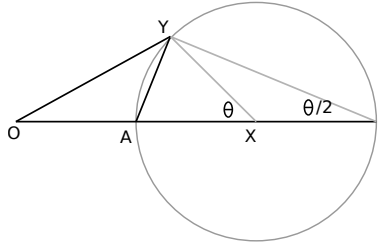}
\caption{Estimate of $|y|$}\label{drawing04}
\end{figure}
On the other hand, Figure \ref{drawing04} shows
\[
 |y|= |OY| \geq \sqrt{|OA|^2 + |AY|^2} = \sqrt{(|x|-r)^2 + 4 r^2 \sin^2 (\theta/2)} \geq
 \sqrt{(|x|-r)^2 + \frac{4r^2 \theta^2}{\pi^2}}
\]
Thus we have
\begin{equation} \label{intlin1}
 \int_{|y-x|=r} |y|^{-\kappa} dS(y)
  \leq 2 \int_0^\pi \left((|x|-r)^2 + \frac{4 r^2 \theta^2}{\pi^2}\right)^{-\kappa/2} r d\theta.
\end{equation}
If $\kappa > 1$, applying change of variable on the right hand of the inequality above, we obtain
\begin{align}
 \int_{|y-x|=r} |y|^{-\kappa} dS(y)
 & \leq \frac{\pi}{(|x|-r)^{\kappa-1}} \int_0^\pi \left(1 + \frac{4 r^2 \theta^2}{\pi^2 (|x|-r)^2}\right)^{-\kappa/2} \frac{2r}{\pi (|x|-r)} d\theta \nonumber\\
 & \leq \frac{\pi}{(|x|-r)^{\kappa-1}} \int_{0}^\infty (1+\tau^2)^{-\kappa/2} d \tau \nonumber\\
 & \leq \frac{C(\kappa)}{(|x|-r)^{\kappa -1}}. \label{intes2}
\end{align}
On the other hand, if $0 < \kappa < 1$, the inequality \eqref{intlin1} yields
\begin{equation} \label{intes3}
 \int_{|y-x|=r} |y|^{-\kappa} dS(y) \leq 2 \int_0^\pi \left(\frac{4 r^2 \theta^2}{\pi^2}\right)^{-\kappa/2} r d\theta
 \leq 2^{1-\kappa} \pi^{\kappa} r^{1-\kappa} \int_0^\pi \theta^{-\kappa} d \theta
 \leq C(\kappa) |x|^{1-\kappa}.
\end{equation}
Combining the inequalities \eqref{intes1}, \eqref{intes2} and \eqref{intes3}, we finish the proof.
\end{proof}

\begin{lemma} \label{lm03}
Assume that the constants $\kappa_1, \kappa_2, r_1, r_2 > 0$ satisfy the conditions
\[
 r_1 < r_2;\;\;\;\;\; \kappa_1 < 1;\;\;\;\;\; \kappa_1 + \kappa_2 > 1.
\]
Then we have
\[
 \int_0^{r_1} (r_1 -r)^{-\kappa_1}(r_2-r)^{-\kappa_2} dr \leq C (r_2-r_1)^{1-\kappa_1-\kappa_2}.
\]
Here the constant $C$ can be chosen as
\[
 C = \frac{1}{1 -\kappa_1} +  \frac{1}{\kappa_1 +\kappa_2 -1}.
\]
\end{lemma}
More generally, we have
\begin{lemma} \label{lm2} Assume that the constants $\kappa_1, \kappa_2, \kappa_3, r_1, r_2, r_3 > 0$ satisfy the conditions
\[
  r_1 < r_2 \leq r_3;\;\;\; \kappa_1 + \kappa_2 < 1;\;\;\; \kappa_1 + \kappa_2 + \kappa_3 > 1.
\]
Then
\[
 \int_0^{r_1} (r_1 -r)^{-\kappa_1}(r_2-r)^{-\kappa_2}(r_3-r)^{-\kappa_3} dr \leq C (r_3-r_1)^{1-\kappa_1-\kappa_2-\kappa_3}.
\]
Here the constant $C$ can be chosen as
\[
 C = \frac{1}{1-\kappa_1-\kappa_2} +  \frac{1}{\kappa_1 +\kappa_2 +\kappa_3 -1}.
\]
\end{lemma}
\begin{proof}[\textbf{Proof of lemma \ref{lm1}}] Since the conclusion holds automatically at $t=0$ if the constant $C \geq 1$, let us assume $t>0$ and $|x| > t +R$. First of all, we can write $u(x,t)$ explicitly in terms of $u_0$, $u_1$ and $F(x,t)$ as
\begin{align*}
 u(x,t) = & \frac{1}{2 \pi t^2} \int_{B(x,t)} \frac{t u_0(y) + t^2 u_1 (y)+  t \nabla u_0 (y) \cdot (y-x)}{[t^2 -|y-x|^2]^{1/2}} dy\\
  & \qquad + \frac{1}{2 \pi} \int_0^t \int_{B(x,t-s)} \frac{F(y,s)}{\left[(t-s)^2 -|y-x|^2\right]^{1/2}} dy ds.
\end{align*}
As a result, we obtain
\begin{align*}
 |u(x,t)| \leq & \frac{1}{2 \pi t} \int_{B(x,t)} \frac{|u_0(y)|}{[t^2 -|y-x|^2]^{1/2}} dy\;\;\; +
 \frac{1}{2\pi} \int_{B(x,t)} \frac{|u_1(y)| + |\nabla u_0 (y)|}{[t^2 -|y-x|^2]^{1/2}} dy\\
 & + \frac{1}{2 \pi} \int_0^t \int_{B(x,t-s)} \frac{|F(y,s)|}{\left[(t-s)^2 -|y-x|^2\right]^{1/2}} dy ds\\
 = & I_1+ I_2 + I_3.
\end{align*}
Let us start with $I_3$:
\begin{align*}
 I_3  = &\frac{1}{2 \pi} \int_0^t \int_{B(x,t-s)} \frac{B |y|^{-5/2} (|y|-s)^{-\beta}}{\left[(t-s)^2 -|y-x|^2\right]^{1/2}} dy ds\\
  = & \frac{1}{2 \pi} \int_0^t \int_{0}^{t-s} \int_{|y-x|=r} \frac{B |y|^{-5/2} (|y|-s)^{-\beta}}{\left[(t-s)^2 -r^2\right]^{1/2}} dS(y) dr ds\\
 \leq &\frac{1}{2 \pi} \int_0^t \int_{0}^{t-s} \left[ \frac{B (|x|-r-s)^{-\beta}}{\left[(t-s)^2 -r^2\right]^{1/2}} \int_{|y-x|=r} |y|^{-5/2} dS(y)\right] dr ds\\
 \lesssim & \int_0^t \int_{0}^{t-s} \left[ \frac{B (|x|-r-s)^{-\beta}}{\left[(t-s)^2 -r^2\right]^{1/2}} \min
 \left\{\frac{1}{(|x|-r)^{3/2}}, \frac{r}{(|x|-r)^{5/2}}\right\} \right] dr ds\\
   \leq & \int_0^{\max\{t-|x|/2,0\}} \int_{0}^{t-s} \left[ \frac{B (|x|-r-s)^{-\beta}}{\left[(t-s)^2 -r^2\right]^{1/2}}\cdot \frac{1}{(|x|-r)^{3/2}} \right] dr ds\\
 & + \int_{\max\{t-|x|/2,0\}}^t \int_{0}^{t-s} \left[ \frac{B (|x|-r-s)^{-\beta}}{\left[(t-s)^2 -r^2\right]^{1/2}} \cdot \frac{r}{(|x|-r)^{5/2}} \right] dr ds\\
 = & I_{3,1} + I_{3,2},
\end{align*}
where we used Lemma \ref{sphere} with $k=\frac{5}{2}$.

The first term $I_{3,1}$ is trivial unless $t > |x|/2$. Thanks to Lemma \ref{lm2}, we obtain
\begin{align*}
  I_{3,1}& \lesssim \frac{B}{|x|^{1/2}} \int_0^{t-|x|/2} \int_{0}^{t-s} \left[ \frac{ (|x|-r-s)^{-\beta}}{\left(t-s-r\right)^{1/2}}\cdot \frac{1}{(|x|-r)^{3/2}} \right] dr ds\\
 & \leq \frac{ B}{|x|^{1/2}} \int_0^{t-|x|/2} \int_{0}^{t-s} \left[ (t-s-r)^{-1/2} (|x|-s-r)^{-\beta}  (|x|-r)^{-3/2} \right] dr ds\\
 & \lesssim_\beta \frac{B}{|x|^{1/2}} \int_0^{t-|x|/2} (|x|- t+ s)^{-1-\beta} ds\\
 & \lesssim_{\beta} B |x|^{-1/2} (|x|-t)^{-\beta}.
\end{align*}
On the other hand, we have
\begin{align*}
 I_{3,2} & \lesssim \frac{ B (|x|-t)^{-\beta}}{|x|^{5/2}} \int_{\max\{t-|x|/2,0\}}^t \int_{0}^{t-s}  \frac{r}{\left[(t-s)^2 -r^2\right]^{1/2}} dr ds\\
 & \leq \frac{ B (|x|-t)^{-\beta}}{|x|^{5/2}} \int_0^t \int_0^{t-s} \frac{r}{\left[(t-s)^2 -r^2\right]^{1/2}} dr ds\\
 & \leq \frac{ B (|x|-t)^{-\beta}}{|x|^{5/2}} \int_0^t (t-s) ds\\
 & \leq B |x|^{-1/2} (|x|-t)^{-\beta}.
\end{align*}
Combining these estimates above for $ I_{3,1}$ and  $I_{3,2}$, we obtain $I_3 \lesssim_\beta B |x|^{-1/2} (|x|-t)^{-\beta}$.
Next let us consider $I_2$.
\begin{align*}
 I_2 & = \frac{1}{2 \pi} \int_{B(x,t)} \frac{|u_1(y)| + |\nabla u_0 (y)|}{[t^2 -|y-x|^2]^{1/2}} dy \\
 & \lesssim \int_{B(x,t)} \frac{A |y|^{-3/2-\alpha}}{[t^2 -|y-x|^2]^{1/2}} dy\\
 & \leq  A \int_{0}^{t} \int_{|y-x|=r} \frac{|y|^{-3/2-\alpha}}{(t^2 -r^2)^{1/2}} dS(y) dr\\
 & \lesssim  A \int_{0}^{t} \frac{1}{(t^2 -r^2)^{1/2}} \min \left\{\frac{1}{(|x|-r)^{1/2+\alpha}}, \frac{r}{(|x|-r)^{3/2+\alpha}} \right\} dr,
\end{align*}
where again we used Lemma \ref{sphere} with $\kappa =\frac{3}{2} +\alpha$.
If $t > |x|/2$, then we have
\begin{align*}
 I_2 & \lesssim A \int_{0}^t \frac{1}{(t^2 -r^2)^{1/2}}\cdot \frac{1}{(|x|-r)^{1/2+\alpha}} dr\\
 & \lesssim A |x|^{-1/2}\int_{0}^t {(t -r)^{-1/2}} {(|x|-r)^{-1/2-\alpha}} dr\\
 & \lesssim_\alpha A  |x|^{-1/2} (|x|-t)^{-\alpha},
\end{align*}
where we applied Lemma \ref{lm03}. If $t \leq |x|/2$, then we have
\begin{align*}
 I_2 & \lesssim A \int_0^{t} \frac{1}{(t^2 -r^2)^{1/2}} \frac{r}{(|x|-r)^{3/2+\alpha}} dr\\
 & \lesssim A {|x|^{-3/2 -\alpha}}\int_0^{t} \frac{r}{(t^2 -r^2)^{1/2}} dr\\
 & \leq A {|x|^{-3/2 -\alpha}}t\\
 & \leq A |x|^{-1/2} (|x|-t)^{-\alpha}.
\end{align*}
Finally, we can estimate $I_1$ by
\begin{align*}
 I_1 & = \frac{1}{2 \pi t} \int_{B(x,t)} \frac{|u_0(y)|}{[t^2 -|y-x|^2]^{1/2}} dy\\
  & \leq \frac{1}{t} \int_{B(x,t)} \frac{A |y|^{-1/2-\alpha}}{[t^2 -|y-x|^2]^{1/2}} dy\\
  & \leq \frac{A}{t} \int_0^t \int_{|y-x|=r} \frac{|y|^{-1/2-\alpha}}{(t^2 -r^2)^{1/2}} dS(y) dr\\
  & \lesssim_\alpha \frac{A}{t} \int_0^t  \frac{1}{(t^2 -r^2)^{1/2}} \min \left\{|x|^{1/2-\alpha}, \frac{r}{(|x|-r)^{1/2+\alpha}} \right\} dr,
\end{align*}
where we use Lemma \ref{sphere} with $\kappa = \frac{1}{2} + \alpha < 1$. If $t > |x|/2$, then we have
\[
 I_1 \lesssim_\alpha \frac{A}{|x|} \int_0^t  \frac{|x|^{1/2-\alpha}}{(t^2 -r^2)^{1/2}}  dr \lesssim A |x|^{-1/2-\alpha} \leq A |x|^{-1/2} (|x|-t)^{-\alpha}.
\]
On the other hand, if $t \leq |x|/2$, we obtain
\begin{align*}
 I_1 &\lesssim_\alpha \frac{A}{t} \int_0^t  \frac{1}{(t^2 -r^2)^{1/2}}\frac{r}{(|x|-r)^{1/2+\alpha}} dr\\
  & \lesssim \frac{A}{t |x|^{1/2+\alpha}} \int_0^t  \frac{r}{(t^2 -r^2)^{1/2}} dr\\
  & = A |x|^{-1/2-\alpha} \leq A |x|^{-1/2} (|x|-t)^{-\alpha}.
\end{align*}
Combining the estimates for $I_1$, $I_2$ and $I_3$, we finish the proof.
\end{proof}

\begin{lemma}
 Let $(u_0,u_1)$ be initial data as in Theorem \ref{main1}, then the solution $u$ to the equation  \eqref{(CP2)} exists globally in time.
\end{lemma}
\begin{proof}
By contradiction we assume that $T_+ (u_0,u_1) < \infty$. By the energy conservation law, we always have
\[
 \int_{\Rm^2} |u(x,t)|^6 dx \leq  6 E(u_0,u_1)
\]
for any $t \in [0,T_+)$. This implies that
\[
 \|u\|_{L^6 L^6(\Rm^2 \times [0,T_+))} \leq (6E T_+)^{1/6} < \infty,
\]
which contradicts the finite time blow-up criterion.
\end{proof}

\begin{proposition} \label{pointwise estimate}
 Let $(u_0,u_1)$ and $A, \eps$ be the initial data and positive constants as in Theorem \ref{main1}. Fix any constant
$\delta < \min\{\eps, 1/{10}\}$. Then there exist constants $B_1= B_1 (\delta) >0$ and $R= R (\delta,\eps, A)> 1$, such that the following inequality holds
\begin{equation} \label{es1}
 |u(x,t)| \leq B_1 |x|^{-1/2} (|x|-t)^{-\delta}\;\;\; \hbox{if}\; t\geq 0\; \hbox{and}\; |x|> t + R.
\end{equation}
\end{proposition}
\begin{proof}
Let $C = C(\delta, 5\delta)$ be the constant as in the conclusion of Lemma \ref{lm1}. We can always find two small positive constants $A_1 = A_1 (\delta)$ and $B_1 = B_1 (\delta)$, such that
\[
 B_1 > C (A_1 + B_1^5).
\]
Since $\delta < \eps$, we can always find a large constant $R = R(A, \eps - \delta) > 1$, such that if $|x| > R$, then
\begin{align*}
 |u_1(x)|, |\nabla u_0 (x)| & < A_1 |x|^{-3/2-\delta};\\
 |u_0(x)| & < A_1 |x|^{-1/2-\delta}.
\end{align*}
Let us check these constants $B_1$ and $R$ work. The idea is a ``double induction'' as below. By our local theory, given any small constant $\eps_1>0$, the interval $[0,T_+)$ can be broken into
\[
 [0,T_+) = [0,t_1] \cup [t_1, t_2] \cup [t_2, t_3] \cup \cdots \cup [t_m, t_{m+1}] \cup \cdots
\]
so that the $L^6 L^6$ norm of $u$ in each sub-interval satisfies (Let $t_0 =0$)
\[
 \|u\|_{L^6 L^6 (\Rm^2 \times [t_m, t_{m+1}])} \leq \eps_1;
\]
By the Strichartz estimates, we have
\begin{align*}
 \|S(t-t_m)(u(\cdot, t_m), & \partial_t u(\cdot, t_m))\|_{L^6 L^6 (\Rm^2 \times [t_m, t_{m+1}])}\\
  &\lesssim \|u\|_{L^6 L^6 (\Rm^2 \times [t_m, t_{m+1}])} + \|F(u)\|_{L^{6/5} L^{6/5} (\Rm^2 \times [t_m, t_{m+1}])}\\
 & \leq \|u\|_{L^6 L^6 (\Rm^2 \times [t_m, t_{m+1}])} + \|u\|_{L^6 L^6 (\Rm^2 \times [t_m, t_{m+1}])}^5\\
 & \leq \eps_1 + \eps_1^5.
\end{align*}
As a result, if $\eps_1$ is sufficiently small, the solution $u$ in the time interval $[t_m,t_{m+1}]$ can be obtained by a fixed-point argument. More precisely, the restriction of $u$ in the time interval $[t_m, t_{m+1}]$ is the limit of $\tilde{u}_{m,n}$ in $L^6 L^6 (\Rm^2 \times [t_m, t_{m+1}])$ as $n \rightarrow \infty$ if we set $\tilde{u}_{m,0} =0$ and define
\[
 \tilde{u}_{m, n+1} (t) = S(t-t_m)(u(\cdot, t_m),\partial_t u(\cdot, t_m)) + \int_{t_m}^t \frac{\sin ((t-\tau)\sqrt{-\Delta})}{\sqrt{-\Delta}} F(\tilde{u}_{m,n} (\tau)) d\tau.
\]
This then implies the restriction of $u$ in the time interval $[0, t_{m+1}]$ is the limit of $u_{m,n}$ in $L^6 L^6 (\Rm^2 \times [0, t_{m+1}])$ as $n \rightarrow \infty$ if we set
\[
 u_{m,0}(x,t) = \left\{\begin{array}{ll} u(x,t),& t \in [0,t_m];\\
 0, & t\in (t_m,t_{m+1}];
 \end{array}\right.
\]
and define
\[
 u_{m,n+1} (t) = S(t)(u_0,u_1) + \int_{0}^t \frac{\sin ((t-\tau)\sqrt{-\Delta})}{\sqrt{-\Delta}} F(u_{m,n} (\tau)) d\tau.
\]
Now let us show that the solution $u$ satisfies (\ref{es1}) in the time interval $[0,t_m]$ for all nonnegative integer $m$ by an induction. If $m=0$, this is trivial. Let us assume $(\ref{es1})$ holds for time $t \in [0,t_m]$. Then it is clear that $u_{m,0}$ satisfies the same inequality for $t \in [0,t_{m+1}]$. If we assume that $u_{m,n} (x,t)$ satisfies the inequality (\ref{es1}) when $t\in [0,t_{m+1}]$, then Lemma \ref{lm1} gives
\begin{align*}
 |u_{m,n+1} (x,t)| &\leq C |x|^{-1/2}\left[ A_1 (|x|-t)^{-\delta} +  B_1^5 (|x|-t)^{-5 \delta} \right]\\
  &\leq C (A_1 + B_1^5) |x|^{-1/2} (|x|-t)^{-\delta} \leq B_1 |x|^{-1/2} (|x|-t)^{-\delta},
\end{align*}
if $|x| > R +t$ and $t \in [0,t_{m+1}]$. By  induction we obtain that the inequality (\ref{es1}) holds for any $u_{m,n}$ with this particular $m$ and an arbitrary nonnegative integer $n$. Passing to the limit, we obtain the estimate (\ref{es1}) for $u$ if $t \in [0, t_{m+1}]$. This finishes the proof by induction.
\end{proof}
Note at this point that none of the arguments above use the radial assumption. However, the following propositions do require a radial assumption. The letter $r$ below is the radius $r = |x|$.
\begin{proposition} \label{lm3}
 Let $(u_0,u_1)$ be initial data as in Theorem \ref{main1} and the constants $\delta$, $B_1$, $R$ be as in the Proposition \ref{pointwise estimate}. Then if $t\geq 0$ and $r> t + R$ we have the following estimates for the solution $u$
\begin{align}\label{ineq1}
 |(\sqrt{r}u)_t + (\sqrt{r}u)_r| & \lesssim_{\delta} r^{-1-\delta}; \\\label{ineq2}
 |u_t + u_r| & \lesssim_{\delta} r^{-3/2};\\\label{ineq3}
 |u_t - u_r| & \lesssim_{\delta} r^{-1/2}.
\end{align}
\end{proposition}
\begin{proof}
A basic computation shows that $\sqrt{r} u$ satisfies the 1-dimensional wave equation
\[
 (\partial_t^2 - \partial_r^2) (\sqrt{r} u) = \sqrt{r}(\partial_t^2 - \Delta_x) u + \frac{1}{4} r^{-3/2} u
 = -\sqrt{r} |u|^4 u + (1/4)r^{-3/2} u \doteq G(r,t).
\]
By the inequality (\ref{es1}), we can estimate the non-linear term $G(r,t)$ by
\[
 |G(r,t)| = |-\sqrt{r} |u|^4 u + (1/4)r^{-3/2} u| \lesssim_{\delta} r^{-2} (r-t)^{-\delta}
\]
if $r > t +R$ and $t\geq 0$. Assume that the pair $(r_0,t_0)$ satisfies $r_0 > t_0 +R$ and $t_0\geq 0$. Then by the identity
\[
 (\partial_t - \partial_r) ((\sqrt{r}u)_t + (\sqrt{r}u)_r) = G(r,t),
\]
\begin{figure}[h]
\centering
\includegraphics[scale=.80]{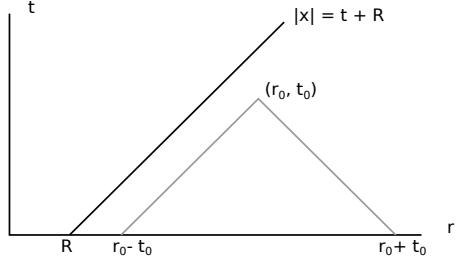}
\caption{Integral Paths}\label{drawing03}
\end{figure}
we obtain (see the integral path shown in Figure \ref{drawing03})
\begin{align*}
 |(\sqrt{r}u)_t + (\sqrt{r}u)_r|_{(r,t)=(r_0,t_0)} & \leq |(\sqrt{r}u)_t + (\sqrt{r}u)_r|_{(r,t)=(r_0+t_0,0)} +
 \left|\int_0^{t_0} G(r_0+t_0-s,s) ds\right|\\
 & \leq \left|\sqrt{r}u_1 + \sqrt{r} \partial_r u_0 + \frac{1}{2}r^{-1/2}u_0\right|_{r=r_0+t_0}\!\! +
 \int_0^{t_0} |G(r_0+t_0-s,s)| ds\\
 & \lesssim_\delta (r_0 +t_0)^{-1-\delta} + \int_0^{t_0}(r_0 + t_0 -s)^{-2} \left[(r_0 +t_0 -s)- s\right]^{-\delta} ds\\
 & \lesssim r_0^{-1-\delta} + r_0^{-2} \int_0^{t_0} (r_0 + t_0 - 2 s)^{-\delta} ds\\
 & \lesssim r_0^{-1-\delta}.
\end{align*}
This gives the first inequality \eqref{ineq1}. By the identity $(\sqrt{r}u)_t + (\sqrt{r}u)_r = r^{1/2} u_t + r^{1/2} u_r +(1/2) r^{-1/2}u$, we have 
\[
 |r^{1/2} u_t + r^{1/2} u_r|_{(r,t)=(r_0,t_0)} \leq |(\sqrt{r}u)_t + (\sqrt{r}u)_r|_{(r,t)=(r_0,t_0)} + (1/2) r_0^{-1/2} |u(r_0,t_0)| \lesssim_\delta r_0^{-1}.
\]
This is in fact equivalent to the second inequality \eqref{ineq2}. Similarly the identity
\[
 (\partial_t + \partial_r) ((\sqrt{r}u)_t - (\sqrt{r}u)_r) = G(r,t)
\]
implies
\begin{align*} 
 |(\sqrt{r}u)_t - (\sqrt{r}u)_r|_{(r,t)=(r_0,t_0)} & \leq |(\sqrt{r}u)_t - (\sqrt{r}u)_r|_{(r,t)=(r_0-t_0,0)} +
 \left|\int_0^{t_0} G(r_0 -t_0 + s, s) ds\right|\\
 & \leq \left|\sqrt{r}u_1 - \sqrt{r} \partial_r u_0 - \frac{1}{2}r^{-1/2}u_0\right|_{r=r_0-t_0} \! + \int_0^{t_0} |G(r_0- t_0 + s,s)| ds\\
 & \lesssim_\delta (r_0 - t_0)^{-1-\delta} +\int_0^{t_0} (r_0 - t_0 + s)^{-2} \left[(r_0 -t_0 + s)-s\right]^{-\delta} ds\\
 & \leq 1 + \int_0^{t_0} (r_0 - t_0 + s)^{-2} ds\\
 & \lesssim 1.
\end{align*}
By the identity $(\sqrt{r}u)_t - (\sqrt{r}u)_r = r^{1/2} u_t - r^{1/2} u_r - (1/2) r^{-1/2}u$, we have
\[
 |r^{1/2} u_t - r^{1/2} u_r|_{(r,t)=(r_0,t_0)} \leq |(\sqrt{r}u)_t - (\sqrt{r}u)_r|_{(r,t)=(r_0,t_0)} + (1/2) r_0^{-1/2} |u(r_0,t_0)| \lesssim_\delta 1.
\]
This finishes the proof of the last inequality \eqref{ineq3}.
\end{proof}
\begin{remark} \label{ut ur estimate}
Proposition \ref{lm3} implies $u_r, u_t \lesssim_\delta r^{-1/2}$ if $r > t + R$ and $t\geq 0$.
\end{remark}
\subsection{Translation Between Two Equations}
Let us start by explaining the transformation between a light cone in $\Rm^2 \times \Rm$ and the space-time $\Hm^2 \times \Rm$. We follow the same method used in \cite{tataru}. Let $(x,t)$ be the Cartesian coordinates in the Minkowski space $\Rm^2 \times \Rm$ 
and $t_0$ be a fixed time. In the forward light cone $\{(x,t): t-t_0 > |x|\}$ we can introduce new coordinates $(\tau, s, \Theta) \in \Rm \times [0,\infty) \times {\mathbb S}^1$ by
\begin{align*}
 x & = \Theta e^\tau \sinh s,\\
 t & = t_0 + e^\tau \cosh s.
\end{align*}
We interpret $(s, \Theta)$ as the polar coordinates in the hyperbolic plane $\Hm^2$ and $\tau$ as the substitute for time. The surface with a constant value $\tau = \tau_0$ is exactly the upper sheet of the hyperboloid $(t-t_0)^2 - |x|^2 = e^{2\tau_0}$ in the Minkowski space. Let $d \mu$ be the volume element in the hyperbolic plane $\Hm^2$. Change of variables gives
\[
 dx dt = e^{3\tau} d\mu d\tau .
\]
A simple computation gives the identity
\[
 e^{\rho \tau} (-\partial_t^2 + \Delta_x ) e^{-\rho \tau} = e^{-2\tau} (-\partial_\tau^2 + \Delta_{\Hm} + \rho^2),
\]
and as a direct corollary we have the following proposition.
\begin{proposition} \label{translation}
 If $u(x,t)$ is a solution to the wave equation \eqref{(CP2)} in $\Rm^2 \times (t_0, \infty)$, then the function $v (s, \Theta, \tau) = e^{\rho \tau} u (\Theta e^\tau \sinh s, t_0 + e^\tau \cosh s)$ is a solution to the following shifted wave equation \eqref{(CP1star)} in $\Hm^2 \times \Rm$
 \[
 \partial_\tau^2 v - (\Delta_{\Hm^2} + \rho^2) v = - |v|^4 v. \]
\end{proposition}

\subsection{Proof of the Scattering Theory}
As we mentioned at the beginning of this section, the idea is to consider the solution $v = e^{\rho \tau} u$ to the shifted wave equation \eqref{(CP1star)} on the hyperbolic plane $\Hm^2$ and see what the Morawetz-type inequality obtained in Sections \ref{section4} and \ref{section5} gives when we come back to the Minkouski space $\Rm \times \Rm^2$. Since the initial data $(u_0,u_1)$ have been assumed to be smooth in Theorem \ref{main1}, we only need to consider smooth solution $u$ in the argument below.

\begin{figure}[h]
\centering
\includegraphics[scale=.60]{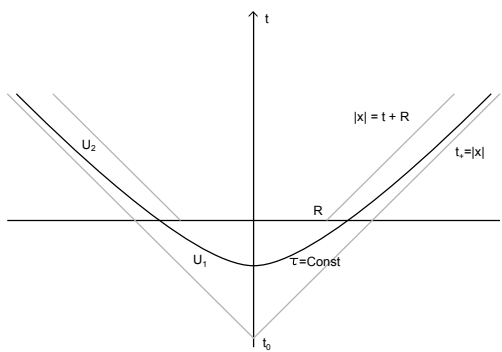}
\caption{Illustration of the transformation}\label{drawing02}
\end{figure}

\paragraph{Set-up} Let $R$ be the constant in Proposition \ref{pointwise estimate}. We choose a time $t_0 < -\sqrt{R^2 + 1}$ and consider the solution
\[
 v(s, \Theta, \tau) = e^{\rho \tau} u (\Theta e^\tau \sinh s, t_0 + e^\tau \cosh s)
\]
to the shifted wave equation \eqref{(CP1star)} on $\Hm^2$. Figure \ref{drawing02} illustrates the transformation from the light cone $\{(x,t): t-t_0 > |x|\}$ to $\Hm^2 \times \Rm$. As in the figure, we call $t_+ = t -t_0$ below. Using  the radial condition, we can calculate the partial derivatives of $v$ as follows:
\begin{align*}
 \partial_\tau v &= e^{\rho \tau} (\rho u + u_r e^\tau \sinh s + u_t e^\tau \cosh s);\\
 \partial_s v &= e^{\rho \tau} (u_r e^\tau \cosh s + u_t e^\tau \sinh s).
\end{align*}
The value of $u$ or its derivatives are taken at $(\Theta e^\tau \sinh s, t_0 + e^\tau \cosh s)$.

\paragraph{Energy at time $\tau$} Let us first fix a time $\tau$. As we mentioned earlier, the data of function $v$ at time $\tau$ are determined by the data of $u$ on the upper sheet of the hyperboloid $(t-t_0)^2 - |x|^2 = e^{2\tau}$. In order to avoid any point on this hyperboloid to fall in the region $\{(x,t)||x|<t+R,t>0\}$, where we can not apply our estimate on $u$ found in Section \ref{sec:pointwisees}, we restrict our choice of $\tau$ in the interval $[-1,0]$. If $U \subseteq \Hm^2$, we can consider the local energy of $v$ in $U$ at time $\tau$ given by the integral
\begin{align}
 2 \EE_{U} = &\int_{U} \left[|\partial_\tau v|^2 + |\partial_s v|^2 - \rho^2 v^2 + (1/3)|v|^6\right] d\mu \nonumber\\
 = &\int_U e^{2\rho \tau} \left[(\rho u + u_r e^\tau \sinh s + u_t e^\tau \cosh s)^2 - \rho^2 u^2
  + (u_r e^\tau \cosh s + u_t e^\tau \sinh s)^2\right] d\mu \nonumber\\
  &+ \frac{1}{3}\int_{U} e^{3\tau}|u|^6 d\mu \nonumber\\
  = &\int_U e^{2\rho\tau} \left[(u_r^2 +u_t^2)e^{2\tau}(\cosh^2 s + \sinh^2 s) + 4 u_r u_t e^{2\tau}\sinh s \cosh s
  \right] d\mu\nonumber\\
  & + \int_U e^{2\rho\tau}\left[2\rho u u_r e^{\tau} \sinh s + 2\rho u u_t e^{\tau} \cosh s\right] d\mu
  + \frac{1}{3}\int_{U} e^{3\tau}|u|^6 d\mu. \label{local energy1}
\end{align}
Let us consider two different types of subsets $U$.

\paragraph{Center Disk} The first type of subset $U$ is the disk
\[
 U_1(\tau) =  \{(s,\Theta) \in \Hm^2: s < s_\tau\}.
\]
The upper bound $s_\tau = \cosh^{-1}(-e^{-\tau} t_0)$ corresponds to the value of $t=0$. We define $J_1(\tau)$ to be the local energy of $v$ in the region $U_1 (\tau)$ at time $\tau$. By the calculation above, we have
\begin{align}
 2J_1 (\tau) = &\int_{U_1(\tau)} e^{2\rho\tau} \left[(u_r^2 +u_t^2)e^{2\tau}(\cosh^2 s + \sinh^2 s) + 4 u_r u_t e^{2\tau}\sinh s \cosh s
  \right] d\mu \nonumber \\
  & + \int_{U_1 (\tau)} e^{2\rho\tau}\left[2\rho u u_r e^{\tau} \sinh s + 2\rho u u_t e^{\tau} \cosh s\right] d\mu
  + \frac{1}{3}\int_{U_1 (\tau)} e^{3\tau}|u|^6 d\mu. \label{def of J1}
\end{align}
This is clearly finite for each given $\tau$ since we know $u(x,t)$ is smooth.

\paragraph{Large Rings} The second type of region $U_2$ is given by
\[
 U_2(\tau, s_0) = \{(s,\Theta) \in \Hm^2: s_\tau < s < s_0\}.
\]
As we did for the region $U_1$, we define $J_2 (\tau, s_0)$ to be the local energy of $v$ in $U_2 (\tau, s_0)$ at the time $\tau$. By rewriting the integral (\ref{local energy1}) into polar coordinates, we can give a formula for $J_2$ as
\begin{align*}
 2 J_2 (\tau,s_0) = & c \int_{s_\tau}^{s_0} \left[(u_r^2 +u_t^2)e^{2\tau}(\cosh^2 s + \sinh^2 s) + 4 u_r u_t e^{2\tau}\sinh s \cosh s
  \right] (e^\tau \sinh s)^{2 \rho} ds\\
  & +  c \int_{s_\tau}^{s_0} \left[2\rho u u_r e^{\tau} \sinh s + 2\rho u u_t e^{\tau} \cosh s\right] (e^\tau\sinh s)^{2\rho} ds
  + \frac{c}{3}\int_{s_\tau}^{s_0} e^{3\tau}|u|^6 \sinh s ds.
\end{align*}
Applying the change of variables $r = e^\tau \sinh s$ and the substitution
\[
 t_+ = e^\tau \cosh s\qquad \Rightarrow \qquad dr = e^{\tau} \cosh s ds = t_+ ds;
\]
we obtain
\begin{align*}
 2 J_2 (\tau,s_0) = & c \int_{\sqrt{t_0^2- e^{2\tau}}}^{e^\tau \sinh s_0} \left[(u_r^2 +u_t^2)(t_+^2 + r^2) + 4 u_r u_t t_+ r
  + 2\rho u u_r r + 2\rho u u_t t_+ \right] \frac{r^{2\rho}}{t_+} dr\\
  & + \frac{c}{3}\int_{\sqrt{t_0^2- e^{2\tau}}}^{e^\tau \sinh s_0} e^{2\tau}|u|^6 \frac{r}{t_+} dr.\\
  = & c \int_{\sqrt{t_0^2- e^{2\tau}}}^{e^\tau \sinh s_0} g(\tau,r) dr,
\end{align*}
where the integrand $g(\tau,r)$ is given by
\[
 g(\tau,r) = \frac{r}{t_+} \left[(u_r^2 +u_t^2)(t_+^2 + r^2) + 4 u_r u_t t_+ r
  + 2\rho u u_r r + 2\rho u u_t t_+ + \frac{1}{3} e^{2\tau} |u|^6\right].
\]
The points $(r,t) = (r, t_0 + \sqrt{r^2 + e^{2\tau}})$, where the function $u$ and its derivatives are evaluated in the formula above satisfy
\[
 r - t = r - (t_0 + \sqrt{r^2 + e^{2\tau}}) \geq \sqrt{t_0^2 -e^{2\tau}} - t_0 + t_0 > R
\]
and $t \geq 0$ by our choice of $t_0$ and $\tau$. Thus one can always apply to $g$ the estimates on $u$ obtained in Proposition \ref{pointwise estimate} and \ref{lm3}.

\paragraph{Estimate of $g(\tau,r)$ } Let us recall $\tau\in[-1,0]$. The function $g(\tau,r)$ can be written as the following sum
\begin{align*}
 g(\tau, r) = &  \left[2r^2 (u_r^2 + u_t^2) + 4r^2 u_r u_t + 2\rho u u_r r + 2\rho u u_t r\right]
 + \frac{e^{2\tau}}{3}\frac{r}{t_+}|u|^6\\ & + \left(\frac{r(r^2 + t_+^2)}{t_+} - 2r^2\right)(u_r^2 + u_t^2)
 + \left(\frac{r^2}{t_+}-r\right)(2\rho u u_r) \\
 = & g_1 (\tau,r) + g_2 (\tau,r) + g_3 (\tau,r) + g_4 (\tau,r).
\end{align*}
The main contribution comes from the first term
\begin{align*}
 |g_1 (\tau,r)|& = \left|  2r \left(\sqrt{r} u_t + \sqrt{r} u_r + \frac{\rho u}{\sqrt{r}}\right)\left(\sqrt{r} u_t + \sqrt{r} u_r \right) \right|\\
 & = \left| 2r \left((\sqrt{r} u)_t + (\sqrt{r} u)_r \right)\cdot\sqrt{r} \left(u_t +  u_r \right) \right|\\
 & \leq  2r^{3/2} \left|(\sqrt{r} u)_t + (\sqrt{r} u)_r \right| \left|u_t + u_r\right| \\
 & \lesssim_\delta  2 r^{3/2} \cdot r^{-1-\delta} \cdot r^{-3/2} \\
 & \lesssim   r^{-1-\delta},
\end{align*}
where we used Proposition \ref{lm3}.
We now observe that
\[
 \left|\frac{r(r^2 + t_+^2)}{t_+} - 2r^2\right| = \frac{r}{t_+}\left|2r^2 + e^{2\tau}- 2r t_+\right| \leq \frac{e^{4\tau}}{2 r^2 + e^{2\tau} + 2r t_+} \leq r^{-2};
\]
\[
 \left|\frac{r^2}{t_+}-r\right| = \frac{r}{t_+} \cdot \frac{e^{2\tau}}{r + t_+} \leq r^{-1};
\]
and substitute $u$, $u_r$, $u_t$ with their upper bounds we found earlier in Propositions \ref{pointwise estimate}, \ref{lm3} and Remark \ref{ut ur estimate}. As a result, we obtain
\begin{align*}
 |g_2 (\tau,r)| & \lesssim_\delta 1 \cdot r^{-3} = r^{-3};\\
 |g_3 (\tau,r)| & \lesssim_\delta r^{-2} \cdot r^{-1} = r^{-3};\\
 |g_4 (\tau,r)| & \lesssim_\delta r^{-1} \cdot r^{-1} = r^{-2}.
\end{align*}
Adding these up, we have $|g(\tau, r)| \lesssim_\delta r^{-1-\delta}$. Therefore
\[
 |J_2 (\tau,s_0)| \lesssim_\delta \int_{\sqrt{t_0^2- e^{2\tau}}}^{e^\tau \sinh s_0} r^{-1-\delta} dr
 \lesssim_\delta 1.
\]
In other words, $J_2$ has an upper bound which is independent of $\tau \in [-1,0]$ and $s_0\gg 1$.

\paragraph{$L^6 L^6$ Norm in the cone} We would like to conclude that the solution $v$ to \eqref{(CP1star)} has a finite energy by using the fact that $J_2(\tau,s_0)$ is uniformly bounded as $s_0 \rightarrow \infty$. However, the negative term $-\rho^2 |v|^2$ in the energy imposes a technical difficulty in the limit process. We solve this problem by applying smooth cut-off techniques. Let $\phi: \Rm \rightarrow [0,1]$ be a smooth cut-off function such that
\[
 \phi(s)=\left\{\begin{array}{ll}1, & \hbox{if}\;\;s<0;\\ 0, & \hbox{if}\;\; s>1. \end{array}\right.
\]
Let us define the cut-off version of initial data of $v$ at time $\tau$ as
\[
 v_{0,s_0} (s, \Theta) = \phi(s-s_0) v(\tau,s,\Theta);\;\;\; v_{1,s_0}(s,\Theta) = \phi(s-s_0) \partial_\tau v(\tau,s,\Theta)
\]
This pair of initial data is smooth and compact-supported. Observing that
\[
 \partial_s v_{0,s_0}(s,\Theta) = \phi'(s-s_0) v(\tau,s,\Theta) + \phi (s-s_0) \partial_s v(\tau,s,\Theta),
\]
we obtain
\begin{equation}\label{ineq4}
 |\nabla v_{0,s_0}(s,\Theta)|^2 \leq C |v(\tau,s,\Theta)|^2 + 2 |\partial_s v(\tau,s,\Theta)|^2.
\end{equation}
Here the constant $C$ depends on $\phi$ only. Using \eqref{ineq4} we have the following estimate on the energy
\begin{align*}
 \EE(v_{0,s_0} &, v_{1,s_0}) - J_1 (\tau) - J_2 (\tau, s_0)\\
& = \int_{s_0<s <s_0+1} \left[\frac{1}{2}|\nabla v_{0,s_0}|^2+ \frac{1}{2}|v_{1,s_0}|^2 - \frac{\rho^2}{2} |v_{0,s_0}|^2 + \frac{1}{6}|v_{0,s_0}|^6  \right] d\mu \\
 &\leq \int_{s_0<s <s_0+1} \left[\frac{C}{2}|v(\tau,\cdot)|^2 + |\partial_s v(\tau,\cdot)|^2 + \frac{1}{2}|\partial_\tau v(\tau,\cdot)|^2+ \frac{1}{6}|v(\tau,\cdot)|^6 \right] d\mu\\
  &\leq 2\left[J_2 (\tau, s_0+1)-J_2 (\tau, s_0)\right] + (C/2 + \rho^2) \int_{s_0<s <s_0+1} |v(\tau,\cdot)|^2 d\mu\\
 &\lesssim_\delta 1 + \int_{s_0}^{s_0+1} |v(\tau,\cdot)|^2 \sinh s ds\\
& = 1 + \int_{e^\tau \sinh s_0}^{e^\tau \sinh (s_0+1)} |u|^2 \frac{r}{t_+} dr\\
 &\lesssim_\delta 1 + \int_{e^\tau \sinh s_0}^{e^\tau \sinh (s_0+1)} r^{-1} dr \lesssim 1.
\end{align*}
This gives us an upper bound of the energy, which does not depend on $s_0$. Namely, for each $\tau\in [-1,0]$
\begin{equation}\label{ineq5}
 \EE(v_{0,s_0}, v_{1,s_0}) < \EE_0 \doteq J_1 (\tau) + C_1 (\delta).
\end{equation}
Let us consider the solution $v_{s_0}$ to the shifted wave equation \eqref{(CP1star)} with initial data $(v_{0,s_0}, v_{1,s_0})$ at time $\tau$. Proposition \ref{scatteringn2pgeq5} guarantees its maximal lifespan is $\Rm$ and the key estimate immediately follows by Theorem \ref{Morawetz2}
\[
 \iint_{\Hm^2 \times \Rm} |v_{s_0}|^6 d\mu d\tau' \leq 6 \EE_0 < \infty.
\]
By finite speed of propagation, we know $v_{s_0} = v$ in the region
\[
 K_{\tau, s_0} = \left\{(\tau', s, \Theta)\in \Rm \times \Hm^2 : s + |\tau' - \tau| < s_0 \right\}.
\]
As a result, we have the inequality
\[
 \iint_{K_{\tau,s_0}} |v|^6 d\mu d\tau' \leq 6 \EE_0.
\]
Letting $s_0 \rightarrow \infty$, we obtain the estimate
\[
 \int_{\Hm^2 \times \Rm} |v|^6  d\mu d\tau' \leq 6 \EE_0.
\]
By a change of variables
\[
 |v|^6 d\mu d\tau' = |e^{\rho \tau'}u|^6 d\mu d\tau' = |u|^6 e^{3\tau'} d\mu d\tau' = |u|^6 dx dt,
\]
we obtain
\[
 \iint_{t_+ >|x|} |u|^6 dx dt \leq 6 \EE_0.
\]
The integral region contains the set $\{(t,x): |x|\leq t+R, t>0\}$, see Figure \ref{drawing02} above. Therefore we have
\begin{equation}\label{ineq6}
 \iint_{|x|\leq t + R, t>0} |u|^6 dx dt \leq 6 \EE_0.
\end{equation}
On the other hand, we can apply the point-wise estimate of $u$ given by Proposition \ref{pointwise estimate} if $|x| > t + R$ and $t>0$,
\begin{align}\notag
 \iint_{|x| > t + R, t> 0} |u(x,t)|^6 dx dt & \lesssim_\delta \iint_{|x| > t + R, t> 0} |x|^{-3}(|x|-t)^{-6\delta} dx dt\\\notag
 & \leq \int_{|x|>R} \left(\int_{0}^{|x|-R} |x|^{-3}(|x|-t)^{-6\delta} dt\right) dx\\\notag
 & \lesssim \int_{|x|>R} |x|^{-3} |x|^{1-6\delta} dx\\\notag
 & \lesssim \int_R^\infty r^{-2-6\delta} r dr\\\label{ineq7}
 & \lesssim_\delta R^{-6\delta} \leq 1.
\end{align}
In summary, we have $\|u\|_{L^6 L^6 (\Rm^+ \times \Rm^2)} < \infty$. Thus we obtain the scattering in the positive time direction. Since the wave equation is time-reversible, the scattering in negative time direction can be proved in the same way.

\subsection{Further Improvement}

In this subsection, we first show that there exists an upper bound of the norm $\|u\|_{L^6 L^6 (\Rm \times \Rm^2)}$, which depends only on the explicit parameters $A$ and $\eps$ in Theorem \ref{main1}. As a result, we can show the smoothness condition is not necessary by smooth approximation.

\paragraph{The upper bound of the $L^6 L^6$ norm} By collecting the upper bounds found in \eqref{ineq5},  \eqref{ineq6} and \eqref{ineq7} in the previous subsection, we obtain
\begin{proposition}
Let $u$ be a solution to the equation \eqref{(CP2)} as in Theorem \ref{main1}. Then the $L^6 L^6$ norm of $u$ has an upper bound of the form
\begin{equation} \label{upper bound}
 \|u\|_{L^6 L^6 (\Rm^+ \times \Rm^2)}^6 \leq C_2 (\delta) + 6 \inf_{\tau \in [-1,0]}  J_1 (\tau),
\end{equation}
where the constant $C_2$ depends on $\delta$ only.
\end{proposition}
\paragraph{Estimate of $J_1 (\tau)$} We start by showing the parameter $\tau$ in the upper bound found in (\ref{upper bound}) can be chosen so that the local energy $J_1 (\tau)$ is dominated by a constant determined solely by the parameters $A$ and $\eps$ in Theorem \ref{main1}. The idea is to integrate $J_1$ in $\tau$. By the identity (\ref{def of J1}) we have
\begin{align*}
 2\inf J_1 (\tau)
\leq &\int_{K} e^{2\rho\tau} \left[(u_r^2 +u_t^2)e^{2\tau}(\cosh^2 s + \sinh^2 s) + 4 u_r u_t e^{2\tau}\sinh s \cosh s \right] d\mu d\tau\\
& + \int_{K} e^{2\rho\tau}\left[2\rho u u_r e^{\tau} \sinh s + 2\rho u u_t e^{\tau} \cosh s\right] d\mu d\tau
  + \frac{1}{3}\int_{K} e^{3\tau}|u|^6 d\mu d\tau,
\end{align*}
where the  region of integration  is defined as $K = \{(y,\tau)\in \Hm^2 \times \Rm: \tau \in [-1,0], y \in U_1 (\tau)\}$. A change of variables $dx dt = e^{3\tau} d\mu d\tau$ shows
\begin{align*}
 2\inf J_1 (\tau)
\leq &\int_{K'} e^{-2\tau} \left[(u_r^2 +u_t^2)(|x|^2 + t_+^2) + 4 u_r u_t |x| t_+ + 2\rho u u_r |x| + 2\rho u u_t t_+\right] dx dt\\
  & + \frac{1}{3}\int_{K'}|u|^6 dx dt,
\end{align*}
where the  region $K' = \{(x,t): e^{-2} < (t-t_0)^2 - |x|^2 < 1, t_0 < t < 0\}$, see Figure \ref{drawing05}.
\begin{figure}[h]
\centering
\includegraphics[scale=.60]{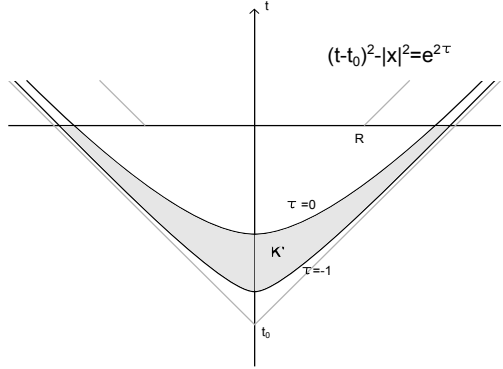}
\caption{The region for integral}\label{drawing05}
\end{figure}

Since this region is compactly supported, we are able to find easily an upper bound:
\begin{align*}
 2 \inf J_1 (\tau) \lesssim & \int_{K'} \left[(u_r^2 +u_t^2)(|x|^2 + t_+^2) + 4 |u_r| |u_t| |x| t_+ + 2\rho |u| |u_r| |x| + 2\rho |u| |u_t| t_+ + |u|^6\right] dx dt\\
 \lesssim & \int_{K'} \left[(u_r^2 +u_t^2) t_+^2 + |u|^2 + |u|^6\right] dx dt\\
 \lesssim & \int_{t_0+ e^{-1}}^0 \int_{e^{-2} < (t-t_0)^2 -|x|^2 < 1} \left[(u_r^2 +u_t^2) t_0^2 + 1 + |u|^6\right] dx dt\\
 \lesssim & \int_{t_0+ e^{-1}}^0 \left[t_0^2 E(u_0,u_1) + 1\right] dt\\
 \leq & |t_0|^3 E(u_0,u_1) + |t_0|.
\end{align*}
According to Remark \ref{H12A}, we also have that the energy $E(u_0,u_1)$ is dominated by
\[
 E(u_0,u_1) \lesssim A^2 + A^6.
\]
As a result, we have $\inf J_1 (\tau) \lesssim |t_0| + |t_0|^3 (A^2 + A^6)$. By (\ref{upper bound}), we obtain
\begin{equation}
 \|u\|_{L^6 L^6 (\Rm^+ \times \Rm^2)}^6 \lesssim C_2(\delta)+ |t_0| + |t_0|^3 (A^2 + A^6).
\end{equation}
A quick review of the calculations we performed reveals that $\delta$ and $t_0$ can be selected in a way that they are uniquely determined by $A$ and $\eps$. As a consequence we have the stronger estimate
\begin{equation}
 \|u\|_{L^6 L^6 (\Rm \times \Rm^2)} \leq C(\eps, A). \label{improved L6 bound on smooth}
\end{equation}
Here we have substituted $\Rm_t^+$ by $\Rm$, because the negative time direction can be handled in the same manner.
\paragraph{Final results} The improved version of the main Theorem \ref{main1} now becomes
\begin{theorem} \label{improved main theorem}
Let $(u_0,u_1)$ be a pair of radial initial data satisfying the following inequalities
\begin{align*}
 |\nabla u_0|, |u_1| & \leq A (1+|x|)^{-\frac{3}{2} -\eps};\\
 |u_0| & \leq A (1+|x|)^{-\frac{1}{2}-\eps};
\end{align*}
with two positive constants $A$ and $\eps$. Then the solution $u$ of the wave equation \eqref{(CP2)}  exists globally in time and scatters with a finite space-time norm
 \[
  \|u\|_{L^6 L^6 (\Rm \times \Rm^2)} < C(\eps,A) <\infty.
 \]
Moreover there exists two pairs $(u_0^\pm, u_1^\pm) \in \dot{H}^{\frac{1}{2}} \times \dot{H}^{-\frac{1}{2}} (\Rm^2)$, such that
 \[
  \lim_{t \rightarrow \pm \infty} \left\|\left(u(t) - S(t)(u_0^\pm, u_1^\pm),
 \partial_t u(t) - \partial_t S(t)(u_0^\pm, u_1^\pm)\right)\right\|_{\dot{H}^{\frac{1}{2}} \times \dot{H}^{-\frac{1}{2}}} = 0.
 \]
\end{theorem}
\begin{remark} \label{H12}
 As in Remark \ref{H12A}, basic computation shows that the initial data in Theorem \ref{improved main theorem} satisfy
 \begin{align*}
  (u_0,u_1) & \in \dot{W}^{1, \frac{4}{3}} \times L^{\frac{4}{3}} (\Rm^2) \hookrightarrow \dot{H}^{\frac{1}{2}} \times \dot{H}^{-\frac{1}{2}} (\Rm^2);\\
  E(u_0,u_1) & \lesssim A^2 + A^6 < \infty.
 \end{align*}
\end{remark}
\begin{proof}[Proof of theorem \ref{improved main theorem}]
 Let us choose a nonnegative, smooth, radial cut-off function $\phi(x)$ compactly supported in the unit disk with $\int_{\Rm^2} \phi(x) dx =1$ and define ($0 < \lambda < 1$)
 \[
  \phi_\lambda(x) = \frac{1}{\lambda^2}\phi(\frac{x}{\lambda}).
 \]
 Then we can smooth out the initial data by defining
 \[
  (u_{0,\lambda}, u_{1,\lambda}) = \phi_\lambda \ast (u_0,u_1).
 \]
 These pairs are obviously smooth and radial with the gradient
 \[
  \nabla u_{0,\lambda} = \phi_\lambda \ast \nabla u_0.
 \]
 We can estimate
 \begin{align*}
  |\nabla u_{0,\lambda}(x_0)| &\leq \left(\phi_\lambda \ast |\nabla u_0|\right)(x_0) \leq \left(\phi_\lambda \ast A (1+|x|)^{-\frac{3}{2}-\eps} \right)(x_0) \leq C(\eps) A (|x_0|+1)^{-\frac{3}{2}-\eps};\\
  |u_{1,\lambda}(x_0)| &\leq \left(\phi_\lambda \ast |u_1|\right)(x_0) \leq \left(\phi_\lambda \ast A (1+|x|)^{-\frac{3}{2}-\eps} \right)(x_0) \leq C(\eps) A (|x_0|+1)^{-\frac{3}{2}-\eps};\\
  |u_{0,\lambda}(x_0)| &\leq \left(\phi_\lambda \ast |u_0|\right)(x_0) \leq \left(\phi_\lambda \ast A (1+|x|)^{-\frac{1}{2}-\eps} \right)(x_0) \leq C(\eps) A (|x_0|+1)^{-\frac{1}{2}-\eps}.
 \end{align*}
 By Theorem \ref{main1} and the $L^6$ norm estimate (\ref{improved L6 bound on smooth}), these inequalities guarantee the existence of a universal constant $C = C(\eps, A)$, such that for any $\lambda \in (0,1)$, the solution $u_\lambda$ to \eqref{(CP2)} with initial data $(u_{0,\lambda}, u_{1,\lambda})$ exists for all $t\in \Rm$ and satisfies
 \begin{equation} \label{unibound}
  \|u_\lambda\|_{L^6 L^6 (\Rm \times \Rm^2)} < C.
 \end{equation}
On the other hand, we know both $|\nabla u_0|$ and $u_1$ are in the space $L^{\frac{4}{3}} (\Rm^2)$ by Remark \ref{H12}. This implies
 \begin{align*}
  \|\nabla (u_{0,\lambda} - u_0)\|_{L^{\frac{4}{3}}(\Rm^2)} & = \|\phi_{\lambda} \ast \nabla u_0 - \nabla u_0\|_{L^{\frac{4}{3}}(\Rm^2)} \rightarrow 0;\\
  \| (u_{1,\lambda} - u_1)\|_{L^{\frac{4}{3}}(\Rm^2)} &=\|\phi_{\lambda} \ast u_1 - u_1\|_{L^{\frac{4}{3}}(\Rm^2)}\rightarrow 0;
 \end{align*}
 as $\lambda \rightarrow 0$. Sobolev embedding then gives
 \begin{equation} \label{convergence of philambda}
  \|(u_{0,\lambda}, u_{1,\lambda})-(u_0,u_1)\|_{\dot{H}^{\frac{1}{2}} \times \dot{H}^{-\frac{1}{2}} (\Rm^2)} \rightarrow 0.
 \end{equation}
 Now we are able to conclude that $\|u\|_{L^6 L^6 ((-T_-, T_+)\times \Rm^2)} \leq C$. If this were not true, we could obtain a contradiction by  (\ref{unibound}), (\ref{convergence of philambda}) and applying long-time perturbation theory, see Proposition \ref{perturbation theory}. The last step is to apply the finite time blow-up criterion Lemma \ref{finite time criterion} and then conclude that the maximal lifespan of $u$ is $\Rm$.
\end{proof}

\section*{Appendix}

\subsection*{Focusing equation and blow up}
\begin{proposition} Assume $2 \leq n \leq 6$ and $1 < p < p_c$. Let $(u_0,u_1)$ be a pair of nonzero initial data in $H^1 \times H^{\frac{1}{2}, -\frac{1}{2}}(\Hm^n)$ so that the energy associated to the solution to \eqref{(CP1)} in the focusing case is
\[
 E(u_0,u_1) = \int_{\Hm^n} \left(\frac{1}{2} |\nabla u_0|^2 -\frac{\rho^2}{2}|u_0|^2 + \frac{1}{2}|u_1|^2 - \frac{1}{p+1} |u_0|^{p+1} \right) d\mu \leq 0.
\]
Then the solution to the focusing \eqref{(CP1)} with initial data $(u_0,u_1)$ blows up in finite time in both positive and negative  time directions.
\end{proposition}
\begin{proof}
Let us assume $T_+ =\infty$ and deduce a contradiction. The negative time direction can be dealt with in the same manner. First of all, the same argument as in the proof of Lemma \ref{CH1} gives
\[
 (u, \partial_t u) \in C((-T_-,T_+); H^1(\Hm^n) \times L^2 (\Hm^n)),
\]
and the energy conservation law ($t_0 \in (-T_-,T_+)$) reads
\begin{equation}
\frac{1}{2}\|u(\cdot,t_0)\|_{H^{0,1}(\Hm^n)}^2 + \frac{1}{2}\|\partial_t u(\cdot,t_0)\|_{L^2(\Hm^n)}^2 - \frac{1}{p+1}\|u(\cdot,t_0)\|_{L^{p+1}(\Hm^n)}^{p+1} = \EE(u_0,u_1) \leq 0, \label{energycsv}
\end{equation}
which implies
\begin{equation}
 \frac{1}{p+1} \|u(\cdot, t_0)\|_{L^{p+1}(\Hm^n)}^{p+1} \geq \frac{1}{2}\|u(\cdot,t_0)\|_{H^{0,1}(\Hm^n)}^2 + \frac{1}{2}\|\partial_t u(\cdot,t_0)\|_{L^2(\Hm^n)}^2. \label{lp1}
\end{equation}
Furthermore, if we apply the Sobolev embedding $H^{0,1} \hookrightarrow L^{p+1}$ given in Proposition \ref{Sobolevem1} to (\ref{energycsv}), we have
\[
 C \|u(\cdot, t_0)\|_{L^{p+1}(\Hm^n)}^2 + \frac{1}{2}\|\partial_t u(\cdot,t_0)\|_{L^2(\Hm^n)}^2 - \frac{1}{p+1}\|u(\cdot,t_0)\|_{L^{p+1}(\Hm^n)}^{p+1} \leq \EE \leq 0.
\]
This implies $\|u(\cdot, t_0)\|_{L^{p+1}} \gtrsim 1$ for all  $t_0 \in (-T_-,T_+)$. Let us define
\[
 \MM(t) = \int_{\Hm^n} |u(x,t)|^2 d\mu(x).
\]
Applying integration by parts as well as necessary smoothing and truncation techniques as we did in the proof of Morawetz inequality, we obtain
\begin{align}
 \MM'(t) & = 2 \int_{\Hm^n} u \partial_t u d\mu;\nonumber \\
 \MM''(t) & = 2 \int_{\Hm^n} \left(|\partial_t u|^2 -|\nabla u|^2 + \rho^2 |u|^2 + |u|^{p+1}\right) d\mu\nonumber \\
 & = - 4 \EE + \int_{\Hm^n} \left(4 |\partial_t u|^2 + \frac{2(p-1)}{p+1} |u|^{p+1}\right) d\mu.\label{esmpp1}\\
 & \geq - 4\EE + \int_{\Hm^n} (p+3) |\partial_t u|^2 d\mu.\nonumber
\end{align}
In the last step above we use the inequality (\ref{lp1}). Combining (\ref{esmpp1}), the assumption $\EE \leq 0$ and the fact $\|u\|_{L^{p+1}(\Hm^n)} \gtrsim 1$, we obtain the inequality $\MM''(t) \gtrsim 1$. This implies $\MM'(t)$ will be eventually positive as $t \rightarrow \infty$. Let us assume $\MM(t), \MM'(t) > 0$ as $t > t_1$. Using the inequality
\[
 \MM(t) \MM''(t) \geq \left(\int_{\Hm^n} |u|^2 d\mu\right)\left(\int_{\Hm^n} (p+3) |\partial_t u|^2 d\mu\right) \geq \frac{p+3}{4} [\MM'(t)]^2,
\]
we obtain
\[
 \frac{d}{dt}\left(\frac{\MM(t)}{\MM'(t)}\right) = \frac{[\MM'(t)]^2 - \MM''(t) \MM(t)}{[\MM'(t)]^2} \leq \frac{1-p}{4} < 0
\]
for all $t > t_1$. This is a contradiction since we have assumed $\MM(t)/\MM'(t) > 0$ for all $t > t_1$.
\end{proof}
\subsection*{Visual rendering of the Strichartz estimates}
\begin{figure}[h]
\centering
\includegraphics[scale=.50]{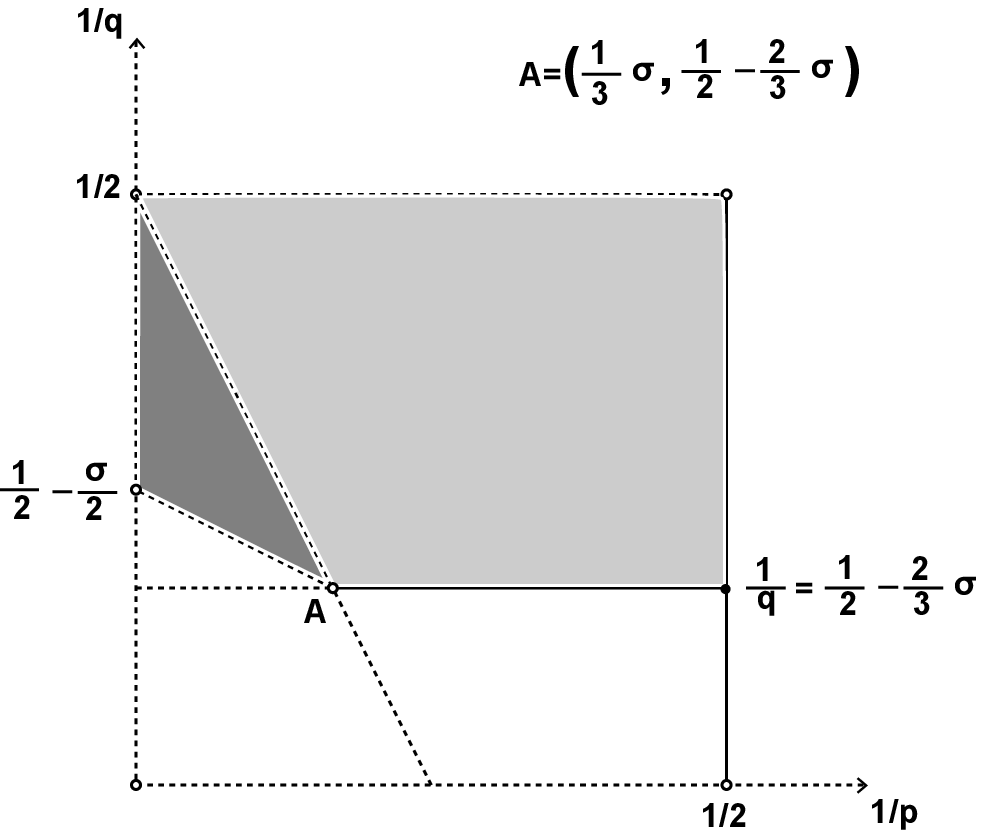}
\caption{Admissible pair $(p_1,q_1)$ in dimension 2, with $\sigma < 3/4$}\label{drawingd21}
\end{figure}
\begin{figure}[h]
\centering
\includegraphics[scale=.50]{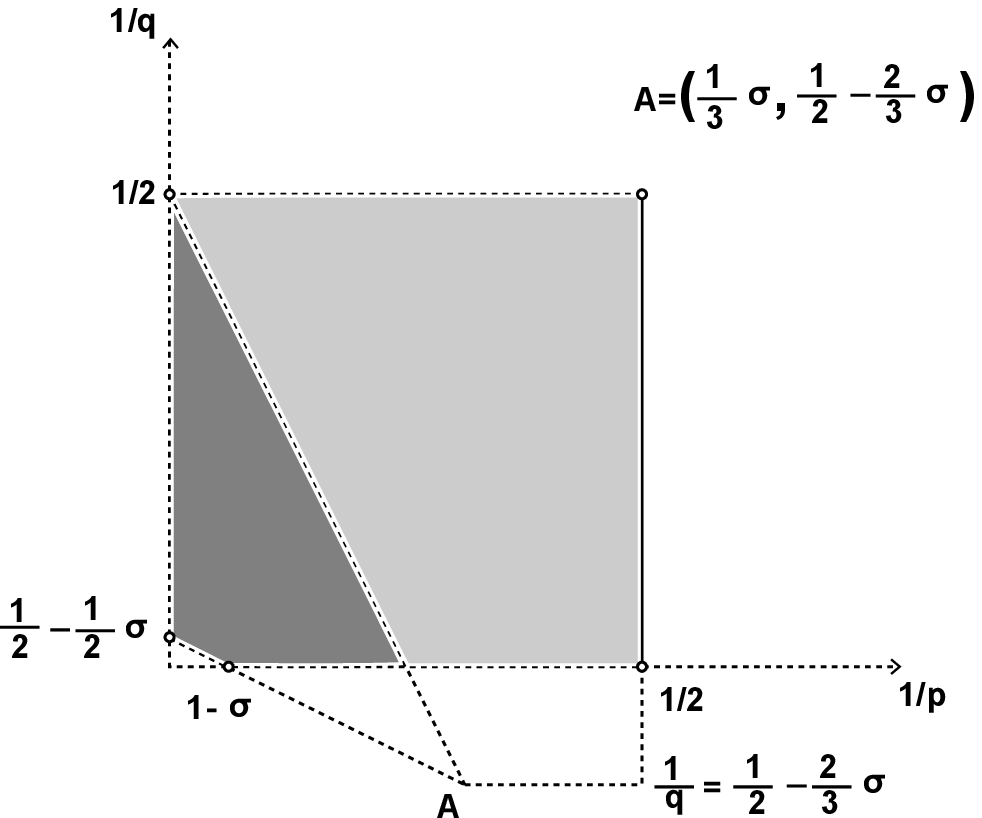}
\caption{Admissible pair $(p_1,q_1)$ in dimension 2, with $\sigma \geq 3/4$}\label{drawingd22}
\end{figure}
\begin{figure}[h]
\centering
\includegraphics[scale=.50]{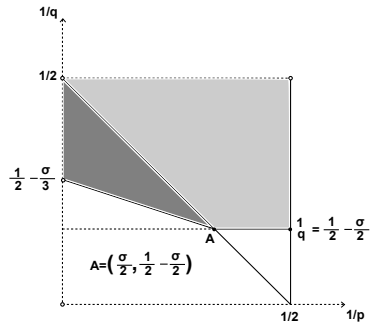}
\caption{Admissible pair $(p_1,q_1)$ in dimension 3}\label{drawingd3}
\end{figure}
\begin{figure}[h]
\centering
\includegraphics[scale=.50]{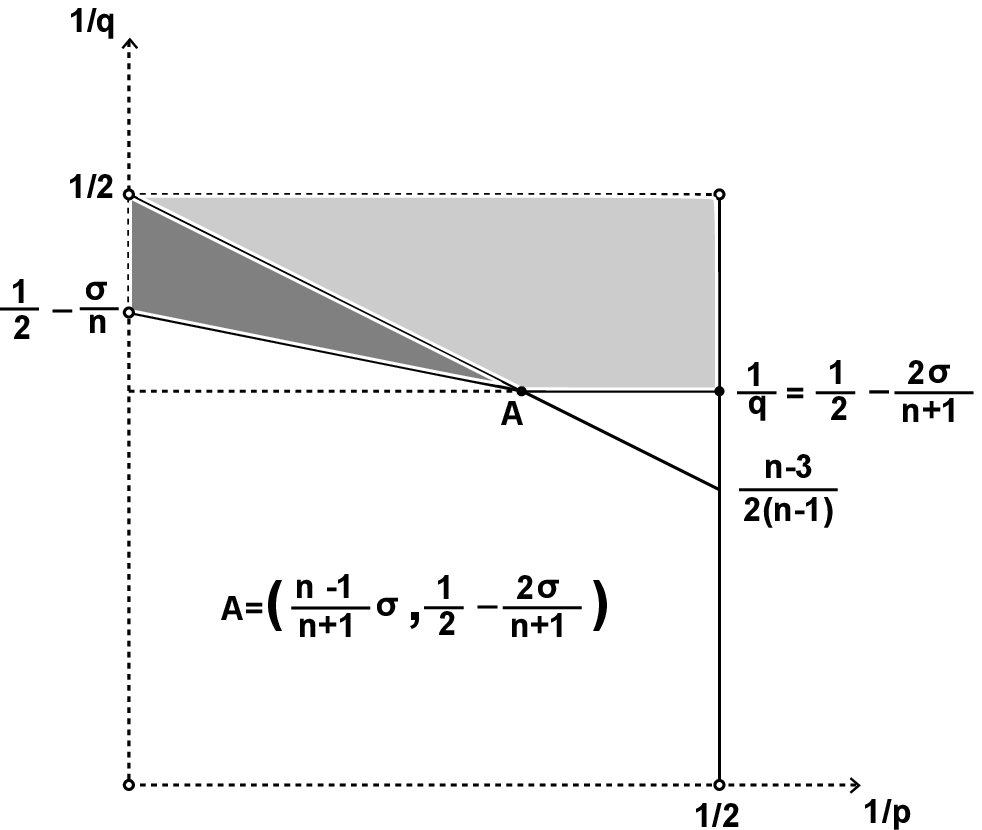}
\caption{Admissible pair $(p_1,q_1)$ in dimension 4 or higher, with $\sigma < \frac{n+1}{2(n-1)}$}\label{drawingd41}
\end{figure}
\begin{figure}[h]
\centering
\includegraphics[scale=.50]{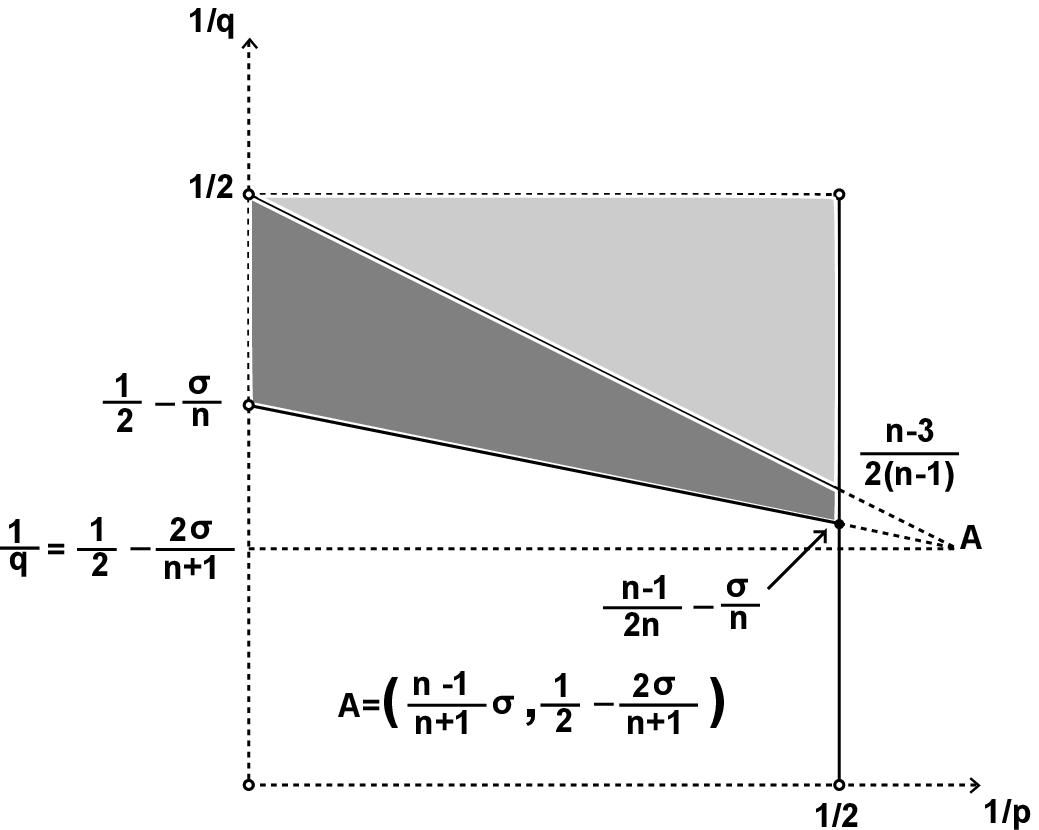}
\caption{Admissible pair $(p_1,q_1)$ in dimension 4 or higher, with $\sigma \geq \frac{n+1}{2(n-1)}$}\label{drawingd42}
\end{figure}

\end{document}